\documentclass[a4paper,12pt,headsepline,cleardoubleempty,reqno,bibtotoc,tablecaptionabove]{amsart}
\usepackage[utf8]{inputenc}
\author{Roland Donninger}
\address{Universität Wien, Fakultät für Mathematik,
  Oskar-Morgenstern-Platz 1, 1090 Vienna, Austria}
\email{roland.donninger@univie.ac.at}
\author{David Wallauch}
\address{Universität Wien, Fakultät für Mathematik,
  Oskar-Morgenstern-Platz 1, 1090 Vienna, Austria}
\email{david.wallauch@univie.ac.at}
\thanks{This work was supported by the Austrian Science Fund FWF,
  Projects P 30076: ``Self-similar blowup in dispersive wave equations''
  and P 34560: ``Stable blowup in supercritical wave equations''.}
\title{Optimal Blowup stability for three-dimensional wave maps}
\usepackage{amsmath,amsfonts,amssymb,amsthm,empheq}
\usepackage{xcolor}
\usepackage{hyperref}
\usepackage{array}
\usepackage{fullpage}
\usepackage{graphicx}
\numberwithin{equation}{section}

\newcommand{\C}{\mathbb{C}}

\newcommand{\R}{\mathbb{R}}

\newtheorem{thm}{Theorem}[section]
\newtheorem{defi}{Definition}[section]
\newtheorem{prop}{Proposition}[section]
\newtheorem{lem}{Lemma}[section]
\renewcommand{\O}{\mathcal{O}}

\DeclareMathOperator{\rg}{rg}
\DeclareMathOperator{\Span}{span}

\newcommand{\hfh}{\textup{\textbf{h}}}

\newcommand{\Cf}{\textup{\textbf{C}}}

\newcommand{\X}{\mathcal{X}}

\newcommand{\Nf}{\textup{\textbf{N}}}
\newcommand{\Lf}{\textup{\textbf{L}}}

\newcommand{\K}{\textup{\textbf{K}}}
\newcommand{\Sf}{\textup{\textbf{S}}}
\newcommand{\I}{\textup{\textbf{I}}}
\newcommand{\Uf}{\textup{\textbf{U}}}
\newcommand{\uf}{\textup{\textbf{u}}}
\newcommand{\vf}{\textup{\textbf{v}}}
\newcommand{\Rf}{\textup{\textbf{R}}}
\newcommand{\gf}{\textup{\textbf{g}}}

\newcommand{\ff}{\textup{\textbf{f}}}
\newcommand{\Pf}{\textup{\textbf{P}}}
\newcommand{\Qf}{\textup{\textbf{Q}}}
\renewcommand{\Re}{\operatorname{Re}}
\renewcommand{\Im}{\operatorname{Im}}
\newcommand{\B}{\mathbb{B}}

\renewcommand{\H}{\mathcal{H}}

\begin{document}
  \begin{abstract}
We study corotational wave maps from $(1+3)$-dimensional Minkowski
space into the three-sphere. We establish the asymptotic stability of
an explicitly known self-similar wave map under perturbations that are
small in the critical Sobolev space. This is accomplished by proving
Strichartz estimates for a radial wave equation with a potential in
similarity coordinates. Compared to earlier work, the main novelty
lies with the fact that the critical Sobolev space is of fractional
order. 
      \end{abstract}
      \maketitle
\section{Introduction}
The present work is concerned with the wave maps equation, the
prototypical example of a geometric wave equation. 
The wave maps equation is a natural generalization of the wave equation when the
unknown takes values in a Riemannian manifold. 
Here, we are only interested in the case where the manifold is the
round sphere, i.e., we consider maps $U: \R^{1,d}\to
\mathbb{S}^d\subset \R^{d+1}$, where $\R^{1,d}$ is the
$(1+d)$-dimensional Minkowski space. In this special case, the wave maps
equation takes the form
\begin{align}\label{eq:wavemaps}
\partial^\mu\partial_\mu U+ (\partial^\mu U\cdot \partial_\mu U)U=0,
\end{align}
where $\cdot$ denotes the Euclidean inner product on $\R^{d+1}$ and
Einstein's summation convention\footnote{As is common in
  relativity, we number the slots of a function on Minkowski space from
  $0$ to $d$ and $\partial^0=-\partial_0$, whereas
  $\partial^j=\partial_j$ for $j\in \{1,2,\dots,d\}$. Two indices,
  where one occurs upstairs and the other one downstairs, are
  automatically summed over and Greek indices take on the values $0,1,\dots,d$.} is in force.
Equation \eqref{eq:wavemaps} is a hyperbolic partial differential
equation and it is natural to study the Cauchy problem. To this end, one prescribes initial data $U(0,.):\R^d\to \mathbb{S}^d$, $\partial_0 U(0,.):\R^d \to \R^{d+1}$ with $ U(0,.)\cdot \partial_0 U(0,.)=0$ and aims to construct a unique solution to Eq.~\eqref{eq:wavemaps} satisfying these initial conditions.
Intriguingly, for $d\geq 2$, it is in general impossible to construct
global-in-time solutions to the Cauchy problem for the wave maps
equation, even if the initial data $(U(0,.),\partial_0 U(0,.))$ are
smooth and nicely behaved towards infinity.
For $d\geq 3$, this is evidenced by an explicit one-parameter family of
self-similar solutions.
Indeed, for $T>0$, let $U_*^T(t,x)=F_*(\frac{x}{T-t})$, where $F_*: \R^d\to \mathbb
S^d\subset \R^{d+1}$ is given by
\[ F_*(\xi):=
  \frac{1}{d-2+|\xi|^2}\begin{pmatrix}
    2\sqrt{d-2}\,\xi \\
    d-2-|\xi|^2
  \end{pmatrix}
  =
  \begin{pmatrix}
    \sin(f_*(\xi))\frac{\xi}{|\xi|} \\
    \cos(f_*(\xi))
  \end{pmatrix} \]
with
  \[ f_*(\xi):=2\arctan\left
            (\frac{|\xi|}{\sqrt{d-2}}\right ).
        \]
        Then $U_*^T$ is a wave map, as one may convince oneself by a
        straightforward computation.
The solution $U_*^T$, which was discovered in
\cite{TurSpe90,Sha88,BizBie15}, starts from smooth initial data but
develops a singularity in finite time in the sense that the gradient
blows up. Moreover, by the finite speed of propagation property
inherent to the wave maps equation, the behavior of the data at
spatial infinity is completely irrelevant.
A natural question that arises immediately is as to whether this explicit
blowup solution has any bearing on the generic behavior of the Cauchy
evolution. Perhaps $U_*^T$ actually belongs to a larger family of
solutions which exhibit similar kinds of singular behavior? And if so,
how large is this family? To answer these questions it is necessary to
study the stability of $U_*^T$ under perturbations of the initial
data. In the present paper, for the case $d=3$, we show that all solutions that start out
close to $U_*^T$ develop a singularity with the same asymptotic
profile as $U_*^T$. Furthermore, the smallness of the perturbation is measured in
the weakest possible ($L^2$-based) Sobolev norm.

In order to state our main theorem precisely, we set $u_*^T(t,x)=\frac{2}{|x|}\arctan(\frac{|x|}{T-t})$ and define $\Omega_{T}\subset \mathbb R\times \mathbb R^3$ for
$T>0$ by
\[ \Omega_{T}:=\left ([0, \infty)\times \mathbb
    R^3\right )\setminus
  \left \{(t,x)\in [T,\infty)\times \mathbb R^3: |x|\leq t-T\right
  \}. \]
In words, $\Omega_T$ is all of the future of the initial surface $t=0$
minus the forward
lightcone emanating from the blowup point $(T,0)$.
Furthermore, for $R>0$ and
$x_0\in\R^d$, we set $\B^d_R(x_0):=\{x\in \R^d: |x-x_0|<R\}$ and abbreviate $\B_R^d:=\B_R^d(0)$.

  \begin{thm}\label{maintheorem}
  There exist constants $\delta_0,M>0$ such that the following
  holds. Let $F: \mathbb R^3\to \mathbb S^3\subset\mathbb R^4$ and $G:
  \mathbb R^3\to \mathbb R^4$ be given by
  \[ F(x)=
    \begin{pmatrix}
      \sin(|x|f(x))\frac{x}{|x|} \\
      \cos(|x|f(x))
    \end{pmatrix},\qquad
    G(x)=
    \begin{pmatrix}
      \cos(|x|f(x))g(x)x \\
      -\sin(|x|f(x))|x|g(x)
    \end{pmatrix}
  \]
  for smooth, radial functions $f,g: \mathbb R^3\to\mathbb R$. Assume
  further that $\delta\in [0,\delta_0]$ and
\[ \||.|[(f,g)-(u_*^1(0,.), \partial_0
    u_*^1(0,.))]\|_{H^{\frac{3}{2}}\times H^{\frac{1}{2}}(\B^3_{1+\delta})}\leq \frac{\delta}{M}. \]
  Then there exists a $T\in [1-\delta,1+\delta]$ and a unique smooth wave map
  $U: \Omega_{T}\to \mathbb S^3\subset\mathbb R^4$
 that satisfies $U(0,x)=F(x)$ and $\partial_0 U(0,x)=G(x)$ for all $x\in
 \mathbb R^3$. Furthermore, in the backward lightcone of the point $(T,0)$, we
 have the weighted Strichartz estimates
 \[ \int_0^T \left \||.|^{-\frac45}\left (U(t,.)-U_*^T(t,.)\right )\right \|_{L^{10}(\mathbb
     B^3_{T-t})}^2dt\leq  \delta^2 \]
 and
 \[ \int_0^T \left \||.|^{-\frac{4}{15}}\left (\partial_j
       U(t,.)-\partial_j U_*^T(t,.)\right )\right \|_{L^{\frac{30}{11}}(\mathbb
     B^3_{T-t})}^6dt\leq \delta^6 \]
 for $j\in \{1,2,3\}$.
\end{thm}
\subsection{Discussion} We would like to make a couple of remarks.

\subsubsection{Stability of blowup} Note that
  \[ U_*^T(t,0)=
    \begin{pmatrix}
      0 \\ 1
    \end{pmatrix}
  \]
  for all $t\in [0,T)$. Hence, a scaling argument shows that 
  \[ \left \||.|^{-\frac45}\left (U_*^T(t,.)-
      U_*^T(t,0)\right )
      \right \|_{L^{10}(\B^3_{T-t})}\simeq
      (T-t)^{-\frac12}, \]
    from which we infer that
    \[ \int_0^T \left \||.|^{-\frac45}\left(U_*^T(t,.)-U_*^T(t,0)\right)\right \|_{L^{10}(\mathbb
        B^3_{T-t})}^2 dt\simeq \int_0^T (T-t)^{-1}dt=\infty. \]
   Similarly,
    \[ \int_0^T \left \||.|^{-\frac{4}{15}}\partial_j
          U_*^T(t,.)\right \|_{L^{\frac{30}{11}}(\mathbb
          B^3_{T-t})}^6 dt\simeq \int_0^T (T-t)^{-1}dt=\infty. \]
      Consequently, these Strichartz norms detect self-similar blowup and
Theorem \ref{maintheorem} shows that $U^T_*$ is
asymptotically stable in the backward
lightcone of the singularity. Put differently, our solution $U$ can be
trivially written as
  \[
    U(t,x)=U_*^T(t,x)+\underbrace{U(t,x)-U_*^T(t,x)}_{\mbox{small}} \]
and this shows shows that $U$ exhibits the same blowup as $U^T_*$ modulo an error which is small in suitable Strichartz spaces.

\subsubsection{Optimality} Eq.~\eqref{eq:wavemaps} is  invariant under the scaling  $U(t,x)\mapsto
  U(\frac{t}{\lambda}, \frac{x}{\lambda})$ for $\lambda>0$ and the
  corresponding scaling-invariant Sobolev space is $\dot
  H^{\frac{d}{2}}\times \dot H^{\frac{d}{2}-1}(\R^d)$.
Moreover, from the ill-posedness of the wave maps equation 
  below scaling \cite{ShaTah94} it follows that the smallness condition imposed on the initial data is measured in the optimal topology in terms of regularity.

\subsubsection{Symmetry} The prescribed initial data belong to the class of corotational maps, a symmetry preserved by the wave maps flow. Further, our Strichartz estimates
  are not translation-invariant and so inherently
  corotational.
  
  \subsubsection{Maximal domain of existence}
 The domain on which we construct solutions is all of
 $[0,\infty)\times \R^3$ except for the part of spacetime that is causally
 influenced by the singularity. Whether one can extend the solution
 even further in a meaningful way is an intriguing open question.
  \subsubsection{Supercriticality}
Lastly, we want to put emphasis on the fact that Theorem \ref{maintheorem} is
 a large-data result for an energy-supercritical geometric wave equation.
   \subsection{Related results}
Due to the sheer volume of intriguing works on the wave maps equation, we can only mention a handful of results which are directly linked to the present paper. For the local theory of corotational wave maps at low regularity we refer to \cite{ShaTah94}. The general case is the focus of the works \cite{KlaMac95, KlaSel97, Tao00, MasPla12}. 
Establishing results concerning the small data global Cauchy problem is of course most delicate when one measures smallness
in a scaling-invariant space. This challenging problem was intensely studied in the 1990s and the beginning of the 2000s and was resolved in
\cite{Tat98, Tat01, Tao01a, Tao01b, KlaRod01, ShaStr02, Kri03,
  NahSteUhl03, Kri04, Tat05, CanHer18}.

Turning to the large data
problem, we start with the case $d=2$ where the strongest results are
available, given that this is the energy-critical case where energy
conservation yields invaluable global information. However, despite
the conservation of energy, finite-time blowup is possible, albeit via
a different, more complicated mechanism than in our case. Singularity formation takes place via a dynamical
rescaling of a soliton (a harmonic map). Consequently, already the construction of finite time blowup is highly nontrivial and was first
achieved in \cite{KriSchTat08, RodSte10, RapRod12}, inspired by
numerical evidence \cite{BizChmTab01}, see also \cite{CanKri15}. Stability results for blowup
are proven in \cite{RapRod12, KriMia20}.
Subsequently, the question of large data global existence has to be addressed in
view of the fact that finite-time blowup is possible. Since the blowup
takes place via the shrinking of a harmonic map, the ``first'' harmonic
map provides a natural threshold for global existence.
 This is expressed in the \emph{threshold
conjecture} 
\cite{SteTat10a, SteTat10b, KriSch12, LawOh16, ChiKriLuh18}, see also
the series of unpublished preprints \cite{Tao08} and the earlier \cite{Str03, CotKenMer08} for the corotational
setting. Recent works on energy-critical wave maps focus on
the precise asymptotic behavior and the \emph{soliton resolution
  conjecture}
\cite{CotKenLawSch15a, CotKenLawSch15b, Cot15, Gri17, JiaKen17,
  JenLaw18, DuyJiaKenMer18}.

The present paper is concerned with the energy-supercritical case
$d\geq 3$,
where the conservation of energy is of no use for
the study of the Cauchy problem. Therefore, the understanding
of large-data evolutions is still comparatively poor. 
The existence of self-similar blowup for $d\geq 3$ is established in
\cite{Sha88, TurSpe90, CazShaTah98, Biz00, BizBie15}. Motivated by
numerical evidence \cite{BizChmTab00}, the stability of self-similar
blowup under perturbations that are small in Sobolev spaces of
sufficiently high order is proved in \cite{DonSchAic12, Don11,
  CosDonXia16, CosDonGlo17, ChaDonGlo17, DonGlo19, BieDonSch21}.
We also remark that starting from dimension 7, another  blowup mechanism occurs which is 
more reminiscent of the energy-critical case \cite{GhoIbrNgu18}, see
\cite{DodLaw15, ChiKri17} for other large-data results.
Blowup stability  in critical Sobolev spaces has so far been established for the four-dimensional wave maps equation \cite{DonWal22} and the simpler energy-criticial wave equation in dimensions $3\leq d\leq 6$ \cite{Don17, DonRao20, Wal22}. We also want to mention \cite{Bri20} for an
extension to randomized perturbations.

 \subsection{Outline of the proof}
To prove Theorem \ref{maintheorem} we follow the strategy laid out in
\cite{DonWal22} and which itself built on the previous works
\cite{Don17} and \cite{DonRao20}. However, in contrast to the
four-dimensional
case studied in \cite{DonWal22}, the optimal Sobolev
spaces here are of fractional order. This causes major additional
problems throughout our analysis of equation \eqref{eq:wavemaps}
which were not present in previous works.

The first step in analysing Eq.~\eqref{eq:wavemaps} is the symmetry reduction. From the special corotational form of the prescribed data and the preservation of that symmetry class by the wave maps flow it follows that the associated solution $U$ is of the form \begin{equation}
\label{eq:cor}
    U(t,x)=
    \begin{pmatrix}
      \sin(|x|u(t,x))\frac{x}{|x|} \\
      \cos(|x|u(t,x))
    \end{pmatrix}
  \end{equation}
  for a smooth function $u: [0,T)\times \R^3\to \R$ such that $u(t,.)$ is
  radial for each $t\in [0,T)$. Further, the corotational ansatz simplifies Eq.~\eqref{eq:wavemaps} to the semilinear equation
  \begin{align}\label{startingeq2}
\left(\partial_t^2-\partial_r^2-\frac{4}{r}\partial_r\right)\widetilde
  u(t,r) +\frac{ \sin(2r\widetilde u(t,r))- 2 r \widetilde u (t,r)}{ r^3}=0
\end{align}
for $r>0$, where $u(t,x)=\widetilde u(t,|x|)$. Note that
\ref{startingeq2} is a 5-dimensional equation rather than
a 3-dimensional one, as one would perhaps expect. Therefore, it is
natural to view $u$ as a radial function on $[0,T)\times \R^5$ instead of $[0,T)\times \R^3$. Moreover, a Taylor expansion shows that the apparent singularity in \eqref{startingeq2} is in fact removable and the nonlinearity is perfectly smooth. Theorem \ref{maintheorem} is then essentially a consequence of the following result.
\begin{thm}\label{stability}
There exist $\delta_0, M>0$ such that the following holds.
Let $f,g\in C^\infty(\overline{\B^5_{1+\delta}})$ be radial and let $\delta\in
[0,\delta_0]$ be such that 
\[
\|(f,g)-(u^1_*(0,.),\partial_0 u_*^1(0,.))\|_{H^\frac32\times H^\frac12(\B^5_{1+\delta})} \leq\frac{\delta}{M}.
\]
Then there exists a blowup time $T\in [1-\delta,1+\delta]$ and a
unique smooth solution
\[ u: \left \{(t,x)\in [0,T)\times \R^5: |x|\leq T-t\right \} \to\R \] of
Eq.~\eqref{startingeq2} satisfying $u(0,.)=f$ and $\partial_0
u(0,.)=g$ on $\overline{\B^5_T}$. Furthermore, we have the Strichartz estimates
\begin{equation}
\int_0^T \left\|u(t,.)-u^T_*(t,.)\right\|^2_{L^{10}(\B^5_{T-t})}dt\leq \delta^2,
\end{equation}
and 
\begin{equation}
\int_0^T \left\|u(t,.)-u^T_*(t,.)\right \|_{\dot{W}^{1,\frac{30}{11}}(\mathbb
     B^5_{T-t})}^6dt \leq \delta^6.
\end{equation}
\end{thm}
We now give a nontechnical outline of the proof of Theorem \ref{stability}.
\begin{itemize}
\item First, we perform preliminary coordinate transformations and choose  the right functional setup. Given the self-similar nature of the blowup, we recast Eq.~\eqref{startingeq2} in the similarity coordinates 
\begin{equation*}
	\tau=-\log(T-t)+\log(T),\,\,\; \rho=\frac{r}{T-t}.
	\end{equation*}
Then, we proceed to show that the operator corresponding to the free wave equation in these coordinates is densely defined and closable in different topologies and that each of these closures generates a semigroup $\Sf_0$. More precisely, we show that
\begin{align*}
\|\Sf_0(\tau)\|_{H^2\times H^1(\B^5_1)}\lesssim e^{-\frac{\tau}{2}} \qquad \|\Sf_0(\tau)\|_{H^1\times L^2(\B^5_1)}\lesssim e^{\frac{\tau}{2}}
\end{align*}
which we interpolate to obtain 
\begin{align*}
\|\Sf_0(\tau)\|_{H^{\frac32}\times H^{\frac{1}{2}}(\B^5_1)}\lesssim 1.
\end{align*}
We then linearise the nonlinearity around $u_*^T$ and study the resulting linear operator $\Lf$. Utilizing  \cite{CosDonGlo17} enables us to infer that $\Lf$, viewed as a densely defined operator on $H^{\frac32}\times H^{\frac{1}{2}}(\B^5_1)$,  has precisely one
eigenvalue $\lambda = 1$ in the (closed) complex right half-plane with
a corresponding rank 1 spectral projection $\Pf$. 
 \item To control the evolution, we next derive Strichartz estimates for $\Sf(\tau)(\I-\Pf)$, where $\Sf$ is the semigroup generated by $\Lf$. We accomplish this by asymptotically constructing the resolvent of $\Lf$  and representing the semigroup as the Laplace inversion of $(\lambda-\Lf)^{-1}=:\Rf_{\Lf}(\lambda)$. For the resolvent construction it is crucial that the spectral equation $(\lambda-\Lf)\uf=\ff $ with 
 $\uf=(u_1,u_2)$ and $\ff=(f_1,f_2)$ reduces to the second order ODE
\begin{equation}\label{eq:outline}
  \begin{split}
    (\rho^2-1)&u_1''(\rho)+\left(2(\lambda+2)\rho-\frac{4}{\rho}\right)u_1'(\rho)+(\lambda+2)(\lambda+1)u_1(\rho)
    \\
    &-\frac{16}{(1+\rho^2)^2}u_1(\rho)=F_\lambda(\rho)
    \end{split}
\end{equation}
with  $F_\lambda(\rho)=f_2(\rho)+(\lambda+2)f_1(\rho)+\rho f_1'(\rho)$ and $\rho \in (0,1)$.
The construction of $\Rf_{\Lf}(\lambda)$ is carried out by an intricate asymptotic ODE analysis of Eq.~\eqref{eq:outline} based on a  Liouville Green transform, Bessel asymptotics, and Volterra iterations.
\item
Having done this, we turn to the somewhat lengthy task of obtaining
Strichartz estimates by estimating the oscillatory integrals occurring
in the Laplace inversion of $\Rf_\Lf$. A first idea would be to obtain estimates of the form
\begin{align}\label{eq:notposs}
\|[\Sf(\tau)(\I-\Pf)\ff]_1\|_{L^{p_1}_\tau(\R_+)L^{q_1}(\B^5_1)}\lesssim \|\ff\|_{H^1\times L^2(\B^5_1)}
\end{align}
and
\begin{align*}
\|[\Sf(\tau)(\I-\Pf)\ff]_1\|_{L^{p_2}_\tau(\R_+)L^{q_2}(\B^5_1)}\lesssim \|\ff\|_{H^2\times H^1(\B^5_1)}
\end{align*}
and interpolate between them, where $[\Sf(\tau)(\I-\Pf)\ff]_1$ denotes the
first component of $\Sf(\tau)(\I-\Pf)\ff$. There is, however, a
problem with this naive ansatz. In the $H^1\times L^2$ universe the
spectrum of $\Lf$, $\sigma(\Lf)$, satisfies $$\left \{z\in \C: \Re z<
  \tfrac{1}{2}\right \} \cup \{1\}\subset\sigma(\Lf).$$
Consequently, the best estimate one can hope for is of the form 
\begin{align}\label{eq:H1 strichfalse}
\|[e^{-\frac{\tau}{2}}\Sf(\tau)(\I-\Pf)\ff]_1\|_{L^{p_1}_\tau(\R_+)L^{q_1}(\B^5_1)}\lesssim \|\ff\|_{H^1\times L^2(\B^5_1)}.
\end{align}
Thus, for the interpolation argument to work, the corresponding $H^2\times H^1$ estimate needs to compensate the added decay in $\tau$. In other words, we have to derive estimates of the type
\begin{align}
\|[e^{\frac{\tau}{2}}\Sf(\tau)(\I-\Pf)\ff]_1\|_{L^{p_2}_\tau(\R_+)L^{q_2}(\B^5_1)}\lesssim \|\ff\|_{H^2\times H^1(\B^5_1)}.
\end{align}
However, we cannot
rigorously exclude the existence of finitely many eigenvalues with real parts
bigger than $-\tfrac12$. But what we do know is that all of these
possible eigenvalues have finite algebraic multiplicities. Hence, the semigroup $\Sf(\tau)$ generated by $\Lf$ satisfies
\begin{align*}
\|\Sf(\tau)(\I-\Qf)(\I-\Pf)\|_{H^{2}\times H^{1}(\B^5_1)}\lesssim_\eta e^{\eta\tau},
\end{align*}
for any $\eta >-\tfrac12$,  where $\Qf$ is the spectral projection associated to all eigenvalues $\lambda_i$ with $-\tfrac12<\Re \lambda_i<0$. 
Furthermore, there might also be eigenvalues sitting on the boundary of the essential spectrum in the $H^2\times H^1 $ universe (i.e. the line $\Re z=-\tfrac12$). Thus, we can only derive an estimate of the form 
\begin{align}
\|[e^{(\frac{1}{2}-\delta)\tau}\Sf(\tau)(\I-\Qf)(\I-\Pf)\ff]_1\|_{L^{p_2}_\tau(\R_+)L^{q_2}(\B^5_1)}\lesssim \|\ff\|_{W^{2,\frac{2}{1+\delta}}\times W^{1,\frac{2}{1+\delta}}(\B^5_1)}
\end{align}
with $\delta $ very close to $0$. Hence, instead of proving \eqref{eq:H1 strichfalse}, we show that 
\begin{align}\label{eq:H1 strich}
\|[e^{-(\frac{1}{2}-\delta)\tau }\Sf(\tau)(\I-\Pf)\ff]_1\|_{L^{p_1}_\tau(\R_+)L^{q_1}(\B^5_1)}\lesssim \|\ff\|_{W^{1,\frac{2}{1-\delta}}\times L^{\frac{2}{1-\delta}}(\B^5_1)}
\end{align}
so that interpolation puts us in the correct spaces.
As a consequence, we still have to control the evolution on the image of $\Qf$. For this, we will make use of the following lemma.
\begin{lem}
Let $H$ be a Hilbert space. Then, for any densely defined
operator $T:D(T)\subset H \to H$ with finite rank, there exists a
dense subset $X \subset H$ with $X\subset\mathcal D(T)$ and a bounded linear operator $\widehat{T}:H \to H$ such that
\begin{equation*}
T|_X=\widehat{T}|_X.
\end{equation*}
\end{lem}
By applying this result to $\Qf$, viewed as a densely defined unbounded operator on $H^{\frac{3}{2}}\times H^{\frac12}(\B^5_1)$, we manage to arrive at the desired estimates
\begin{align}
\|[\Sf(\tau)(\I-\Pf)\ff]_1\|_{L^{p}_\tau(\R_+)L^{q}(\B^5_1)}\lesssim \|\ff\|_{H^{\frac{3}{2}}\times H^{\frac12}(\B^5_1)}.
\end{align}
Analogously, we derive other spacetime estimates involving (fractional) derivatives on the left-hand side.
\item Finally, the full nonlinear problem is treated by fixed point arguments in an appropriate Strichartz space.
\end{itemize}
\section{Transformations and Semigroup theory}
In all what follows we identify radial functions with their radial representatives. Moreover, any vector space, for instance $H^k(\B^5_1)$ or $C^k(\overline{\B^5_1})$, always denotes the corresponding radial subspace within that space.
Before we can properly analyze Eq.~(\ref{startingeq2}) in the lightcone $\Gamma^T:=\{(t,r)\in [0,\infty)^2:r\leq T-t\}$, we first need the right choice of coordinates. For our purposes, suitable coordinates are given by the similarity coordinates 
	\begin{equation}\label{coordinate}
	\tau=-\log(T-t)+\log(T),\,\,\; \rho=\frac{r}{T-t}.
	\end{equation}
Thus, we set $\psi(\tau, \rho)=Te^{-\tau}u(T-Te^{-\tau},Te^{-\tau}\rho)$ and switch to the similarity coordinates which turns Eq.~\eqref{startingeq2} into 
	\begin{align}\left(2+3\partial_\tau+\partial_\tau^2 +2\rho\partial_\tau\partial_\rho+ 4 \rho\partial_\rho-\frac{4}{\rho}\partial_\rho+(\rho^2-1)\partial_\rho^2
\right)\psi +\frac{\sin(2\rho\psi)-2\rho\psi}{\rho^3}=0,
	\end{align}
        where we omit the arguments of $\psi$ for brevity.
	Next, we define
	\begin{equation*}
	\psi_1(\tau,\rho):=\psi(\tau,\rho) 
	\end{equation*}
and 
\begin{equation*}\label{transform}
\psi_2(\tau,\rho):=(1+\partial_\tau+\rho\partial_\rho)\psi_1(\tau,\rho),
\end{equation*}
	which yields the system
	\begin{equation}
          \label{eq:syspsi}
          \begin{split}
	\partial_\tau \psi_1 &=\psi_2-\psi_1-\rho\partial_\rho\psi_1\\
\partial_\tau \psi_2
                             &=\partial_\rho^2\psi_1+\frac{4}{\rho}\partial_\rho \psi_1-\rho\partial_\rho\psi_2
                             -2\psi_2-\frac{3\sin(2\rho\psi_1)-6\rho\psi_1}{2\rho^3},
                             \end{split}
	\end{equation}
	with initial data
	\begin{equation*}\label{inidata}
	\psi_1(0,\rho)=Tf(T\rho),\,\, \psi_2(0,\rho)=T^2 g(T\rho).
	\end{equation*}

We also remark that in these coordinates the blowup function $u_*^T$ is of the form 
\begin{align*}
\Psi_*(\rho)=\begin{pmatrix}
&\frac{2}{\rho}\arctan(\rho)\\
& \frac{2}{1+\rho^2}
\end{pmatrix}.
\end{align*}

 \subsection{Semigroup theory}

Motivated by the above evolution equation, we define the differential operator $\widetilde{\Lf}_0$ as
\begin{align*}
\widetilde{\Lf}_0\uf(\rho):=\begin{pmatrix}
-\rho u_1'(\rho)-u_1(\rho)+u_2(\rho)\\
u_1''(\rho)+\frac{4}{\rho}u_1'(\rho)-\rho u_2'(\rho)-2u_2(\rho)
\end{pmatrix},
\end{align*}
where $\mathbf u=(u_1, u_2)$ with domain
\[
D(\widetilde{\Lf}_0):=\{\uf \in C^3\times C^2\big(\overline{\B_1^5} \big):\uf \text{ radial}\}.
\]
We also define two inner products $(.,.)_{\mathcal{E}_1}$ and $(.,.)_{\mathcal{E}_2}$ on $D(\widetilde{\Lf}_0)$ as 
\begin{align*}
(\uf,\vf)_{\mathcal{E}_1}:=&\int_0^1u_1'(\rho)\overline{v_1'(\rho)}\rho^4
                                       d \rho
+\int_0^1 u_2(\rho)\overline{v_2(\rho)}\rho^4 d\rho \\
&+u_1(1)\overline{v_1(1)}
\end{align*}
and
\begin{align*}
(\uf,\vf)_{\mathcal{E}_2}:&=8\int_0^1u_1''(\rho)\overline{v_1''(\rho)}\rho^4
                                       d \rho +32\int_0^1
                                       u_1'(\rho)\overline{v_1'(\rho)}\rho^2
                                       d\rho 
+2\int_0^1 u_2'(\rho)\overline{v_2'(\rho)}\rho^4 d\rho \\
&\quad+u_1(1)\overline{v_1(1)}+u_2(1)\overline{v_2(1)}.
\end{align*}
Further, we denote the associated norms by $\|.\|_{\mathcal{E}_j}.$
Then the following estimate holds.

\begin{lem}\label{lem:lumph1}
The operator $\widetilde{\Lf}_0$ satisfies
\begin{align*}
\Re(\widetilde{\Lf}_0\uf,\uf)_{\mathcal{E}_1}\leq \frac{1}{2}\|\uf\|_{\mathcal{E}_1}^2.
\end{align*}
for all $\uf \in D(\widetilde{\Lf}_0)$.
\end{lem}
\begin{proof}
Integrating by parts shows
\begin{align*}
-\int_0^1 u_1''(\rho)\overline{u_1'(\rho)} \rho^5 d\rho= -|u_1'(1)|^2 +5\int_0^1|u_1'(\rho)|^2 \rho^4 d\rho+\int_0^1u_1'(\rho)\overline{u_1''(\rho)} \rho^5 d\rho
\end{align*}
and so,
\begin{align*}
-\Re\int_0^1u_1''(\rho)\overline{u_1'(\rho)} \rho^5 d\rho= -\frac{|u_1''(1)|^2}{2} +\frac{5}{2}\int_0^1|u_1'(\rho)|^2 \rho^4 d\rho.
\end{align*}
Consequently,
\begin{align*}
\int_0^1[\widetilde{\Lf}_0\uf]_1'(\rho)\overline{u_1'(\rho)}\rho^4 d \rho=\frac{1}{2}\int_0^1|u'_1(\rho)|^2\rho^4 d \rho+\int_0^1 u_2'(\rho)\overline{u_1'(\rho)}\rho^4 d\rho -\frac{1}{2}|u_1'(1)|^2.
\end{align*}
Similarly,
\begin{align*}
\int_0^1[\widetilde{\Lf}_0\uf]_2(\rho)\overline{u_2(\rho)}\rho^4 d \rho & =\frac{1}{2}\int_0^1|u_2(\rho)|^2\rho^4 d \rho+4\int_0^1 u_1'(\rho)\overline{u_2(\rho)}\rho^3 d\rho
\\
&\quad+ \int_0^1 u_1''(\rho) \overline{u_2(\rho)}\rho^4 d\rho-\frac{1}{2}|u_1'(1)|^2.
\end{align*}
Further, given that
\begin{align*}
\int_0^1 u_1''(\rho)\overline{u_2(\rho)}\rho^4 d\rho= u_1'(1)\overline{u_2(1)}-\int_0^1 u_1'(\rho) \overline{u_2'(\rho)}\rho^4 d\rho -4\int_0^1u_1'(\rho) \overline{u_2(\rho)}\rho^3 d\rho,
\end{align*}
we obtain
\begin{align*}
\Re (\widetilde{\Lf}_0 \uf,\uf)_{\mathcal{E}_1}&=\frac{1}{2}\int_0^1\left(|u_1'(\rho)|^2\rho^4+|u_2(\rho)|^2\rho^4 \right)d\rho -\frac{1}{2}\left(|u_1'(1)|^2+|u_2(1)|^2\right)
\\
&\quad+\Re\left(u_2(1)\overline{u_1'(1)}-u_1'(1)\overline{u_1(1)}-|u_1(1)|^2+u_2(1)\overline{u_1(1)}\right).
\end{align*}
By employing the elementary inequality
\begin{equation*}
\Re(a\overline{b}+a\overline{c}-b\overline{c})\leq \frac{1}{2}(|a|^2+|b|^2+|c|^2),
\end{equation*}
with $a=u_2(1)$, $b=u_2(1)$, $c=u_1(1)$ 
we deduce that
\begin{align*}
\Re (\widetilde{\Lf}_0 \uf,\uf)_{\mathcal{E}_1}&\leq \frac{1}{2}\int_0^1\left(|u_1'(\rho)|^2\rho^4+|u_2(\rho)|^2\rho^4\right) d\rho \leq \frac{1}{2}\|\uf\|_{\mathcal{E}_1}^2.
\end{align*}
\end{proof}
For the inner product $(.,.)_{\mathcal{E}_2}$ we can derive a similar but better estimate.
\begin{lem}\label{lem:lumph2}
The operator $\widetilde{\Lf}_0$ satisfies
\begin{align*}
\Re(\widetilde{\Lf}_0\uf,\uf)_{\mathcal{E}_2}\leq -\frac{1}{2}\|\uf\|_{\mathcal{E}_2}^2
\end{align*}
for all $\uf \in D(\widetilde{\Lf}_0)$.
\end{lem}
\begin{proof}
Let $\uf \in D(\widetilde{\Lf}_0)$. 
Integrating by parts as above shows that
\begin{align*}
\Re\int_0^1[\widetilde{\Lf}_0\uf]_1''(\rho)\overline{u_1''(\rho)}\rho^4 d\rho
=&\Re\left(-\int_0^1u_1^{(3)}(\rho)\overline{u_1''(\rho)} \rho^5 d\rho+\int_0^1 u_2''(\rho)\overline{u_1''(\rho)}\rho^4 d \rho\right)
\\&-3 \int_0^1 |u_1''(\rho)|^2 \rho^4 d\rho
\\
=&-\frac{1}{2}\int_0^1 |u_1''(\rho)|^2 \rho^4 d\rho-\frac{|u_1''(1)|^2}{2}+\Re\int_0^1 u_2''(\rho)\overline{u_1''(\rho)}\rho^4 d \rho.
\end{align*}
Similarly, we see that 
\begin{align*}
\Re\int_0^1 [\widetilde{\Lf}_0\uf]_2'(\rho)\overline{u'_2(\rho)}\rho^4 d \rho
=&\Re\int_0^1 u_1^{(3)}(\rho)\overline{u_2'(\rho)}\rho^4 d
   \rho
   \\
   &\quad+\Re \left(4\int_0^1\left[ u_1''(\rho)\overline{u_2'(\rho)} \rho^3 -u_1'(\rho)\overline{u_2'(\rho)} \rho^2\right] d \rho\right)
\\
&-\frac{1}{2}\int_0^1 |u_2'(\rho)|^2 \rho^4 d\rho-\frac{|u_2'(1)|^2}{2}
\\
=&\Re\left(-\int_0^1 u_1''(\rho)\overline{u_2''(\rho)}\rho^5 d \rho-4\int_0^1u_1'(\rho)\overline{u_2'(\rho)} \rho^3 d \rho\right)
\\
&\quad + \Re (u_1''(1)\overline{u_2'(1)})-\frac{|u_2'(1)|^2}{2}-\frac{1}{2}\int_0^1 |u_2'(\rho)|^2 \rho^4 d\rho.
\end{align*}
It follows that
\begin{align*}
\Re\int_0^1\left([\widetilde{\Lf}_0\uf]_1''(\rho)\overline{u_1''(\rho)}+ [\widetilde{\Lf}_0\uf]_2'(\rho)\overline{u'_2(\rho)}\right)\rho^4 d \rho=& -\frac{1}{2}(|u_1''(1)|^2+|u_2'(1)|^2)
\\
&+\Re(u_1''(1)\overline{u_2'(1)})-\frac{1}{2}\int_0^1 |u_1''(\rho)|^2 \rho^4 d\rho
\\
&-\frac{1}{2}\int_0^1 |u_2'(\rho)|^2 \rho^4 d\rho
\\
&-4\Re \int_0^1 u_1'(\rho)\overline{u_2'(\rho)}\rho^{3} d \rho=:I_1.
\end{align*}
A short calculation then shows
\begin{align*}
8I_1+32\Re  \int_0^1 [\widetilde{\Lf}_0\uf]_1'(\rho)\overline{ u_1'(\rho)}\rho^2d\rho \leq &
-4\int_0^1 |u_1''(\rho)|^2 \rho^4d\rho-4 \int_0^1 |u_2'(\rho)|^2 \rho^4 d\rho
\\
&-16\int_0^1 |u_1'(\rho)|^2 \rho^2
-\frac{1}{2}\left(|u_1''(1)|^2+|u_2'(1)|^2 \right)
\\
&-16|u_1'(1)|^2
+\Re(u_1''(1)\overline{u_2'(1)}).
\end{align*}
Consequently, adding up all the boundary terms yields
\begin{align*}
&\quad-\frac{1}{2}\left(|u_1''(1)|^2+|u_2'(1)|^2\right) -16|u_1'(1)|^2
+\Re(u_1''(1)\overline{u_2'(1)})
\\
&\quad+[\widetilde{\Lf}_0\uf]_1(1)\overline{u_1(1)}+[\widetilde{\Lf}_0\uf]_2(1)\overline{u_2(1)}
\\
&=
-\frac{1}{2}\left(|u_1''(1)|^2+|u_2'(1)|^2\right) -16|u_1'(1)|^2
+\Re(u_1''(1)\overline{u_2'(1)})
\\
&\quad+\Re\left(u_1(1)\overline{u_2(1)}-u_1'(1)\overline{u_1(1)}\right)-|u_1(1)|^2\\
&\quad+\Re\left( u_1''(1)\overline{u_2(1)}+4u_1'(1)\overline{u_2(1)}-u_2'(1)\overline{u_2(1)}\right)-2|u_2(1)|^2.
\end{align*}
By again employing the inequality
\begin{equation*}
\Re(a\overline{b}+a\overline{c}-b\overline{c})\leq \frac{1}{2}(|a|^2+|b|^2+|c|^2),
\end{equation*}
once with $a=u_2(1)$, $b=u_1'(1)$, $c=u_1(1)$
and once with $a=u_1''(1)$, $b=u_2'(1)$, $c=u_2(1)$,
we obtain
\begin{align*}
\Re\left(\widetilde{\Lf}_0\uf,\uf\right)_{\mathcal{E}_2}&\leq
-4\int_0^1 |u_1''(\rho)|^2 \rho^4d\rho-4\int_0^1 |u_2'(\rho)|^2 \rho^4 d\rho-16\int_0^1 |u_1'(\rho)|^2 \rho^2 d \rho\\
&\quad+
\Re\left(3 u_1'(1)\overline{u_2(1)}\right)-15|u_1'(1)|^2-\frac{1}{2}|u_1(1)|^2-|u_2(1)|^2
\\
&\leq -\frac{1}{2}\|\uf\|_{\mathcal{E}_2}^2.
\end{align*}
\end{proof}

To be able to invoke the Lumer Phillips Theorem we carry on by showing the density of the range of $(1-\widetilde{\Lf}_0)$.

\begin{lem}\label{lem:density range}
Let $\ff \in C^\infty\times C^\infty(\overline{\B^5_1})$.
Then there exists a $\uf$ in $D(\widetilde{\Lf}_0)$ such that
\begin{align*}
(1-\widetilde{\Lf}_0)\uf=\ff
\end{align*}
\end{lem}
\begin{proof}
The equation $(\lambda-\widetilde{\Lf}_0)\uf=\ff$ written out explicitly reads
\begin{align*}
(1+\lambda)u_1(\rho)+\rho u_1'(\rho) -u_2(\rho)=f_1(\rho)\\
(2+\lambda)u_2(\rho)+\rho u_2'(\rho)-u_1''(\rho)-\frac{4}{\rho}u_1'(\rho)=f_2(\rho)
\end{align*}
and the first of the above equations implies that
\begin{equation}\label{eq:u2}
u_2(\rho)=(1+\lambda)u_1(\rho)+\rho u_1'(\rho)-f_1(\rho).
\end{equation}
Setting $\lambda=1$ and plugging this into the second one yields
\begin{equation}\label{eq: lambda=1}
(\rho^2-1)u_1''(\rho)+
\left(6\rho-\frac{4}{\rho}\right)u_1'(\rho)+ 6u_1(\rho)=F_1(\rho)
\end{equation}
with $F_1(\rho)=f_2(\rho)+3 f_1(\rho)+\rho f_1'(\rho)$.
A fundamental system for the homogeneous equation 
\begin{equation*}
(\rho^2-1)u_1''(\rho)+
\left(6\rho-\frac{4}{\rho}\right)u_1'(\rho)+ 6u_1(\rho)=0
\end{equation*}is given by
$$\psi_0(\rho):=\frac{\tanh^{-1}(\rho)-\rho}{\rho^3},\qquad \psi_1(\rho):=\rho^{-3},
$$
and the Wronskian of these two is given
by
\begin{align*}
W(\psi_0,\psi_1)(\rho)=-\frac{1}{\rho^4(1-\rho^2)}.
\end{align*}
By the variation of constants formula, a solution $u_1$ of Eq. ~\eqref{eq: lambda=1} is then given by
\begin{align*}
u_1(\rho)&=\psi_0(\rho)\int_\rho^1\frac{\psi_1(s)F_1(s)}{W(\psi_0,\psi_1)(s)\left(s^2-1\right)} ds+\psi_1(\rho)\int_0^\rho\frac{\psi_0(s)F_1(s)}{W(\psi_0,\psi_1)(s)\left(s^2-1\right)} ds\\
&=\frac{\tanh^{-1}(\rho)-\rho}{\rho^3}\int_\rho^1 sF_1(s)ds
+ \rho^{-3}\int_0^\rho (s\tanh^{-1}(s) -s^2)F_1(s) ds.
\end{align*}
From standard ODE theory it follows that $u_1\in C^\infty ((0,1))$. Moreover, a Taylor expansion shows
that
$\psi_0$ is a smooth even function on $ [0,1)$ and so 
\[
\rho\mapsto \psi_0(\rho)\int_\rho^1 sF_1(s)ds \in C^\infty([0,1))
\]
 Next, we rescale according to $\rho t= s$ to obtain that
 \begin{align*}
 i_1(\rho):=\rho^{-3}\int_0^\rho (s\tanh^{-1}(s) -s^2)F_1(s) ds=\int_0^1 t\left(\frac{\tanh^{-1}(\rho t)}{\rho} -t\right)F_1(\rho t) dt.
 \end{align*}
 For $\rho$ close to $0$ a Taylor expansions shows that
 \begin{align*}
 \frac{\tanh^{-1}(\rho t)}{\rho} -t=\frac{\rho^2 t}{3}-\frac{\rho^4t^5}{5}+O(\rho^6t^7)
 \end{align*}
 where the $O$ term is a smooth function.
 Consequently, we infer that
$i_1\in C^\infty([0,1))
$
with
\begin{equation}\label{eq: i_0=0}
i_1'(0)= i_1^{(3)}(0)=0.
\end{equation}
Thus, $u_1\in C^\infty([0,1))$ and by combining \eqref{eq: i_0=0} with the fact that $\psi_0$ is even one easily establishes that $u_1(0)=u_1^{(3)}(0)=0$.
Therefore, we are left with checking the behavior of $u_1$ at $\rho=1$.
For this we remark that we can recast $u_1$ as
\begin{align*}
u_1(\rho)=\frac{\tanh^{-1}(\rho)}{\rho^3}\int_\rho^1sF_1(s)ds +\rho^{-3}\int_0^\rho s\tanh^{-1}(s) F_1(s) ds +r_1(\rho)
\end{align*}
where $r_1$ is a smooth function at $\rho=1$.
So, we only have to show that
$$
v_1(\rho):=\tanh^{-1}(\rho)\int_\rho^1 s F_1(s)ds +\int_0^\rho s\tanh^{-1}(s) F_1(s) ds
$$
is regular enough at 1.
Clearly, $v_1$ is continuous at $1$ and 
\begin{align*}
v'_1(\rho)=\frac{1}{1-\rho^2}\int_\rho^1 sF_1(s) ds.
\end{align*}
Further,
\begin{align*}
\int_\rho^1 sF_1(s) ds&=\int_{\rho-1}^0 (s+1)F_1(s+1) ds
\\
&=\int_{\frac{\rho-1}{1-\rho^2}}^0 (y(1-\rho^2)+1)F_1(y(1-\rho^2)+1)(1-\rho^2) ds.
\end{align*}
Hence,
\begin{align*}
v'_1(\rho)=\int_{-(1+\rho)^{-1}}^0 (y(1-\rho^2)+1)F_1(y(1-\rho^2)+1) ds
\end{align*}
and this is visibly smooth at $\rho=1$.
Summarizing, we see that \[
u_1\in C^3([0,1]), \qquad u_1'(0)=u_1^{(3)}(0)=0
\]
and from Eq.~\eqref{eq:u2} it follows that $\uf\in D(\widetilde{\Lf}_0)$.
\end{proof}
The last few lemmas allow us to invoke the Lumer Phillips Theorem. 
However, since we would rather like to work in standard $H^k$ spaces, we first prove the equivalences of the norms $\mathcal{E}_j$ with standard radial Sobolev norms. For this we will require the following version of Hardy's inequality
\begin{lem}\label{teclem1}
The estimates
\begin{align*}
\||.|^{-1} f\|_{L^{2}(\B^5_1)}\lesssim \|f\|_{H^1(\B^5_1)}
\end{align*}
and
\begin{align*}
\||.|^{-2} f\|_{L^{2}(\B^5_1)}\lesssim \|f\|_{H^2(\B^5_1)}
\end{align*}
hold for all $f\in C^2(\overline{\B^5_1})$. 
\end{lem}
\begin{proof}
The first estimate is just Lemma 4.7 in \cite{DonRao20}. For the second one, we let $$E:H^2(\B^5_1)\to H^2(\R^5)$$ be a bounded extension operator.
Then, by Hardy's inequality,
\begin{align*}
\||.|^{-2} f\|_{L^2(\B^5_1)}\leq\||.|^{-2} Ef\|_{L^2(\R^5)}\lesssim
\| Ef\|_{\dot{H}^2(\R^5)}\lesssim \| f\|_{H^2(\B^5_1)}.
\end{align*}
\end{proof}
\begin{lem}\label{teclem3}
The estimate
\begin{align*}
\||.|^{-1}f'\|_{L^2(\B^5_1)}\lesssim \|f\|_{H^2(\B^5_1)}
\end{align*}
holds for all $f\in C^2(\overline{\B^5_1})$. 
\end{lem}
\begin{proof}
This is an immediate consequence of Lemma 4.1 in \cite{DonSch16}.
\end{proof}

\begin{lem}
The norms $\|.\|_{\mathcal{E}_j}$ and $\|.\|_{H^j\times H^{j-1}(\B^5_1)}$ are equivalent on $D(\widetilde{\Lf}_0)$.
Consequently, they are also equivalent on $H^2\times H^1(\B^5_1)$.
\end{lem}

\begin{proof}
For $j=1$ this is Lemma 2.2 in \cite{DonRao20}. For $j=2$, the inequality 
\begin{align*}
\|.\|_{H^2\times H^1(\B^5_1)}\lesssim \|.\|_{\mathcal{E}_2}
\end{align*}  
is an immediate consequence of the estimate
\begin{align*}
\int_0^1 |u(\rho)|^2 \rho^4 d\rho\lesssim \int_0^1 |u'(\rho)|^2 \rho^4 d\rho +|u(1)|^2
\end{align*}
and the triangle inequality.
For the other inequality, we first note that 
\begin{align*}
|u(1)|\lesssim \left|\int_0^1 \partial_\rho(u(\rho)\rho^4) d\rho\right| \lesssim \|u\|_{H^1(\B^5_1)}+\||.|^{-1}u\|_{L^2(\B^5_1)}\lesssim \|u\|_{H^1(\B^5_1)}
\end{align*}
for all $u\in C^1(\overline{\B^5_1})$.
Therefore, 
\begin{align*}
|u_1(1)|^2+|u_2(1)|^2\lesssim \|(u_1,u_2)\|_{H^2\times H^1(\B^5_1)}^2.
\end{align*}
Further, 
\begin{align*}
\int_0^1 |u_1'(\rho)|^2 \rho^2 d\rho \lesssim \|(u_1, u_2)\|_{H^2\times H^1(\B^5_1)}^2 
\end{align*}
thanks to Lemma \ref{teclem3}.
Finally, 
\begin{align*}
\int_0^1 |u_1''(\rho)|^2 \rho^4 d\rho \lesssim \|(u_1, u_2)\|_{H^2\times H^1(\B^5_1)}^2 +
\int_0^1 |u_1'(\rho)|^2 \rho^2 d\rho   \lesssim \|(u_1, u_2)\|_{H^2\times H^1(\B^5_1)}^2.
\end{align*}
\end{proof}
Thus, the Lumer-Phillips Theorem immediately yields the following Lemma.
\begin{lem}
The operator $\widetilde{\Lf}_0$ is closable and its closure, denoted by $\Lf_0$, generates a semigroup $\Sf_0$ on $H^1\times L^2(\B^5_1)$ such that
\begin{align*}
\|\Sf_0(\tau)\ff\|_{H^1\times L^2(\B^5_1)}\lesssim e^{\frac{\tau}{2}} \|\ff\|_{H^1\times L^2(\B^5_1)}
\end{align*}
for all $\ff\in H^1\times L^2(\B^5_1)$ and all $ \tau\geq 0 $. 
Furthermore, the restriction of $\Sf_0$ to $H^2\times H^1(\B^5_1)$ satisfies
\begin{align*}
\|\Sf_0(\tau)\ff\|_{H^2\times H^1(\B^5_1)}\lesssim e^{-\frac{\tau}{2}} \|\ff\|_{H^2\times H^1(\B^5_1)}
\end{align*}
for all $\ff\in H^2\times H^1(\B^5_1)$ and all $\tau \geq0 $.
\end{lem}
To proceed, we use that $H^{\frac{3}{2}}\times H^{\frac{1}{2}}(\B^5_1)$ is an exact interpolation space of  $H^{2}\times H^1(\B^5_1)$ and  $H^{1}\times L^2(\B^5_1)$ of order $\frac{1}{2}$, see \cite[p.~317, Subsection 4.3.1.1, Theorem 1]{Tri95}, to conclude the next result.
\begin{lem}
The semigroup $\Sf_0$ satisfies
\begin{align*}
\|\Sf_0(\tau)\ff\|_{H^{\frac{3}{2}}\times H^{\frac{1}{2}}(\B^5_1)}\lesssim \|\ff\|_{H^{\frac{3}{2}}\times H^{\frac{1}{2}}(\B^5_1)}
\end{align*}
 for all $\ff \in H^{\frac{3}{2}}\times H^{\frac{1}{2}}(\B^5_1)$ and all $\tau \geq 0$.
\end{lem}
It is also vital for us that $\Sf_0$ satisfies appropriate Strichartz estimates, provided we restrict $T$ to the interval $[\frac{1}{2},\frac{3}{2}]$. This restriction leads to no loss of generality for us, as we are  only interested in values of $T$ which lie close to $1$ anyway. Henceforth, we assume that $T\in [\frac{1}{2},\frac{3}{2}]$ from now on.
\begin{lem}\label{lem:freeStrichartz}
Let $p\in [2,\infty]$ and $q\in [\frac{10}{3},\infty]$ be such that $\frac{1}{p}+\frac{5}{q}=1$. Then we have the estimate
\begin{align*}
\|\left[\Sf_0(\tau)\ff\right]_1\|_{L^p_\tau(\R_+)L^q(\B^5_1)}\lesssim \|\ff\|_{H^{\frac{3}{2}}\times H^{\frac{1}{2}}(\B^5_1)}
\end{align*}
for all $\ff \in H^{\frac{3}{2}}\times H^{\frac{1}{2}}(\B^5_1)$.
Furthermore, also the inhomogeneous estimate
\begin{align*}
\left\|\int_0^\tau\left[\Sf_0(\tau-\sigma)\hfh(\sigma)\right]_1 d\sigma\right\|_{L^p_\tau(I)L^q(\B^5_1)}\lesssim \|\hfh\|_{L^1(I)H^{\frac{3}{2}}\times H^{\frac{1}{2}}(\B^5_1)}
\end{align*}
holds for all $\hfh \in L^1(\R_+,H^{\frac{3}{2}}\times H^{\frac{1}{2}}(\B^5_1))$ and all intervals $I
\subset [0,\infty)$ containing $0$.
\end{lem}
\begin{proof}
This follows by restricting the standard Strichartz estimates for the
free wave equation to the
lightcone, cf.~\cite{DonWal22}.
\end{proof}
\begin{lem}\label{lem:freeStrichartz2}
The estimates
\begin{align*}
\|\left[\Sf_0(\tau)\ff\right]_1\|_{L^2_\tau(\R_+)W^{\frac{1}{2},5}(\B^5_1)}\lesssim \|\ff\|_{H^{\frac{3}{2}}\times H^{\frac{1}{2}}(\B^5_1)}
\end{align*}
and
\begin{align*}
\|\left[\Sf_0(\tau)\ff\right]_1\|_{L^6_\tau(\R_+)W^{1,\frac{30}{11}}(\B^5_1)}\lesssim \|\ff\|_{H^{\frac{3}{2}}\times H^{\frac{1}{2}}(\B^5_1)}
\end{align*}
hold for all $\ff \in H^{\frac{3}{2}}\times H^{\frac{1}{2}}(\B^5_1)$.
Furthermore, also the inhomogeneous estimates
\begin{align*}
\left\|\int_0^\tau\left[\Sf_0(\tau-\sigma)\hfh(\sigma)\right]_1 d\sigma\right\|_{L^2_\tau(\R_+)W^{\frac{1}{2},5}(\B^5_1)}\lesssim \|\hfh\|_{L^1(I)H^{\frac{3}{2}}\times H^{\frac{1}{2}}(\B^5_1)}
\end{align*}
and
\begin{align*}
\left\|\int_0^\tau\left[\Sf_0(\tau-\sigma)\hfh(\sigma)\right]_1 d\sigma\right\|_{L^6_\tau(\R_+)W^{1,\frac{30}{11}}(\B^5_1)}\lesssim \|\hfh\|_{L^1(I)H^{\frac{3}{2}}\times H^{\frac{1}{2}}(\B^5_1)}
\end{align*}
hold for all $\hfh \in L^1(\R_+,H^{\frac{3}{2}}\times H^{\frac{1}{2}}(\B^5_1))$ and all intervals $I
\subset [0,\infty)$ containing $0$.
\end{lem}
To get a better understanding of the dynamics of solutions which are close to $u^T_*$, we linearise the nonlinearity around this solution.
For this, we set $\Psi=\Phi+\Psi^*$, where $\Psi^*$ is the transformed blow up solution $u_*^T$, and formally linearize the nonlinearity around $\Psi^*$.
This results in a linear operator $\Lf'$ given by
\begin{align*}
\Lf' \uf(\rho) =\begin{pmatrix}
0\\
\frac{16}{(1+\rho^2)^2}u_1(\rho)
\end{pmatrix}
\end{align*}
and a formal nonlinear operator 
$\Nf$ given by
\begin{align*}
\Nf (\uf)(\rho) :=\begin{pmatrix}
0\\
N(\psi_{*_1}+u_1)(\rho)-N(\psi_{*_1})(\rho)-\frac{16}{(1+\rho^2)^2}u_1(\rho)
\end{pmatrix}.
\end{align*}
Lastly, we define
$\Lf:=\Lf_0+\Lf'$ and note that we have the following result.
\begin{lem}
The operator $\Lf'$ is a compact operator on $H^{s}\times H^{s-1}(\B^5_1)$ for any $s\geq 1$.
\end{lem}
\begin{proof}
This is an immediate consequence of the compactness of the embedding $H^{s}(\B^5_1)\hookrightarrow H^{s'}(\B^5_1)$
for $s>s'\geq0$.
\end{proof}
Consequently, the Bounded Perturbation Theorem implies that $\Lf$ will also generate a semigroup on each of the previously employed Sobolev spaces $H^s\times H^{s-1}(\B^5_1)$, which we denote by $\Sf$. 
With this, we can at least formally rewrite our equation in Duhamel form as
\begin{align}\label{eq:integraleq}
\Phi(\tau)=\Sf(\tau)\uf+\int_0^\tau \Sf(\tau-\sigma)\Nf(\Phi(\sigma)) d \sigma.
\end{align}
To make sense of this equation, we will show in the following that
$\Sf$ satisfies Strichartz estimates as in Lemma \ref{lem:freeStrichartz}, provided we project away the unstable direction. This will naturally give meaning to Eq.~\eqref{eq:integraleq} in an appropriate Strichartz space.

\subsection{Spectral analysis of $\Lf$}
From now on $\Lf$ will always denote the version of $\Lf$ that is a densely defined closed operator with
$$\Lf:D(\Lf)\subset H^2\times H^1(\B^5_1)\to H^2\times H^1(\B^5_1),$$
unless specifically stated otherwise. Then, for any $\lambda\in \C$ with
$\Re\lambda>-\frac12$, we have $\lambda \in \rho(\Lf_0)$ since
\begin{align*}
\|\Sf_0(\tau)\ff\|_{H^2\times H^1(\B^5_1)}\lesssim e^{-\frac12\tau}\|\ff\|_{H^2\times H^1(\B^5_1)}
\end{align*}
for all $\tau\geq0$ and all $\ff\in H^2\times H^1(\B^5_1)$. As a consequence, the identity
$$\lambda-\Lf=(1-\Lf'\Rf_{\Lf_0}(\lambda))(\lambda-\Lf_0)$$ with
$\Rf_{\Lf_0}(\lambda):=(\lambda-\Lf_0)^{-1}$ implies that any
spectral point $\lambda$  with $\Re\lambda> -1$ has to be an eigenvalue of finite algebraic multiplicity
by the spectral theorem for compact operators.

\begin{lem}
   \label{lem:spec}
  The point spectrum $\sigma_p(\Lf)$ of $\Lf$ is contained in the set $\{z\in \C:
  \Re z< 0\}\cup \{1\}$.
  Furthermore, the eigenvalue $1$ has geometric and algebraic multiplicity one and an associated eigenfunction is given by
$$
\gf(\rho)=\begin{pmatrix}
\frac{1}{1+\rho^2}\\ \frac{2}{(1+\rho^2)^2}
\end{pmatrix}.
$$
\end{lem}
\begin{proof}
That the point spectrum really is a subset of $\{z\in \C:
  \Re z< 0\}\cup \{1\}$ follows as in \cite{DonWal22}. To discern the properties of the eigenvalue $1$, we start by noting that
obviously $\gf\in D(\Lf)$ and a straightforward computation shows
that $(1-\Lf)\gf=\mathbf 0$.
Moreover, the calculations in the proof of Lemma \ref{lem:density range} show that the equation $(\lambda-\Lf)\uf=0$ is equivalent to
\begin{equation}\label{eq:u22}
u_2(\rho)=(1+\lambda)u_1(\rho)+\rho u_1'(\rho)
\end{equation}
and the second order linear differential equation
\begin{equation} \label{eq:spectraleq}
(\rho^2-1)u_1''(\rho)+\left(2(\lambda+2)\rho-\frac{4}{\rho}\right)u_1'(\rho)+(\lambda+2)(\lambda+1)u_1(\rho)-\frac{16}{(1+\rho^2)^2}u_1(\rho)=0.
\end{equation}
For $\lambda=1$ we use reduction of order to obtain a second solution to equation \eqref{eq:spectraleq},
$$
\widetilde{g}_1(\rho)=\frac{12\rho^3\tanh^{-1}(\rho)-9\rho^2-1}{\rho^3(\rho^2+1)}.
$$
Hence, any solution of Eq.~\eqref{eq:spectraleq} has to be a linear combination of $g_1$ and $\widetilde{g}_1$. As $\widetilde{g}_1 \notin H^1(\B^5_1)$, we conclude that an eigenfunction has to be a multiple of $\gf$ since the second component of any eigenfunction is uniquely determined by its first through \eqref{eq:u22}. Therefore, the geometric multiplicity of the eigenvalue $1$ is one. Moving on, we define $\Pf$ to be the spectral projection associated to this eigenvalue, i.e.,
$$
\Pf:=\int_\gamma \Rf_{\Lf}(\lambda) d \lambda,
$$
where $\gamma:[0,1]\to \C$, $\gamma(t)=1+\frac{e^{2\pi i t}}{2}$. Moreover, as the essential spectrum of $\Lf_0$ is invariant under compact perturbations, we see that $\dim \Pf < \infty$. Now, given that $\Pf$ is a projection, we can decompose $ H^2\times H^1(\B^5_1)$ into the closed subspaces $\rg \Pf$
and $\ker \Pf$. This also yields a decomposition of $\Lf$ into the
operators $\Lf_{\rg \Pf}$ and $\Lf_{\ker \Pf}$, which act as operators
on $\rg \Pf$ and $\ker \Pf$, respectively.
The inclusion $\langle \gf \rangle \subset \rg \Pf$ is immediate and
we claim that in fact $\rg\Pf=\langle\gf\rangle$. To show this, we first remark
that the finite-dimensional operator $(\I_{\rg \Pf}-\Lf_{\rg \Pf}): \rg \Pf \to \rg \Pf$
is nilpotent as its only eigenvalue is
$0$. Thus, there exists a minimal $n\in \mathbb{N}$ such that
$(\I_{\rg\Pf}-\Lf_{\rg \Pf})^n \uf=0$ for all $\uf \in \rg \Pf$. If
$n=1$, we are done. If not, then there exists a $\vf \in \rg \Pf$ such
that $(\I_{\rg \Pf}-\Lf_{\rg \Pf})\vf =\gf$. This implies that $v_1 $ satisfies the inhomogeneous ODE
\begin{equation*}
(\rho^2-1)v_1''(\rho)+\left(6\rho-\frac{4}{\rho}\right)v_1'(\rho)+\left(6-\frac{16}{(1+\rho^2)^2}\right)v_1(\rho)=G(\rho)
\end{equation*}
with $G(\rho)=g_2(\rho)+3 g_1(\rho)+\rho g_1'(\rho)=\frac{7+\rho^2}{(1+\rho^2)^2}.$
By the variation of constants formula, $v_1$ has to be of the form
\begin{align*}
v_1(\rho)=&c_1 g_1(\rho)+c_2\widetilde{g}_1(\rho)-g_1(\rho)\int_{\rho}^{1} \frac{\widetilde{g_1}(s)G(s)}{\left(1-s^2\right)W(g_1,\widetilde{g_1})(s)} d s
\\
&-
\widetilde{g_1}(\rho)\int_{0}^\rho \frac{g_1(s)G(s)}{\left(1-s^2\right)W(g_1,\widetilde{g_1})(s)} ds
\end{align*}
with $c_1,c_2\in \C$.
Note that
$$
W(g_1,\widetilde{g}_1)(\rho)=\frac{3}{\rho^4(1-\rho^2)}
$$
is strictly positive on $(0,1)$ and therefore nonvanishing on that interval.
Evidently, both $
\frac{g_1(\rho)G(\rho)}{(1-\rho^2)W(g_1,\widetilde{g_1})(\rho)}$ and $
\frac{\widetilde{g_1}(\rho)G(\rho)}{(1-\rho^2)W(g_1,\widetilde{g_1})(\rho)}$
are continuous on $[0,1]$. 
Consequently, since $\widetilde{g}_1 \notin L^2(\B^5_1)$, we must have $c_2=0$.
Furthermore, $|\widetilde g_1'(\rho)|\simeq (1-\rho)^{-1}$ near
$\rho=1$ and thus, for $v$ to be in $H^1(\B^5_1)$, we must necessarily have
\[ \int_0^1 \frac{g_1(s)G(s)}{(1-s^2)W(g_1, \widetilde g_1)(s)}ds=0. \]
 This is however impossible due to the strict positivity of the integrand on $(0,1)$.
\end{proof}

\begin{lem}\label{lem:spec2}
The essential spectrum of $\Lf$, denoted by $\sigma_e(\Lf)$ satisfies $$\sigma_e(\Lf)\subset\{z\in
\C:\Re(z)\leq -\frac12\}.$$ In addition, any spectral point $\lambda$ with $\Re \lambda>-\frac12$ is an eigenvalue of finite algebraic multiplicity and there exist only finitely many such spectral points.
\end{lem}
\begin{proof}
The first claim is an immediate consequence of the growth bound 
\begin{align*}
\|\Sf_0(\tau)\ff\|_{H^2\times H^1(\B^5_1)}\lesssim e^{-\frac12\tau} \|\ff\|_{H^2\times H^1(\B^5_1)}.
\end{align*}
The second follows from invoking Theorem B.1 in \cite{Glo21}.
\end{proof}
A calculation which is very similar to the one done in the proof of Lemma 2.6 in \cite{DonRao20} yields our next result.
\begin{lem}\label{resbound}
Let $\eta > -\frac12$. Then there exist constants $C_\eta,
K_\eta>0$ such that
\[
	\|\Rf_{\Lf}(\lambda)\ff\|_{H^2\times H^1(\B^5_1)}\leq C_\eta\|\ff\|_{H^2\times H^1(\B^5_1)}
\]
for all $\lambda \in \C$ satisfying $|\lambda|\geq K_\eta$ and $\Re \lambda\geq\eta$ and all $\ff\in H^2\times H^1(\B^5_1)$. 
\end{lem}
Let now $\Qf$ be the spectral projection associated to the finite set of eigenvalues 
$$
\{\lambda\in \sigma(\Lf):-\frac12<\Re \lambda<0\}.$$
Moreover, we remark that when viewed as a densely defined, closed operator on $H^{\frac{3}{2}}\times H^{\frac{1}{2}}(\B^5_1)$, the calculations in the proof of Lemma \ref{lem:spec} show that in this case we have that
\begin{align*}
\sigma(\Lf)\subset\{z\in
\C:\Re(z)\leq 0\}\cup \{1\} \text{\, and \,}\sigma_p(\Lf)\subset\{z\in
\C:\Re(z)<0\} \cup \{1\}
\end{align*}
and $1$ remains a simple eigenvalue.
We denote by $\Pf$ the corresponding bounded projection  $\Pf:H^{\frac{3}{2}}\times H^{\frac{1}{2}}(\B^5_1)\to \langle\gf\rangle$.

\begin{lem}\label{proectionsemigroup}
Let $\eta >-\frac12$. Then there exists a constant $C_\eta >0$ such that
$$
\| \Sf(\tau)(\I-\Qf)(\I-\Pf)\ff\|_{H^2\times H^1(\B^5_1)} \leq C_\eta e^{\eta\tau}\|\ff\|_{H^2\times H^1(\B^5_1)}
$$
for all $\ff \in H^2\times H^1(\B^5_1)$ and all $\tau \geq 0$.
\end{lem}
\begin{proof}
This Lemma follows immediately from Lemma \ref{resbound} and the
Gearhart-Prüss-Greiner Theorem, see e.g.~\cite{EngNag99}, p.~302,
Theorem 1.11, since $$\sigma(\Lf_{\ker \Pf\cap \ker\Qf})\subset \{\lambda\in \C: \Re(\lambda)\leq -\frac12\}.$$
\end{proof}
As the growth estimate from Lemma \ref{proectionsemigroup} does not
help us at the critical regularity, at which analogous considerations
would yield an exponentially growing bound for the semigroup, a more sophisticated analysis is needed. So, let $\ff \in C^\infty \times C^\infty(\overline{\B^5_1})$ and set $$\widetilde{\ff}:= (\I-\Qf) (\I-\Pf)\ff \in D(\Lf).$$ Then, for any $\eta>-\frac12$, Laplace inversion yields
\begin{equation}
\Sf(\tau)\widetilde{\ff}=\lim_{N \to \infty}\frac{1}{2\pi i}\int_{\eta-i N}^{\eta+i N}e^{\lambda\tau}\Rf_{\Lf}(\lambda) \widetilde{\ff}d \lambda,
\end{equation}
see \cite{EngNag99}, p.~234, Corollary 5.15.
Hence, to obtain enough qualitative information on the semigroup $\Sf(\I-\Qf)(\I-\Pf)$, we need to investigate $\Rf_{\Lf}(\lambda)$. To that end we remark that 
$\uf=\Rf_{\Lf}(\lambda)\widetilde{\ff}$ implies $(\lambda-\Lf)\uf=\widetilde{\ff},$ which in turn implies
\begin{equation}\label{eq:resolventeq}
(\rho^2-1)u_1''(\rho)+\left(2(\lambda+2)\rho-\frac{4}{\rho}\right)u_1'(\rho)+(\lambda+2)(\lambda+1)u_1-\frac{16}{(1+\rho^2)^2}u_1(\rho)=F_\lambda(\rho)
\end{equation}
where $F_\lambda(\rho)=f_2(\rho)+(\lambda+2)f_1(\rho)+\rho f_1'(\rho)$.
Accordingly, our next step will be a detailed analysis of
Eq.~\eqref{eq:resolventeq}.
\section{ODE analysis}

\subsection{Preliminary transformations}
To put many of the tediously involved function into a manageable fashion, we introduce function of symbol type as follows. Let $I\subset \R$, $\rho_0\in \R \setminus I$, and $\alpha \in \R$. We say that a smooth function $f:I \to \C$ is of symbol type and write $f(\rho)=\O((\rho_0-\rho)^{\alpha})$ if
\begin{align*}
|\partial_\rho^n f(\rho)|\lesssim_n |\rho_0-\rho|^{\alpha-n},
\end{align*}
 for all $\rho \in I$ and all $n\in \mathbb{N}_0$. Similarly, for $g:\C\to \R$ we write $g(\lambda)=\O(\langle\omega\rangle^{\alpha})$, if
 \begin{align*}
|\partial_\omega^n f(\omega)|\lesssim_n \langle\omega\rangle^{\alpha-n}
\end{align*}
where $\langle\omega\rangle$ denotes the Japanese bracket $\sqrt{1+|.|^2}$.
Analogously,
$$
h(\rho,\lambda)=\O((\rho-\rho_0)^{\alpha} \langle\omega\rangle^{\beta}) \quad \text{ if } \quad |\partial_\rho^n\partial_\omega^k h(\rho,\lambda)|\lesssim_{n,k} |\rho_0-\rho|^{\alpha-n}\langle\omega\rangle^{\beta-k},
$$
for all $\ell,k\in \mathbb N$ and $\alpha,\beta \in \R$.
Motivated by the spectral equation \eqref{eq:resolventeq}, we study the ODE
\begin{align}\label{eq:generaleq} 
(1-\rho^2) u''(\rho)+\left(\frac{4}{\rho}-2(\lambda+2)\rho\right)u'(\rho)-(\lambda+1)(\lambda+2)u(\rho)-V(\rho)u(\rho)=-F_\lambda(\rho)
\end{align}
for $\Re\lambda\in [-\frac{3}{4},\frac{3}{4}]$, $\lambda\neq 0$, and an arbitrary even potential $V\in C^{\infty}([0,1])$.
To get rid of the first order term we set
 $$v(\rho)=\rho^{2}(1-\rho^2)^{\frac{\lambda}{2}}u(\rho),$$
 which, for $F_\lambda=0$, turns Eq.~\eqref{eq:generaleq} into
 \begin{align}\label{eq:nofirstorder}
v''(\rho)+\frac{-2+\rho^2(2+2\lambda-\lambda^2)}{\rho^2(1-\rho^2)^2}v(\rho)=\frac{V(\rho)}{1-\rho^2}v(\rho).
 \end{align}
One of the main tools to study Eq.~\eqref{eq:nofirstorder} is the diffeomorphism $\varphi:(0,1)\to (0,\infty)$, given by
 $$
 \varphi(\rho):=\frac{1}{2}(\log(1+\rho)-\log(1-\rho)).
 $$
Observe that
 $$
 \varphi'(\rho)=\frac{1}{1-\rho^2}
 $$
 and that the associated Liouville-Green Potential $Q_{\varphi}$, defined by
 $$
 Q_\varphi(\rho):=-\frac{3}{4}\frac{\varphi''(\rho)^2}{\varphi'(\rho)^2}+\frac{1}{2}\frac{\varphi'''(\rho)}{\varphi'(\rho)},
 $$ is given by
 $$
 Q_{\varphi}(\rho )=\frac{1}{(1-\rho^2)^2}.
 $$
Hence, we rewrite Eq.~\eqref{eq:nofirstorder} as
\begin{equation}
  \begin{split}
&\quad v''(\rho)+\frac{-1+2\lambda-\lambda^2}{(1-\rho^2)^2}v(\rho)-\frac{2}{\varphi(\rho)^2(1-\rho^2)^2}v(\rho)+Q_{\varphi}(\rho)v(\rho)
\\
&=\frac{V(\rho)}{1-\rho^2}+\left(\frac{2}{\rho^2(1-\rho^2)^2}-\frac{2}{(1-\rho^2)^2}-\frac{2}{\varphi(\rho)^2(1-\rho^2)^2}\right)v(\rho).
\end{split}
\end{equation}
Next, we perform a Liouville-Green transformation, that is, we set  $w(\varphi(\rho)):= \varphi'(\rho)^{\frac{1}{2}}v(\rho)$, which transforms
\begin{align}\label{eq:beforebessel}
v''(\rho)+\frac{-1+2\lambda-\lambda^2}{(1-\rho^2)^2}v(\rho)-\frac{2}{\varphi(\rho)^2(1-\rho^2)^2}v(\rho)+Q_{\varphi}(\rho)v(\rho)=0
\end{align}
into
\begin{align}\label{eq:bessel}
w''(\varphi(\rho))-(1-\lambda)^2w(\varphi(\rho))-\frac{2}{\varphi(\rho)^2}w(\varphi(\rho))=0.
\end{align}
This is now a Bessel equation with a fundamental system given by
\begin{align*}
&\cos(a(\lambda)\varphi(\rho))-\frac{\sin(a(\lambda)\varphi(\rho))}{a(\lambda)\varphi(\rho)}
\\
&\sin(a(\lambda)\varphi(\rho))+\frac{\cos(a(\lambda)\varphi(\rho))}{a(\lambda)\varphi(\rho)}
\end{align*}
with $a(\lambda)=i(1-\lambda)$.
From this we infer that
\begin{align*}
b_1(\rho,\lambda)&=\sqrt{1-\rho^2}\left(\frac{\sin(a(\lambda)\varphi(\rho))}{a(\lambda)\varphi(\rho)}-\cos(a(\lambda)\varphi(\rho))\right)
\\
b_2(\rho,\lambda)&=\sqrt{1-\rho^2}\left(\sin(a(\lambda)\varphi(\rho))+\frac{\cos(a(\lambda)\varphi(\rho))}{a(\lambda)\varphi(\rho)}\right)
\end{align*}
is a fundamental system of Eq.~\eqref{eq:beforebessel}.

\subsection{Construction of fundamental systems}

\begin{lem}\label{lem: free ODE near 1}
There exist $r>0$ and $\rho_0\in [0,1)$ such that for $\rho \in [\rho_\lambda,1),$ where \\$\rho_\lambda:= \min\{\frac{r}{|1-\lambda|},\rho_0\}$, and $\lambda\neq 0$ with $-\frac34 \leq \Re\lambda\leq \frac{3}{4}$
the equation
\begin{equation}\label{eq:no V}
v''(\rho)+\frac{-2+\rho^2(2+2\lambda-\lambda^2)}{\rho^2(1-\rho^2)^2}v(\rho)=0
\end{equation} has a fundamental system of the form
\begin{align*}
h_1(\rho,\lambda)=&\sqrt{1-\rho^2}\left(\frac{1-\rho}{1+\rho}\right)^{\frac{1-\lambda}{2}}
\left[1+(1-\rho)\O(\langle\omega\rangle^{-1})+\O(\rho^{-1}(1-\rho)^2\langle\omega\rangle^{-1})\right]
\\
h_2(\rho,\lambda)=&\sqrt{1-\rho^2}\left(\frac{1+\rho}{1-\rho}\right)^{\frac{1-\lambda}{2}}
\left[1+(1-\rho)\O(\langle\omega\rangle^{-1})+\O(\rho^{-1}(1-\rho)^2\langle\omega\rangle^{-1})\right],
\end{align*}
where $\omega=\Im\lambda$.
\end{lem}
\begin{proof}
We rewrite Eq.~\eqref{eq:no V} as
\begin{align*}
v''(\rho)+\frac{2\lambda-\lambda^2}{(1-\rho^2)^2}v(\rho)=\frac{2}{\rho^2(1-\rho^2)}v(\rho)
\end{align*}
and note that the equation 
\begin{align*}
w''(\rho)+\frac{2\lambda-\lambda^2}{(1-\rho^2)^2}w(\rho)=0
\end{align*}
has a fundamental system of solutions given by 
\begin{align*}
w_1(\rho,\lambda)&=\sqrt{1-\rho^2}\left(\frac{1-\rho}{1+\rho}\right)^{\frac{1-\lambda}{2}}
\\
w_2(\rho,\lambda)&=\sqrt{1-\rho^2}\left(\frac{1+\rho}{1-\rho}\right)^{\frac{1-\lambda}{2}}.
\end{align*}
The Wronskian of these solutions is given by
\begin{align*}
W(w_1(.,\lambda),w_2(.,\lambda))=2(1-\lambda).
\end{align*}
Therefore, Duhamel's formula suggests the Volterra equation
\begin{equation}\label{eq:int1}
\begin{split}
w(\rho,\lambda)=& w_1(\rho,\lambda)+\int_\rho^{\rho_1}\frac{ w_1(\rho,\lambda)w_2(s,\lambda)}{(1-\lambda)s^2(1-s^2)} w(s,\lambda) ds 
\\
&-\int_\rho^{\rho_1}\frac{ w_2(\rho,\lambda)w_1(s,\lambda)}{(1-\lambda)s^2(1-s^2)} w(s,\lambda) ds
\end{split}
\end{equation}
for $\rho>\frac{1}{|1-\lambda|}$.
As $w_1(.,\lambda)$ does not vanish on $(0,1)$, we can divide Eq.~\eqref{eq:int1} by $ w_1$. For the new variable  $\widetilde w=\frac{w}{w_1}$, we then obtain the equation
\begin{align*}
\widetilde w(\rho,\lambda)&= 1+\int_\rho^{\rho_1}\frac{ w_1(s,\lambda)w_2(s,\lambda)}{(1-\lambda)s^2(1-s^2)} \widetilde w(s,\lambda) ds 
\\
&\quad -\int_\rho^{\rho_1}\frac{ w_2(\rho,\lambda)w_1^2(s,\lambda)}{w_1(\rho,\lambda)(1-\lambda)s^2(1-s^2)} \widetilde w(s,\lambda) ds
\\
&=1+\int_\rho^{\rho_1}\frac{ 1-\left(\frac{1+\rho}{1-\rho}\frac{1-s}{1+s}\right)^{1-\lambda}}{(1-\lambda)s^2} \widetilde w(s,\lambda) ds
\\
&=:1+\int_\rho^{\rho_1}K(\rho,s,\lambda) \widetilde w(s,\lambda) ds.
\end{align*} 
From 
$\frac{1}{|1-\lambda|}\leq \rho\leq s,
$ we conclude that
\begin{align*}
\int_{\frac{1}{|1-\lambda|}}^{\rho_1} \sup_{\rho\in [\frac{1}{|1-\lambda|},s]}\left|\frac{1-\left(\frac{1+\rho}{1-\rho}\frac{1-s}{1+s}\right)^{1-\lambda}}{ (1-\lambda)s^2} \right| ds\lesssim \int_{\frac{1}{|1-\lambda|}}^{\rho_1}\frac{1}{s^2|1-\lambda|} ds \lesssim 1
\end{align*}
independent of $\rho_1 \in [\frac{1}{|1-\lambda|},1]$. Consequently, we are able to set $\rho_1=1$
and use Lemma B.1 in \cite{DonSchSof11} to infer the existence of a
unique solution $\widetilde w$ to Eq.~\eqref{eq:int1} of the form 
$$
\widetilde w(\rho,\lambda)=1+ O(\rho^{-1}\langle\omega\rangle^{-1}).
$$
Strictly speaking, the $O$-term also depends on
$\Re\lambda$ but as this dependence is of no relevance to
us, we suppress it in our notation.
Having established the existence of $h_1=w\widetilde w_1$, one
proceeds in the same manner as in \cite{DonWal22}, Lemma 4.1, to conclude that $h_1$ is indeed of the desired form and that a second solution $h_2$ to Eq.~\eqref{eq:no V} of the claimed form can be constructed.
\end{proof}
Without loss of generality we can assume that neither $h_1(.,\lambda)$
nor $h_2(.,\lambda)$ vanishes anywhere on $[\rho_\lambda,1)$, as we can enlarge $r$ and $\rho_0$ if necessary.
We now set
$\widehat{\rho}_\lambda:=\min\{\frac{1}{2}(\rho_0+1),\frac{2r}{|a(\lambda)|}\}\in
(\rho_\lambda,1)$
and with this, we turn to the full equation \eqref{eq:nofirstorder}.

\begin{lem}\label{Besselsol}
Eq.~\eqref{eq:nofirstorder}
has a fundamental system of the form
\begin{align*}
\psi_1(\rho,\lambda)=&b_1(\rho,\lambda)[1+\O(\rho^2\langle\omega\rangle^0)]
\\
\psi_2(\rho,\lambda)=& b_2(\rho,\lambda)[1+\O(\rho^2\langle\omega\rangle^0)]+\O(\rho\langle\omega\rangle^{-2}) 
\end{align*}
for all $\rho\in (0,\widehat \rho_\lambda]$ and all $\lambda\neq 0$ with $-\frac34 \leq \Re\lambda\leq \frac{3}{4}$.
\end{lem}
\begin{proof}
We start by noting that for $\rho\in(0,\widehat \rho_{\lambda}]$, the functions $b_1$ and $b_2$ satisfy
\begin{equation}\label{eq:symbolformb}
\begin{split}
b_1(\rho,\lambda)&=\O(\rho^2\langle\omega\rangle^{2})
\\
b_2(\rho,\lambda)&=\O(\rho^{-1}\langle\omega\rangle^{-1})
\end{split}
\end{equation}
and that their Wronskian is given by
\begin{align*}
W(b_1(.,\lambda),b_2(.,\lambda))=i(1-\lambda).
\end{align*}
Therefore, we have to solve the fixed point problem
\begin{equation}\label{eq:int2}
\begin{split}
b(\rho,\lambda)=& b_1(\rho,\lambda)-\int_0^\rho \frac{ b_1(\rho,\lambda)b_2(s,\lambda)\widetilde{V}(s)}{i(1-\lambda)(1-s^2)} b(s,\lambda) ds 
\\
&+\int_0^\rho\frac{ b_2(\rho,\lambda)b_1(s,\lambda)\widetilde{V}(s)}{i(1-\lambda)(1-s^2)} b(s,\lambda) ds
\end{split}
\end{equation}
with $$
\widetilde{V}(\rho)=V(\rho)+\frac{2}{\rho^2(1-\rho^2)}-\frac{2}{1-\rho^2}-\frac{2}{\varphi(\rho)^2(1-\rho^2)}
$$
Observe that $\widetilde V\in C^\infty([0,1))$.
We claim that $b_1(.,\lambda)$ does not vanish on $(0,1)$. This follows from the fact that
the zeros of $J_{\frac{3}{2}}$ are all real (see \cite{Olv97}, p.~244
Theorem 6.2) and any zero of $b_1(.,\lambda)$ is a zero of
$J_{\frac{3}{2}}$. Since $a(\lambda)$ always has nonzero imaginary
part for $\Re\lambda\in [-\frac{3}{4},\frac{3}{4}]$, we see that the
argument of the Bessel function is always nonreal. Hence,  we can divide Eq.~\eqref{eq:int2} by $b_1$. Upon setting $\widetilde{b}=\frac{b}{b_1}$ we obtain the Volterra equation
\begin{align*}
\widetilde b(\rho,\lambda)=& 1-\int_0^\rho \frac{ b_1(s,\lambda)b_2(s,\lambda)\widetilde{V}(s)}{i(1-\lambda)(1-s^2)} \widetilde b(s,\lambda) ds 
\\
&+\int_0^\rho\frac{ b_2(\rho,\lambda)b_1^2(s,\lambda)\widetilde{V}(s)}{i(1-\lambda)(1-s^2)b_1(\rho,\lambda)} \widetilde b(s,\lambda) ds
=:1+\int_0^\rho K(\rho,s,\lambda) \widetilde b(s,\lambda) ds.
\end{align*}
Using the estimates \eqref{eq:symbolformb},
we see that
\begin{align*}
|b_2(\rho,\lambda)b_1(\rho,\lambda)|&\lesssim \rho \langle\omega\rangle,
\\
\left|\frac{b_2(\rho,\lambda)}{b_1(\rho,\lambda)}b_1(s,\lambda)^2\right|&\lesssim s \langle\omega\rangle
\end{align*}
for all $0\leq s\leq\rho\leq \widehat{\rho}_\lambda$ 
and so
$$
\int_0^{\widehat{\rho}_\lambda} \sup_{\rho\in [s,\widehat\rho_\lambda]}|K(\rho,s,\lambda)|ds \lesssim \langle\omega\rangle^{-2}.
$$
Hence, a Volterra iteration yields
the existence of a unique solution $\widetilde{b}(\rho,\lambda)$ to Eq.~\eqref{eq:int2} that satisfies $\widetilde{b}(\rho,\lambda)=1+O(\rho^2\langle \omega\rangle^0)$.
Furthermore, since all the involved functions behave like symbols\footnote{Strictly
  speaking, $\widehat{\rho}_\lambda$ not differentiable at $\frac{1}{2}(\rho_0+1)$. However, this is inessential and can easily be remedied by using a smoothed out version of $\widehat{\rho}_\lambda$.},
Appendix B of \cite{DonSchSof11} shows that
$\widetilde{b}(\rho,\lambda)=1+\O(\rho^2\langle\omega\rangle^0)$ and thus, we
obtain the existence of a solution to
Eq.~\eqref{eq:nofirstorder} of the form
$$\psi_1(\rho,\lambda)=b_1(\rho,\lambda)[1+\O(\rho^2\langle\omega\rangle^0)].$$

To construct the second solution stated in the lemma, we pick a $\rho_1 \in (0,1]$ such that $\psi_1$ does not vanish for $\rho\leq\min\{\rho_1,\widehat{\rho}_\lambda\}=:\widetilde{\rho}_\lambda$ for any $\lambda \in \C\setminus\{0\}$ with $-\frac{3}{4} \leq \Re\lambda\leq \frac{3}{4}$. Moreover, as
$\widetilde{b}_1(\rho,\lambda):=b_1(\rho,\lambda)\int_{\rho}^{\widetilde{\rho}_\lambda} b_1(s,\lambda)^{-2} ds$ is also a solution of Eq.~\eqref{eq:beforebessel}, there exist constants $c_1(\lambda),c_2(\lambda)$ such that
\begin{align*}
b_2(\rho,\lambda)=c_1(\lambda) b_1(\rho,\lambda)+ c_2(\lambda)\widetilde{b}_1(\rho,\lambda).
\end{align*}
Explicitly, these constants are given by
\begin{align*}
c_1(\lambda)&=\frac{W(b_2(.,\lambda),\widetilde{b}_1(.,\lambda))}{W(b_1(.,\lambda),\widetilde{b}_1(.,\lambda))}\\
c_2(\lambda)&=-\frac{W(b_2(.,\lambda),b_1(.,\lambda))}{W(b_1(.,\lambda),\widetilde{b}_1(.,\lambda))}.
\end{align*}
Using that
$W(b_2(.,\lambda),b_1(.,\lambda))=-\frac{2}{\pi}$ and $ W(b_1(.,\lambda),\widetilde{b}_1(.,\lambda))=-1$, we infer that
$c_2=-\frac{2}{\pi}$ and $c_1(\lambda)=-W(b_2(.,\lambda),\widetilde{b_1}(.,\lambda))$.
Next, evaluating $W(b_2(.,\lambda),\widetilde{b}_1(.,\lambda))$ at $\widetilde{\rho}_\lambda$ yields
\[
W(b_2(.,\lambda),\widetilde{b_1}(.,\lambda))=-b_2(\widetilde{\rho}_\lambda,\lambda)b_1(\widetilde{\rho}_\lambda,\lambda)^{-1}=\O(\langle\omega\rangle^{0}).
\]
Keeping these facts in mind, we now turn our attention to constructing $\psi_2$. For this, we remark that a second solution of Eq.~\eqref{eq:nofirstorder} is given by
$\widetilde{\psi}_1(\rho,\lambda)=\psi_1(\rho,\lambda)\int_{\rho}^{\widetilde{\rho}_\lambda} \psi_1(s,\lambda)^{-2} ds$.
Considering this, we calculate
\begin{align*}
\psi_2(\rho,\lambda):&=c_1(\lambda)\psi_1(\rho,\lambda)+ c_2\psi_1(\rho,\lambda)\int_{\rho}^{\widetilde{\rho}_\lambda} \psi_1(s,\lambda)^{-2}ds 
\\
&=c_1(\lambda)\psi_1(\rho,\lambda)+
   c_2\psi_1(\rho,\lambda)\int_{\rho}^{\widetilde{\rho}_\lambda}
   b_1(s,\lambda)^{-2}ds \\
&\quad +c_2\psi_1(\rho,\lambda)\int_{\rho}^{\widetilde{\rho}_\lambda}\left [
   \psi_1(s,\lambda)^{-2}-b_1(s,\lambda)^{-2}\right ] ds
\\
&=b_2(\rho,\lambda)[1+\O(\rho^2\langle\omega\rangle^0)]
+c_2\psi_1(\rho,\lambda)\int_{\rho}^{\widetilde{\rho}_\lambda} \frac{\O(s^2\langle\omega\rangle^0)}{b_1(s,\lambda)^2[1+\O(s^2\langle\omega \rangle^0)]^2} ds.
\end{align*}
Since $b_1(\rho,\lambda)^{-2}=\O(\rho^{-4}\langle\omega\rangle^{-4})$,
we obtain
 \[
\int_{\rho}^{\widetilde{\rho}_\lambda} \frac{\O(s^2\langle\omega\rangle^0)}{b_1(s,\lambda)^2[1+\O(s^2\langle\omega \rangle^0)]^2} ds=\O(\rho^0\langle\omega\rangle^{-3})+\O(\rho^{-1}\langle\omega\rangle^{-4})=\O(\rho^{-1}\langle\omega\rangle^{-4})
\]
where the last inequality follows as we only consider values of  $\rho$ that are smaller than $ \widetilde{\rho}_\lambda$.
Finally, for $|\lambda|$ large enough, we see that
$\widetilde{\rho}_\lambda=\widehat{\rho}_\lambda$ and so we can safely
assume that $\widetilde{\rho}_\lambda=\widehat{\rho}_\lambda$.
\end{proof}
One final Volterra iteration based on $h_1$ and $h_2$ yields the following result.
\begin{lem}
There exists a fundamental system
for Eq.~\eqref{eq:nofirstorder} of the form
\begin{align*}
\psi_3(\rho,\lambda)&= h_1(\rho,\lambda)[1+(1-\rho) \O(\langle\omega\rangle^{-1})+\O(\rho^0 (1-\rho)^2 \langle \omega \rangle^{-1})]
\\
\psi_4(\rho,\lambda)&=h_2(\rho,\lambda)[1+(1-\rho)\O(\langle\omega\rangle^{-1})+\O(\rho^0 (1-\rho)^2 \langle \omega \rangle^{-1})]
\end{align*}
for all $\rho\in [\rho_\lambda,1)$ and all $\lambda\neq 0$ with $-\frac34 \leq \Re\lambda\leq \frac{3}{4}$.
\end{lem}
\begin{lem}\label{connectioncoef}
For $\rho\in[\rho_\lambda,\widehat{\rho}_\lambda]$ the solutions $\psi_3$ and $\psi_4$ have the representations
\begin{align*}
\psi_3(\rho,\lambda) &= c_{1,3}(\lambda)\psi_1(\rho,\lambda)+ c_{2,3}(\lambda)\psi_2(\rho,\lambda)\\
\psi_4(\rho,\lambda) &= c_{1,4}(\lambda)\psi_1(\rho,\lambda)+ c_{2,4}(\lambda)\psi_2(\rho,\lambda),
\end{align*}
with
\begin{align*}
c_{1,3}(\lambda)&= \frac{W(h_1(.,\lambda),b_2(.,\lambda))(\rho_\lambda)}{i(1-\lambda)}+\O(\langle\omega\rangle^{-1})
\\
c_{2,3}(\lambda)&= -\frac{W(h_1(.,\lambda),b_1(.,\lambda))(\rho_\lambda)}{i(1-\lambda)}+\O(\langle\omega\rangle^{-1})
\end{align*}
and
\begin{align*}
	c_{1,4}(\lambda)&= \frac{W(h_2(.,\lambda),b_2(.,\lambda))(\rho_\lambda)}{i(1-\lambda)}+\O(\langle\omega\rangle^{-1})
	\\
	c_{2,4}(\lambda)&=-\frac{W(h_2(.,\lambda),b_1(.,\lambda))(\rho_\lambda)}{i(1-\lambda)}+\O(\langle\omega\rangle^{-1}).
\end{align*}
\end{lem}
\begin{proof}
We know the explicit representations
\begin{align*}
c_{1,3}(\lambda)&=\frac{W(\psi_3(.,\lambda),\psi_2(.,\lambda))}{W(\psi_1(.,\lambda),\psi_2(.,\lambda))}\\
c_{2,3}(\lambda)&=-\frac{W(\psi_3(.,\lambda),\psi_1(.,\lambda))}{W(\psi_{1}(.,\lambda),\psi_{2}(.,\lambda))}
\end{align*}
and computing the connection coefficients reduces to calculating these Wronskians.
Evaluating the Wronskian $W(\psi_1(.,\lambda),\psi_2(.,\lambda))$ at
$\rho=0$ yields
$$
W(\psi_{1}(.,\lambda),\psi_{2}(.,\lambda))=i(1-\lambda)
$$
while an evaluation at $\rho_\lambda$ yields
\begin{align*}
 W(\psi_3(.,\lambda),\psi_2(.,\lambda))&= W(
    h_1(.,\lambda), b_2(.,\lambda))(\rho_\lambda)[1+\O(\langle
    \omega\rangle^{-1})] \\
  &\quad +h_1(\rho_\lambda,\lambda)b_2(\rho_\lambda,\lambda)\left[ \O(\langle\omega\rangle^{0})+\O(\langle\omega\rangle^{-1})\right] \\
&=W(h_1(.,\lambda),b_2(.,\lambda))(\rho_\lambda)+\O(\langle\omega\rangle^{0}).
\end{align*}
Consequently, 
$$c_{1,3}(\lambda)= \frac{ W(h_1(.,\lambda),b_2(.,\lambda))(\rho_\lambda)}{i(1-\lambda)}+\O(\langle\omega\rangle^{-1})$$
and the remaining coefficients are computed analogously.
\end{proof}
We can patch together the solutions of the ``free equation" in the
same fashion. To this end, let $\psi_{\mathrm{f}_1}$ and
$\psi_{\mathrm{f}_2}$ be the solutions obtained from Lemma
\ref{Besselsol} in the case $V=0$ and,
for notational convenience, we set $\psi_{\mathrm{f}_3}:=h_1$ and $\psi_{\mathrm{f}_4}:=h_2$.
\begin{lem}\label{connectioncoeffree}
For $\rho\in [\rho_\lambda,\widehat{\rho}_\lambda]$, the solutions $\psi_{\mathrm{f}_3}$ and $\psi_{\mathrm{f}_4}$ have the representations
\begin{align*}
\psi_{\mathrm{f}_3}(\rho,\lambda) &= c_{\mathrm{f}_{1,3}}(\lambda)\psi_{\mathrm{f}_1}(\rho,\lambda)+ c_{\mathrm{f}_{2,3}}(\lambda)\psi_{\mathrm{f}_2}(\rho,\lambda)\\
\psi_{\mathrm{f}_4}(\rho,\lambda) &= c_{\mathrm{f}_{1,4}}(\lambda)\psi_{\mathrm{f}_1}(\rho,\lambda)+ c_{\mathrm{f}_{2,4}}(\lambda)\psi_{\mathrm{f}_2}(\rho,\lambda)
\end{align*}
with
\begin{align*}
	c_{\mathrm{f}_{1,3}}(\lambda)=&=\frac{W(h_1(.,\lambda),b_2(.,\lambda))(\rho_\lambda)}{i(1-\lambda)}+\O(\langle\omega\rangle^{-1})
\\
c_{\mathrm{f}_{2,3}}(\lambda)&= -\frac{W(h_1(.,\lambda),b_1(.,\lambda))(\rho_\lambda)}{i(1-\lambda)}+\O(\langle\omega\rangle^{-1})
\end{align*}
and
\begin{align*}
	c_{\mathrm{f}_{1,4}}(\lambda)&= \frac{W(h_2(.,\lambda),b_2(.,\lambda))(\rho_\lambda)}{i(1-\lambda)}+\O(\langle\omega\rangle^{-1})
	\\
	c_{\mathrm{f}_{2,4}}(\lambda)&=-\frac{W(h_2(.,\lambda),b_1(.,\lambda))(\rho_\lambda)}{i(1-\lambda)}+\O(\langle\omega\rangle^{-1}).
\end{align*}
\end{lem}                                                       
Next, let $\chi: [0,1]\times \{z\in \C:-\frac{1}{2}\leq \Re z\leq \frac{3}{4}\} \to
[0,1]$, $\chi_\lambda(\rho):=\chi(\rho,\lambda)$, be a smooth cutoff function that satisfies
$\chi_\lambda(\rho)=1$ for $\rho \in [0,\rho_\lambda]$, $
\chi_\lambda(\rho)=0$ for $\rho \in [\widehat{\rho}_\lambda,1] $, and
$|\partial_\rho^k\partial_\omega^\ell \chi_\lambda(\rho)|\leq
C_{k,\ell}\langle\omega\rangle^{k-\ell}$ for $k,\ell\in\mathbb N_0$. 
We then define two solutions of Eq.~\eqref{eq:nofirstorder} as
\begin{align*}
v_1(\rho,\lambda):=&\chi_\lambda(\rho)[c_{1,4}(\lambda)\psi_1(\rho,\lambda)+c_{2,4}(\lambda)\psi_2(\rho,\lambda)]
+\left(1-\chi_\lambda(\rho)\right)\psi_4(\rho,\lambda)\\
v_2(\rho,\lambda):=& \chi_\lambda(\rho)[c_{1,3}(\lambda)\psi_1(\rho,\lambda)+c_{2,3}(\lambda)\psi_2(\rho,\lambda)]
+\left(1-\chi_\lambda(\rho)\right)\psi_3(\rho,\lambda)
\end{align*}
and note that an evaluation at $\rho=1$ yields 
$$
W(v_1(.,\lambda),v_2(.,\lambda))=W(\psi_4(.,\lambda),\psi_3(.,\lambda))=2(1-\lambda).
$$
With this remark we return to the full equation~\eqref{eq:generaleq}.
\section{Resolvent construction}
We now return to our specific potential $V(\rho)=-\frac{16}{(1+\rho^2)^2}$.
Setting
$u_j(\rho,\lambda)=\rho^{-2}(1-\rho^2)^{-\frac{\lambda}{2}}v_j(\rho,\lambda)$
for $j\in \{1,2\}$ yields two solutions to Eq.~\eqref{eq:generaleq} with $F_\lambda=0$. 
\begin{lem}\label{lem:asymptotic}
The solutions $u_1$ and $u_2$ are of the form
\begin{align}
u_1(\rho,\lambda)&=
\rho^{-2}(1-\rho^2)^{-\frac{\lambda}{2}}h_2(\rho,\lambda)\left[1+(1-\rho)\O(\langle\omega\rangle^{-1})+\O(\rho^{0}(1-\rho)^2\langle\omega\rangle^{-1})\right]
\nonumber\\
&=
\rho^{-2}(1+\rho)^{1-\lambda}\left[1+(1-\rho)\O(\langle\omega\rangle^{-1})+\O(\rho^{-1}(1-\rho)^2\langle\omega\rangle^{-1})\right]\nonumber
\\
u_2(\rho,\lambda)&=\rho^{-2}(1-\rho^2)^{-\frac{\lambda}{2}}h_1(\rho,\lambda)\left[1+(1-\rho)\O(\langle\omega\rangle^{-1})+\O(\rho^{0}(1-\rho)^2\langle\omega\rangle^{-1})\right]\nonumber
\\
&= 
\rho^{-2}(1-\rho)^{1-\lambda}\left[1+(1-\rho)\O(\langle\omega\rangle^{-1})+\O(\rho^{-1}(1-\rho)^2\langle\omega\rangle^{-1})\right]\nonumber
\end{align}
for all $\rho \geq
\widehat{\rho}_\lambda=\min\{\frac{1}{2}(\rho_0+1),\frac{2r}{|a(\lambda)|}\}$
and all $\lambda\in \C\setminus\{0\}$ with $-\frac{3}{4}\leq \Re\lambda \leq \frac{3}{4}$.
Moreover, we have $u_1(.,\lambda)\in C^\infty((0,1])$ for all such values of $\lambda$.
\end{lem}
\begin{proof}
The explicit forms of $u_1$ and $u_2$ follow immediately from our ODE
construction. To see that $u_1$ is smooth away from $\rho=0$, we first
remark that clearly $u_1(.,\lambda)\in C^\infty(0,1)$. Furthermore,
the Frobenius indices of Eq.~\eqref{eq:spectraleq}  at $\rho=1 $ are
$\{0,1-\lambda\}$. Hence, there exist coefficients $c_1(\lambda)$ and
$c_2(\lambda)$ such that
$c_1(\lambda)u_1(.,\lambda)+c_2(\lambda) u_2(.,\lambda)$ is
nontrivial and smooth on $(0,1]$.
However, since $\Re\lambda\geq -\frac{1}{2}$, we clearly have that $u_1(.,\lambda)\in C^2((0,1])$ while $u_2(.,\lambda)\notin C^2((0,1])$ and so $c_2(\lambda)=0$.
\end{proof}

\subsection{Considerations on the point spectrum}
We aim to establish $H^{\frac32}\times H^{\frac{1}{2}}$-type  Strichartz estimates by proving bounds of the form
\begin{align*}
\|[e^{\frac{\tau}{2}}\Sf(\tau)(\I-\Pf)\ff]_1\|_{L^p_\tau L^q(\B^5_1)}\lesssim \|\ff\|_{H^2\times H^1(\B^5_1)}
\end{align*}
and 
\begin{align*}
\|[e^{-\frac{\tau}{2}}\Sf(\tau)(\I-\Pf )\ff]_1\|_{L^p_\tau L^q(\B^5_1)}\lesssim \|\ff\|_{H^1\times L^2(\B^5_1)}
\end{align*}
and interpolating between these two. An obvious obstruction to the
first estimate is the existence of eigenvalues $\lambda$ of $\Lf$ with
$-\frac12\leq \Re\lambda<0$. Unfortunately, we cannot rigorously rule out
such
eigenvalues, even though they are not expected to exist (see
\cite{Biz05, DonAic10} for numerical evidence). To circumvent this, we recall Lemmas \ref{lem:spec} and \ref{lem:spec2} which tell us that 
$$
\sigma_u(\Lf):=\sigma(\Lf)\cap \{\lambda\in \C:\Re \lambda >-\frac12\}=\{\lambda_1,\dots,\lambda_n,1\}
$$ 
with $n\in\mathbb N_0$ and where $\Re \lambda_i<0$ for
$i=1,\dots,n$. Moreover, each element of $\sigma_u(\Lf)$ is an
eigenvalue of finite algebraic multiplicity. This alone does of course
still not settle our spectral problem, but the following general property of finite rank operators will enable us to deal with these eigenvalues. 
\begin{lem}\label{lem:proj}
Let $H$ be a Hilbert space. Then, for any densely defined
operator $T:D(T)\subset H \to H$ with finite rank, there exists a
dense subset $X \subset H$ with $X\subset D(T)$ and a bounded linear operator $\widehat{T}:H \to H$ such that
\begin{equation*}
T|_X=\widehat{T}|_X.
\end{equation*}
\end{lem}
\begin{proof}
If $T$ is bounded, we choose $X=H$ and $\widehat T$ the unique
extension of $T$ to all of $H$. Consequently, we may assume that $T$ is
not bounded. 
We prove the result by induction on the rank of $T$. Let $\dim \rg
T=1$. Then we have $Tx=\varphi(x)x_0$ for a suitable $x_0\in H$ and a
linear functional $\varphi: D(T)\subset H\to \C$ that is not bounded. Thus, we find a sequence $(\widetilde y_n)_{n\in\mathbb N}\subset D(T)$ with
$\|\widetilde y_n\|_H\leq 1$ and $|\varphi(\widetilde y_n)|\geq n$ for all $n\in\mathbb
N$. We set $y_n:=\frac{\widetilde y_n}{\varphi(\widetilde y_n)}$. Then
we have $\|y_n\|_H\leq \frac{1}{n}$ and $\varphi(y_n)=1$.
Now let $x\in D(T)$ be arbitrary and set $x_n:=x-\varphi(x)y_n$. Then
we have $\varphi(x_n)=\varphi(x)-\varphi(x)\varphi(y_n)=0$ and thus,
$x_n \in \ker T$. Furthermore, $\|x-x_n\|_H=|\varphi(x)|\|y_n\|\leq
\frac{|\varphi(x)|}{n}$ and thus, $x_n\to x$ as
$n\to\infty$. Consequently, $\ker T$ is dense in $H$.

Assume now that the claim has been established for operators with rank
$n$. Let $T$ be a densely defined operator with rank $n+1$ that is not bounded
and let $\{e_1,\dots, e_{n+1}\}$ be an orthonormal basis of $\rg
T$. Further, denote by $P_j:\rg T \to \Span\{e_j\}$ the associated
orthonormal projections. Then $T=\sum_{j=1}^{n+1} P_jT$ and thus, at
least one of the operators $T_j:=P_j T$ cannot be bounded. After
relabeling we can assume that $T_1$ is not bounded. Since $T_1$ has rank
$1$, we see by the above that $\ker T_1$ is dense in $H$. Consider now
the restriction of $T$ to the kernel of $T_1$. Then this is either a
bounded operator and we are done or it is not bounded and has rank $n$, and we can use the induction hypothesis to conclude as well.
\end{proof}
Recall now that $\Qf: H^2\times H^1(\B^5_1) \to H^2\times H^1(\B^5_1)$
denotes the Riesz projection associated to the set $\sigma_u(\Lf) \setminus \{1\}$. 
We can of course also view $\Qf$ as a potentially unbounded operator
on $H^{\frac32}\times H^{\frac12}(\B^5_1)$ with domain
$D(\Qf)=H^2\times H^1(\B^5_1)$. Then Lemma \ref{lem:proj} applies and
we obtain a dense subset $X\subset H^{\frac32}\times
H^{\frac12}(\B^5_1)$ together with a bounded linear operator $\widehat{\Qf}$ on $H^{\frac32}\times H^{\frac12}(\B^5_1)$ which agrees with $\Qf$ on $X$.
Concretely, we have the following lemma.
\begin{lem}
There exists a dense subset $X$ in $H^{\frac32}\times
H^{\frac12}(\B^5_1)$ with $X\subset H^2\times H^1(\B^5_1)$ and a
bounded linear operator
$\widehat{\Qf}$: $H^{\frac32}\times H^{\frac12}(\B^5_1) \to H^{\frac32}\times H^{\frac12}(\B^5_1)$ such that $$
\widehat{\Qf}|_X=\Qf|_X.
$$
\end{lem}
Unfortunately, this does not help us with eigenvalues that lie one the line $\Re z=-\tfrac12$. We can however pick a $1>>\delta>0$ such that $\sigma_u \cap \{z\in C: -\frac{1-\delta}{2}\leq \Re z\leq -\frac{1-3\delta}{2}\} =\emptyset $ and aim to prove estimates of the form 
\begin{align*}
\|[e^{(\tfrac{1}{2}-\delta)\tau}\Sf(\tau)(\I-\Qf)(\I-\Pf)\ff]_1\|_{L^p_\tau L^q(\B^5_1)}\lesssim \|(\I-\Qf)\ff\|_{W^{2,\frac{2}{1+2\delta}}\times W^{1,\frac{2}{1+2\delta}}(\B^5_1)}
\end{align*}
and 
\begin{align*}
\|[e^{-(\tfrac{1}{2}-\delta)\tau}\Sf(\tau)(\I-\Qf)(\I-\Pf )\ff]_1\|_{L^p_\tau L^q(\B^5_1)}\lesssim \|(\I-\Qf)\ff\|_{W^{1,\frac{2}{1-2\delta}}\times L^{\frac{2}{1-2\delta}}(\B^5_1)}.
\end{align*} 
For this the following classification will be vital for us.
\begin{lem}\label{lem:classeigen}
Any point $\lambda\in \mathbb{C}\setminus\{0\}$ with $ -\frac{1}{2}\leq\Re\lambda\leq \frac{3}{4}$ is an eigenvalue of $\Lf$ if and only if $c_{2,4}(\lambda)=0$.
\end{lem}

\begin{proof}
  Assume that $c_{2,4}(\lambda)=0$. Then, as the Frobenius indices of
  Eq.~\eqref{eq:spectraleq} are $\{0,-3\}$ at $\rho=0$ and
  $\{0,1-\lambda\}$ at $\rho=1,$ one readily checks
  $u_1(.,\lambda)\in H^2(\B^5_1)$ and from equation \eqref{eq:u2} we
  see that the associated vector valued function $\uf_1$ satisfies
  $\uf_1(.,\lambda)\in H^2\times H^1(\B^5_1)$. Thus,
  $\lambda\in \sigma_p(\Lf)$. Conversely, let
  $\lambda\in \C\setminus\{0\}$ be an eigenvalue of $\Lf$ with
  $\Re\lambda\in [-\frac{1}{2},\frac34]$ and let $\ff$ be an
  eigenfunction. Now, the first component of any eigenfunction has to
  be a linear combination of $u_1$ and $u_2$. Since
  $u_2\notin H^2((\frac{1}{2},1))$ and $u_1\in H^2((\frac{1}{2},1))$,
  $f_1$ has to be a multiple of
  $u_1$. However, given that
  $|.|^{-2}\psi_1(.,\lambda)\in L^2(\B^5_{\rho_\lambda})$ while
  $\rho^{-2}\psi_2(\rho,\lambda)\simeq \rho^{-3}$ as $\rho\to 0$, we
  see that $u_1(.,\lambda)\in H^2(\B^5_1)$ is only possible if
  $c_{2,4}(\lambda)=0$.
\end{proof}
\subsection{The reduced resolvent}
Lemma \ref{lem:classeigen} enables us to construct a third solution to
 Eq.~\eqref{eq:generaleq} with $F_\lambda=0$, whenever $\Re(\lambda)\in [-\frac{1}{2},\frac34]$ and $\lambda\notin \sigma(\Lf)\cup \{0\}$, by setting
 \[ u_0(\rho,\lambda):=u_2(\rho,\lambda)-\frac{c_{2,3}(\lambda)}{c_{2,4}(\lambda)}u_1(\rho,\lambda). \] 
Note that
$$
W(u_1(.,\lambda),u_0(.,\lambda))(\rho)=2(1-\lambda)\rho^{-4}(1-\rho^2)^{-\lambda}.
$$
So, to solve Eq.~\eqref{eq:generaleq} when $\Re\lambda \in (0,\frac{3}{4}]$, we make the ansatz
\begin{align}\label{eq: ansatz1}
u(\rho,\lambda)&=-u_0(\rho,\lambda) \int_\rho^{1}\frac{u_1(s,\lambda)}{W(u_1(.,\lambda),u_0(.,\lambda))(s)}\frac{F_\lambda(s)}{1-s^2} d s \nonumber
\\
&\quad-u_1(\rho,\lambda) \int_0^\rho\frac{u_0(s,\lambda)}{W(u_1(.,\lambda),u_0(.,\lambda))(s)}\frac{F_\lambda(s)}{1-s^2} d s
\\
&=-\frac{u_0(\rho,\lambda)}{2(1-\lambda)} \int_\rho^{1}\frac{s^4u_1(s,\lambda)F_\lambda(s)}{(1-s^2)^{1-\lambda}} d s-\frac{u_1(\rho,\lambda)}{2(1-\lambda)} \int_0^\rho\frac{s^4u_0(s,\lambda)F_\lambda(s)}{(1-s^2)^{1-\lambda}} d s \nonumber
\end{align}
and one can check that $u(.,\lambda)\in H^2(\B^5_1)$.
However, for $\Re\lambda \leq 0$ we need some more considerations stemming from the simple fact that for any $F_\lambda$ with $F_\lambda(1)\neq 0$, one has  \begin{align*}
 \int_\rho^{1}\frac{s^4u_1(s,\lambda)F_\lambda(s)}{(1-s^2)^{1-\lambda}} d s =\infty
\end{align*} 
for all $\rho\in (0,1)$.
To remedy this, we slightly modify our ansatz. Thus, for $F_\lambda \in C^\infty(\overline{\B^5_1})$, let $\rho, \rho_1\in (0,1), c\in \C$, and set 
\begin{align*}
u(\rho,\lambda)&=c u_0(\rho,\lambda)-\frac{u_0(\rho,\lambda)}{2(1-\lambda)} \int_\rho^{\rho_1}\frac{s^4u_1(s,\lambda)F_\lambda(s)}{(1-s^2)^{1-\lambda}} d s-\frac{u_1(\rho,\lambda)}{2(1-\lambda)} \int_0^\rho\frac{s^4u_0(s,\lambda)F_\lambda(s)}{(1-s^2)^{1-\lambda}} d s
\end{align*}
as well as $$U_j(\rho,\lambda)=\frac{1}{2(1-\lambda)}\int_0^\rho\frac{s^4u_j(s,\lambda)}{(1-s^2)^{1-\lambda}} d s$$
for $j=0,1,2$.
Integrating by parts yields
\begin{align*}
u(\rho,\lambda)&= u_0(\rho,\lambda)\left[c+U_1(\rho,\lambda)F_\lambda(\rho)-U_1(\rho_1,\lambda)F_\lambda(\rho_1)+\int_\rho^{\rho_1}U_1(s,\lambda)F_\lambda'(s) ds \right]
\\
&\quad-u_1(\rho,\lambda) U_0(\rho,\lambda)F_\lambda(\rho)+u_1(\rho,\lambda)\int_0^\rho U_0(s,\lambda)F'_\lambda(s)d s
\end{align*}
which, upon setting $c=\widetilde{c}+U_1(\rho_1,\lambda)F_\lambda(\rho_1)$ with $\widetilde{c}\in \C$, reduces to 
\begin{align*}
&u_0(\rho,\lambda)\left[\widetilde{c}+U_1(\rho,\lambda)F_\lambda(\rho)+\int_\rho^{\rho_1}U_1(s,\lambda)F_\lambda'(s) ds \right]
\\
&-u_1(\rho,\lambda) U_0(\rho,\lambda)F_\lambda(\rho)+u_1(\rho,\lambda)\int_0^\rho U_0(s,\lambda)F'_\lambda(s)d s
\end{align*}
and also allows us to safely take the limit $\rho_1\to 1$.

\begin{lem}\label{lem:reslambda<0}
Let $f\in C^\infty(\overline{\B^5_1})$ and $\Re\lambda\in [-\frac{1}{2},\frac{3}{4}]$ with $\lambda\notin \sigma_p(\Lf) \cup \{0\}$. Then the unique solution $u(.,\lambda)\in H^2(\B^5_1)$ of the equation
\begin{align*}
(1-\rho^2) u''(\rho)+\left(\frac{4}{\rho}-2(\lambda+2)\rho\right)u'(\rho)-(\lambda+1)(\lambda+2)u(\rho)+\frac{16}{(1+\rho^2)^2}u(\rho)=-f(\rho)
\end{align*} 
with $\rho\in(0,1)$ is given by
\begin{align*}
\mathcal{R}(f)(\rho,\lambda):=&u_0(\rho,\lambda)[b_\lambda(f)+U_1(\rho,\lambda)f(\rho)]+u_0(\rho,\lambda)\int_\rho^1U_1(s,\lambda)f'(s) ds 
\\
&-u_1(\rho,\lambda) U_0(\rho,\lambda)f(\rho)+u_1(\rho,\lambda)\int_0^\rho U_0(s,\lambda)f'(s)d s
\end{align*}
with
\begin{align*}
b_\lambda(f):=-\frac{ f(1)}{2\lambda(1-\lambda)}\int_0^1\partial_s[s^4u_1(s,\lambda)(1+s)^{-1+\lambda}](1-s)^\lambda ds.
\end{align*}
\end{lem}
\begin{proof}
As the Frobenius indices of Eq.~\eqref{eq:spectraleq} are given by $\{0,-3\}$ at $\rho=0$ and $\{0,1-\lambda\}$ at $\rho=1,$ we see that $u_0(.,\lambda)\in H^2(\B^5_{\frac{1}{2}})\cap C^2(\B^5_1)$, while $u_1(.,\lambda)$ is smooth for $0<\rho \leq 1$. Therefore, one easily verifies that $\mathcal{R}(f)(.,\lambda)\in H^2(\B^5_{\frac12})$. To study the behavior of $\mathcal{R}(f)(\rho,\lambda)$ at $\rho=1$, we rewrite the boundary terms as
\begin{align*}
u_2(\rho,\lambda)[b_\lambda(f)+U_1(\rho,\lambda)f(\rho)]-u_1(\rho,\lambda)U_2(\rho,\lambda)f(\rho)- b_\lambda(f)\frac{c_{2,3}(\lambda)}{c_{2,4}(\lambda)} u_1(\rho,\lambda).
\end{align*}
Observe now that $U_2(.,\lambda)$ and hence $u_1(.,\lambda)U_2(.,\lambda)f$
is twice continuously
differentiable at $\rho=1$. Further, since $u_1(.,\lambda)\in C^2((0,1]),$ the only remaining boundary term we have to check is
$$u_2(\rho,\lambda)[b_\lambda(f)+U_1(\rho,\lambda)f(\rho)].$$
For this we integrate by parts once more to infer that
\begin{align}\label{eq:U1}
U_1(\rho,\lambda)=& -\frac{u_1(\rho,\lambda)\rho^4(1-\rho)^{\lambda}(1+\rho)^{-1+\lambda}}{2\lambda(1-\lambda)}
\\
&\quad+\frac{1}{2\lambda(1-\lambda)}\int_0^\rho
     \partial_s[s^4u_1(s,\lambda)(1+s)^{-1+\lambda}](1-s)^\lambda ds. \nonumber
\end{align}
Then,
\begin{align}\label{eq:nonsmoothbound}
b_\lambda(f)+U_1(\rho,\lambda)f(\rho)=& \frac{1}{2\lambda(1-\lambda)}\bigg[[f(\rho)-f(1)]\int_0^1\partial_s[s^4u_1(s,\lambda)(1+s)^{-1+\lambda}](1-s)^\lambda ds \nonumber
\\
&\quad-u_1(\rho,\lambda)\rho^4(1-\rho)^{\lambda}(1+\rho)^{-1+\lambda}f(\rho)
\\
&\quad-f(\rho)\int_\rho^1 \partial_s[s^4u_1(s,\lambda)(1+s)^{-1+\lambda}](1-s)^\lambda ds\bigg]\nonumber
\end{align}
and by using this form, one readily checks that $u_2(\rho,\lambda)[b_\lambda(f)+U_1(\rho,\lambda)f(\rho)]$
belongs to $C^2((0,1])$.
We turn to the integral terms which we rewrite as
\begin{align*}
&u_2(\rho,\lambda)\int_\rho^1U_1(s,\lambda)f'(s) ds
                 +u_1(\rho,\lambda)\int_0^\rho U_2(s,\lambda)f'(s)  ds
\\
&-\frac{c_{2,3}(\lambda)}{c_{2,4}(\lambda)} u_1(\rho,\lambda)\int_0^1U_1(s,\lambda)f'(s) ds.
\end{align*}
In this form one can promptly verify by scaling that all these terms are elements of $C^2((0,1])$ as well.
\end{proof}
For $\lambda>0$ we can use the ansatz \eqref{eq: ansatz1} to recast $\mathcal{R}(f)$ in a simpler form.
\begin{lem}\label{lem:reslambda>0}
Let $f\in C^\infty(\overline{\B^5_1})$ and $\Re\lambda\in (0,\frac{3}{4}]$. Then $\mathcal{R}(f)$ satisfies
\begin{align}\label{eq:reslambda>0}
\mathcal{R}(f)(\rho,\lambda)=-\frac{u_0(\rho,\lambda)}{2(1-\lambda)} \int_\rho^{1}\frac{s^4u_1(s,\lambda)f(s)}{(1-s^2)^{1-\lambda}} d s-\frac{u_1(\rho,\lambda)}{2(1-\lambda)} \int_0^\rho\frac{s^4u_0(s,\lambda)f(s)}{(1-s^2)^{1-\lambda}} d s
\end{align}
for all $\rho \in (0,1)$.
\end{lem}
\begin{proof}
This can either be seen by directly undoing the integrations by parts in the construction of $\mathcal{R}(f)$ or by noting that both $\mathcal{R}(f)$ and
$$
\widetilde{\mathcal{R}}(f)(\rho,\lambda):=-\frac{u_0(\rho,\lambda)}{2(1-\lambda)} \int_\rho^{1}\frac{s^4u_1(s,\lambda)f(s)}{(1-s^2)^{1-\lambda}} d s-\frac{u_1(\rho,\lambda)}{2(1-\lambda)} \int_0^\rho\frac{s^4u_0(s,\lambda)f(s)}{(1-s^2)^{1-\lambda}} d s
$$
solve Eq.~\eqref{eq:generaleq} and are elements of $H^2(\B^5_1)$ (that $\widetilde{\mathcal{R}}(f)\in H^2(\B^5_1)$ for $\Re \lambda>0$ follows in the same manner as $\mathcal{R}(f)\in H^2(\B^5_1))$. Given that both of these functions solve Eq.~\eqref{eq:generaleq} their difference has to be a linear combination of $u_1$ and $u_0$. However, as $u_j\notin H^2(\B^5_1)$ for $j=1,2$, we see that $\mathcal{R}(f) $ and $\widetilde{\mathcal{R}}(f)$ have to coincide.
\end{proof}

Having constructed a suitable solution to Eq.~\eqref{eq:generaleq}, we
remark that we can copy the same construction in the ``free'' case
$V=0$. This follows from the fact that $\Lf_0$ generates a semigroup
on $H^2\times H^1(\B^5_1)$ which satisfies the growth bound
$\|\Sf_0(\tau)\|_{H^2\times H^1(\B^5_1)}\lesssim e^{-\frac{\tau}{2}}$.  We
denote the corresponding free solutions by
$\mathcal{R}_{\mathrm{f}}(f)$. For
$\ff \in C^\infty\times C^\infty(\overline{\B^5_1})$ we
set $$\widetilde{\ff}=(\I-\Qf)(\I-\Pf) \ff$$ and use Laplace inversion
to explicitly write down $[\Sf(\tau)(\I-\Qf)(\I-\Pf) \ff]_1$ for any
such $\ff$ and
$\kappa\in (-\frac{1}{2}, \frac{3}{4}]$ as
\begin{align}\label{eq:semirep}
[\Sf(\tau) \widetilde{\ff}]_1(\rho)=[\Sf_0(\tau) \widetilde{\ff}]_1(\rho)
+\frac{1}{2\pi i}\lim_{N \to \infty} \int_{\kappa-i N}^{\kappa+ i
  N}e^{\lambda\tau}
  [\mathbf R_{\mathbf L}(\lambda)\widetilde
{\mathbf f}-\mathbf R_{\mathbf L_0}(\lambda)\widetilde
{\mathbf f}]_1d\lambda.
\end{align}

\section{A first set of Strichartz estimates}
Using Eq.~\eqref{eq:semirep} we can obtain establish the desired Strichartz estimates on $\Sf$ by bounding the integral term
\begin{align}\label{eq:integralterm}
\lim_{N \to \infty} \int_{\kappa-i N}^{\kappa+ i N}e^{\lambda\tau}[\mathcal{R}(F_\lambda)(\rho,\lambda)-\mathcal{R}_{\mathrm{f}}(F_\lambda)(\rho,\lambda)] d\lambda.
\end{align}
To accomplish this we will need some preliminary Lemmas
\subsection{Preliminary and technical lemmas}
The first set of lemmas will be concerned with oscillatory integrals.
\begin{lem}\label{osci1}
	Let $\alpha >0$. Then 
	$$
	\left|\int_\R e^{i \omega a}\O(\langle\omega\rangle^{-(1+\alpha)}) d \omega \right|\lesssim \langle a\rangle^{-2},
	$$
	 for any $a\in \R$. 
\end{lem}

\begin{proof}
	Since the integral is absolutely convergent the claim follows by two integrations by parts.
\end{proof}
\begin{lem}\label{osci2}
	Let $\alpha \in (0,1)$. Then
	\begin{align*}
	\left|\int_\R e^{i \omega a}\O(\langle\omega\rangle^{-\alpha}) d \omega \right|\lesssim |a|^{\alpha-1}\langle a\rangle^{-2}
	\end{align*}
	holds for $a\in \R\setminus\{0\}$.
\end{lem}
\begin{proof}
	See Lemma 4.2 in \cite{DonRao20}.
\end{proof}
\begin{lem}\label{osci3}
We have
	\begin{align*}
	\left|\int_\R e^{i \omega a}(1-\chi_\lambda(\rho))\O\left(\rho^{-n}\langle\omega\rangle^{-(n+1)}\right) d \omega \right|\lesssim \langle a\rangle^{-2}
	\end{align*}
	for all $n\geq 1$, $\rho \in (0,1)$, and $a\in \R$.
\end{lem}
\begin{proof}
	This can be proven in the same manner as Lemma 4.3 in \cite{DonRao20}
\end{proof}

\begin{lem}\label{osci4}
We have
	\begin{align*}
	\left|\int_\R e^{i \omega a}(1-\chi_\lambda(\rho))\O\left(\rho^{-n}\langle\omega\rangle^{-n}\right)d\omega\right|\lesssim|a|^{-1}\langle a\rangle^{-2}
	\end{align*}
	for any $n \geq 2$, $\rho \in (0,1)$, and $a\in \R \setminus\{0\}$.
\end{lem}
\begin{proof}
This can be proven as Lemma 4.4 in \cite{DonRao20}.
\end{proof}
Finally, by interpolating between Lemma \ref{osci3} and \ref{osci4} one obtains the following result.
\begin{lem}\label{osci5}
We have
	\begin{align*}
	\left|\int_\R e^{i \omega a}(1-\chi_\lambda(\rho))\O\left(\rho^{-n}\langle\omega\rangle^{-n}\right)d\omega\right|\lesssim\rho^{-\theta}|a|^{-(1-\theta)}\langle a\rangle^{-2}
	\end{align*}
	for any $n \geq 2$, $\rho \in (0,1)$, $\theta\in [0,1]$, and $a\in \R \setminus\{0\}$.
\end{lem}

We will also rely on the following estimate.
\begin{lem}\label{teclem4}
Let $\alpha\in (0,1)$ and $\beta \in [0,1)$. Then we have the estimate
	\begin{align*}
	\int_0^1 s^{-\beta}|a+\log(1\pm s)|^{-\alpha} ds\lesssim |a|^{-\alpha}
	\end{align*}
	for all $a\in \R\setminus\{0\}$.
\end{lem}
\begin{proof}
	We only prove the - case as the + case can be shown analogously. 
For $a<0$ the estimate
\begin{align*}
\left|a+ \log(1-s) \right|^{-\alpha} \leq |a|^{-\alpha}
\end{align*}  
holds for all $s\in [0,1]$ and so the claim follows.
For $a>0$  we change variables according to $s=1-e^{ax}$ and compute
	\begin{align*}
	\int_0^1s^{-\beta}|a+\log(1-s)|^{-\alpha} ds &= \int_{-\infty}^0(1-e^{a x})^{-\beta}|a+ax|^{-\alpha} a e^{a x} dx
	\\
	&\lesssim |a|^{1-\alpha}\int_{-\frac{1}{2}}^0(1-e^{a x})^{-\beta}e^{ax} dx 
	\\
	&\quad +|a|^{1-\alpha}(1-e^{-\frac a2})^{-\beta}e^{-\frac a2}\int_{-2}^{-\frac{1}{2}} |1+x|^{-\alpha}dx 
	\\
	&\quad +|a|^{1-\alpha}\int^{-2}_{-\infty}(1-e^{a x})^{-\beta}e^{a x} dx.
	\end{align*}
	The claimed estimate is now an immediate consequence of the two identities
\[
\partial_x \frac{(1-e^{ax})^{1-\beta}}{a(1-\beta)}=-(1-e^{ax})^{-\beta}e^{ax} 
\]
and
\[(1-e^{-\frac{a}{2}})^{-\beta}e^{-\frac{a}{2}} \lesssim a^{-\beta}.
\]
\end{proof}
Similarly, one can show the next technical Lemma 
\begin{lem} \label{teclem5}
Let $\alpha\in (0,1)$ and $\beta \in [0,1)$. Then the estimate
	\begin{align*}
	\int_0^1 s^{-\beta}\left|a\pm \frac{1}{2}\log(1-s^2) \right|^{-\alpha} ds\lesssim |a|^{-\alpha}
	\end{align*}
	holds for all $a\in \R\setminus\{0\}$.
\end{lem}
Lastly, we will also require the following result on weighted norms.

\begin{lem}\label{teclem2}
The estimate
\begin{align*}
\||.| f\|_{L^{6}(\B^5_1)}\lesssim \|f\|_{H^1(\B^5_1)}
\end{align*}
holds for all $f\in C^\infty(\overline{\B^5_1})$. 
\end{lem}
\begin{proof}
This follows by a minor adaptation of the argument given in the proof of Lemma 4.8 in \cite{DonRao20}.
\end{proof}

\subsection{Kernel estimates}
We will now begin bounding the integral term
\eqref{eq:integralterm}. We start with the case $\kappa=\frac{1}{2}-\delta$.
Therefore, we suppose $f\in C^\infty(\overline{\B^5_1})$ and take a look at the difference 
$\mathcal{R}(f)-\mathcal{R}_{\mathrm{f}}(f)$.

\begin{lem}\label{lem:decomp1}
Let $\Re\lambda=\frac{1}{2}-\delta $ and $f\in C^\infty(\overline{\B^5_1})$.
Then, we can decompose $\mathcal{R}(f)-\mathcal{R}_{\mathrm{f}}(f)$
as 
\begin{align*}
\mathcal{R}(f)(\rho,\lambda)-\mathcal{R}_{\mathrm{f}}(f)(\rho,\lambda)=\sum_{j=1}^9
G_j(f)(\rho,\lambda)
\end{align*}
\begin{align*}
G_1(f)(\rho,\lambda)&=\rho^{-2}(1-\rho^2)^{ -\frac{\lambda}{2}}b_1(\rho,\lambda) \int_\rho^1\frac{s^2\chi_\lambda(s)[b_1(s,\lambda)\alpha_1(\rho,s,\lambda)+b_2(s,\lambda)\alpha_2(\rho,s,\lambda)]}{2(1-\lambda)(1-s^2)^{1-\frac{\lambda}{2}}}f(s) ds
\\
&\quad+\rho^{-2}(1-\rho^2)^{ -\frac{\lambda}{2}}b_1(\rho,\lambda)[1+\O(\rho^2)]\int_\rho^1\frac{s^2\chi_\lambda(s)\O(s\langle\omega\rangle^{-2}) }{2(1-\lambda)(1-s^2)^{1-\frac{\lambda}{2}}}f(s) ds
\\
G_2(f)(\rho,\lambda)&=\chi_\lambda(\rho)\rho^{-2}(1-\rho^2)^{ -\frac{\lambda}{2}}b_1(\rho,\lambda)\int_\rho^1\frac{s^2(1-\chi_\lambda(s))h_2(s,\lambda)\beta_1(\rho,s,\lambda)}{2(1-\lambda)(1-s^2)^{1-\frac{\lambda}{2}}}f(s) ds
\\
G_3(f)(\rho,\lambda)&=(1-\chi_\lambda(\rho))\rho^{-2}(1-\rho^2)^{ -\frac{\lambda}{2}}h_2(\rho,\lambda)\int_\rho^1\frac{s^2 h_2(s,\lambda)\gamma_1(\rho,s,\lambda)}{2(1-\lambda)(1-s^2)^{1-\frac{\lambda}{2}}}f(s) ds
\\
G_4(f)(\rho,\lambda)&=(1-\chi_\lambda(\rho))\rho^{-2}(1-\rho^2)^{ -\frac{\lambda}{2}}h_1(\rho,\lambda)\int_\rho^1\frac{s^2 h_2(s,\lambda)\gamma_2(\rho,s,\lambda)}{2(1-\lambda)(1-s^2)^{1-\frac{\lambda}{2}}}f(s) ds
\\
G_5(f)(\rho,\lambda)&=\chi_\lambda(\rho)\rho^{-2}(1-\rho^2)^{ -\frac{\lambda}{2}}b_1(\rho,\lambda)\int_0^\rho\frac{s^2b_1(s,\lambda)\alpha_1(s,\rho,\lambda)}{2(1-\lambda)(1-s^2)^{1-\frac{\lambda}{2}}}f(s) ds
\\
G_6(f)(\rho,\lambda)&=\chi_\lambda(\rho)\rho^{-2}(1-\rho^2)^{ -\frac{\lambda}{2}}b_2(\rho,\lambda)\int_0^\rho\frac{s^2b_1(s,\lambda)\alpha_2(s,\rho,\lambda)}{2(1-\lambda)(1-s^2)^{1-\frac{\lambda}{2}}}f(s) ds
\\
&\quad +\chi_\lambda(\rho)(1-\rho^2)^{-\frac{\lambda}{2}}\O(\rho^{-1}\langle\omega\rangle^{-2})\int_0^\rho\frac{s^2b_1(s,\lambda)[1+\O(s^2)]}{2(1-\lambda)(1-s^2)^{1-\frac{\lambda}{2}}}f(s) ds
\end{align*}
and
\begin{align*}
G_7(f)(\rho,\lambda)&=(1-\chi_\lambda(\rho))\rho^{-2}(1-\rho^2)^{ -\frac{\lambda}{2}}h_2(\rho,\lambda)\int_0^\rho\frac{s^2\chi_\lambda(s) b_1(s,\lambda)\beta_1(s,\rho,\lambda)}{2(1-\lambda)(1-s^2)^{1-\frac{\lambda}{2}}}f(s) ds
\\
G_8(f)(\rho,\lambda)&=(1-\chi_\lambda(\rho))\rho^{-2}(1-\rho^2)^{ -\frac{\lambda}{2}}h_2(\rho,\lambda)
\int_0^\rho\frac{s^2(1-\chi_\lambda(s))h_2(s,\lambda)\gamma_1(s,\rho,\lambda)}{2(1-\lambda)(1-s^2)^{1-\frac{\lambda}{2}}}f(s) ds
\\
G_9(f)(\rho,\lambda)&=(1-\chi_\lambda(\rho))\rho^{-2}(1-\rho^2)^{ -\frac{\lambda}{2}}h_2(\rho,\lambda)\int_0^\rho\frac{s^2(1-\chi_\lambda(s)) h_1(s,\lambda)\gamma_2(s,\rho,\lambda)}{2(1-\lambda)(1-s^2)^{1-\frac{\lambda}{2}}}f(s) ds
\end{align*}
where
\begin{align*}
\alpha_j(\rho,s,\lambda)&=\O(\langle\omega\rangle^{-1})+\O(\rho^2\langle\omega\rangle^0)+\O(s^2\langle\omega\rangle^0)+\O(\rho^2s^2\langle\omega\rangle^0)
\\
\beta_1(\rho,s,\lambda)&=\O(\langle\omega\rangle^{-1})+\O(\rho^2\langle\omega\rangle^0)+\O(s^0(1-s)\langle\omega
\rangle^{-1})+\O(\rho^2 s^0(1-s)\langle\omega\rangle^{-1})
\\
\gamma_j(\rho,s,\lambda)&=\O(\langle\omega\rangle^{-1})+\O(\rho^0(1-\rho)\langle\omega
\rangle^{-1})+\O(s^0(1-s)\langle\omega
\rangle^{-1})
\\
&\quad+\O(\rho^0(1-\rho)s^0(1-s)\langle\omega\rangle^{-2}).
\end{align*}
\end{lem}
\begin{proof}
This follows by plugging the definitions of the $u_j$ into \eqref{eq:reslambda>0} and a straightforward calculation
using estimates like 
$$
\psi_1(\rho,\lambda)-\psi_{\mathrm{f}_1}(\rho,\lambda)=b_1(\rho,\lambda)\O(\rho^2\langle\omega\rangle^0)
$$
and
\begin{align*}
c_{2,3}(\lambda)- c_{\mathrm{f}_{2,3}}(\lambda) =\O(\langle\omega\rangle^{-1}).
\end{align*}
\end{proof}
Next, we will recast the $G_j$ into a more controllable form.
\begin{lem}\label{lem:decomp1symbol}
The functions $G_j(f)$ satisfy
\begin{align*}
G_1(f)(\rho,\lambda)&=(1-\rho^2)^{ -\frac{\lambda}{2}}\int_\rho^1\frac{\chi_\lambda(s)\O(\rho^0s\langle\omega\rangle^{-1})}{(1-s^2)^{1-\frac{\lambda}{2}}}f(s) ds
\\
G_2(f)(\rho,\lambda)&=\chi_\lambda(\rho)(1-\rho^2)^{ -\frac{\lambda}{2}}\O(\rho^0\langle\omega\rangle^{2})
\\
&\quad\times\int_\rho^1\frac{s^2(1-\chi_\lambda(s))[1+\O(s^{-1}(1-s)\langle\omega\rangle^{-1})]\beta_1(\rho,s,\lambda)}{2(1-\lambda)(1-s)^{1-\lambda}}f(s) ds
\\
G_3(f)(\rho,\lambda)&=(1-\chi_\lambda(\rho))\rho^{-2}(1+\rho)^{1-\lambda}[1+\O(\rho^{-1}(1-\rho)\langle\omega\rangle^{-1})]
\\
&\quad\times\int_\rho^1\frac{s^2(1-\chi_\lambda(s))[1+\O(s^{-1}(1-s)\langle\omega\rangle^{-1})]\gamma_1(\rho,s,\lambda)}{2(1-\lambda)(1-s)^{1-\lambda}}f(s) ds
\\
G_4(f)(\rho,\lambda)&=(1-\chi_\lambda(\rho))\rho^{-2}(1-\rho)^{1-\lambda}[1+\O(\rho^{-1}(1-\rho)\langle\omega\rangle^{-1})]
\\
&\quad\times\int_\rho^1\frac{s^2(1-\chi_\lambda(s))[1+\O(s^{-1}(1-s)\langle\omega\rangle^{-1})]\gamma_2(\rho,s,\lambda)}{2(1-\lambda)(1-s)^{1-\lambda}}f(s) ds
\\
G_5(f)(\rho,\lambda)&=\chi_\lambda(\rho)(1-\rho^2)^{-\frac{\lambda}{2}}\int_0^\rho\frac{\O(\rho^{-1} s^2\langle\omega\rangle^{-1}) }{(1-s^2)^{1-\frac{\lambda}{2}}}f(s) ds
\\
G_6(f)(\rho,\lambda)&=\chi_\lambda(\rho)(1-\rho^2)^{-\frac{\lambda}{2}}\int_0^\rho\frac{\O(\rho^{-1} s^2\langle\omega\rangle^{-1}) }{(1-s^2)^{1-\frac{\lambda}{2}}}f(s) ds
\end{align*}
and
\begin{align*}
G_7(f)(\rho,\lambda)&=(1-\chi_\lambda(\rho))\rho^{-2}(1+\rho)^{1-\lambda}[1+\O(\rho^{-1}(1-\rho)\langle\omega\rangle^{-1})]\\
& \quad \times \int_0^\rho\frac{\chi_\lambda(s) \O(s^4\langle\omega\rangle)\beta_1(s,\rho,\lambda)}{(1-s^2)^{1-\frac{\lambda}{2}}}f(s) ds
\\
G_8(f)(\rho,\lambda)&=(1-\chi_\lambda(\rho))\rho^{-2}(1+\rho)^{1-\lambda}[1+\O(\rho^{-1}(1-\rho)\langle\omega\rangle^{-1})]
\\
&\quad\times\int_0^\rho\frac{s^2(1-\chi_\lambda(s))[1+\O(s^{-1}(1-s)\langle\omega\rangle^{-1})]\gamma_1(s,\rho,\lambda)}{2(1-\lambda)(1-s)^{1-\lambda}}f(s) ds
\\
G_9(f)(\rho,\lambda)&=(1-\chi_\lambda(\rho))\rho^{-2}(1+\rho)^{1-\lambda}[1+\O(\rho^{-1}(1-\rho)\langle\omega\rangle^{-1})]
\\
&\quad\times\int_0^\rho\frac{s^2(1-\chi_\lambda(s))[1+\O(s^{-1}(1-s)\langle\omega\rangle^{-1})]\gamma_2(s,\rho,\lambda)}{2(1-\lambda)(1+s)^{1-\lambda}}f(s) ds.
\end{align*}
\end{lem}
Motivated by this decomposition we define operators $T_j(\tau)f(\rho)$ for $j=1,\dots,9$ and $f\in C^\infty(\overline{\B^5_1})$ as
\begin{align*}
T_j(\tau)f(\rho):=\lim_{N\to \infty} \int_{-N}^N e^{i\omega\tau} G_j(f)(\rho,\tfrac{1}{2}-\delta+i\omega) d\omega .
\end{align*}
Given that the integrals above are absolutely convergent, which follows from Lemma \ref{lem:decomp1symbol}, one concludes that $T_j(\tau)f(\rho)$ is meaningful for all $j=1,\dots,9, \tau\in \R, \rho \in (0,1)$, and $f\in C^\infty(\overline{\B^5_1})$. In addition, we have the following estimates.
\begin{lem}\label{lem:Gbounds1}
The operators $T_j$ satisfy the estimates
\begin{align*}
 \|T_j(\tau)f\|_{L^{\frac{2}{1-\delta}}_\tau(\R_+)L^{\frac{45}{8}}(\B^5_1)}&\lesssim   \|f\|_{L^{\frac{2}{1-2\delta}}(\B^5_1)}
 \\
 \|T_j(\tau)f\|_{L^\infty _\tau(\R_+)L^{\frac{10}{3}}(\B^5_1)}&\lesssim
                                              \|f\|_{L^{\frac{2}{1-2\delta}}(\B^5_1)} \\
\end{align*}
for all $f\in C^\infty(\overline{\B^5_1})$ and $j=1,\dots,9$.
\end{lem}

\begin{proof}
We start with $T_1$ and use
\begin{equation}\label{Eq:interchange}
 \chi_\lambda(\rho)\O(\rho^0s\langle\omega\rangle^{-1}) =\chi_{\lambda}(\rho)\O(\rho^{-\frac{4}{5}}s^{\frac{7}{4}}\langle\omega\rangle^{-1-\frac{1}{20}}) 
 \end{equation}
 which holds for $0< \rho\leq s$.
This enables us to use dominated convergence and Fubini's Theorem to conclude that
\begin{align*}
T_1(\tau)f(\rho)&=\int_\rho^1\int_\R e^{i\omega\tau}(1-\rho^2)^{-\frac{1}{4}+\frac\delta2 -\frac{i\omega}{2}}\frac{\chi_{\frac{1}{2}-\delta+i \omega}(s)\O(\rho^{-\frac{4}{5}}s^{\frac{7}{4}}\langle\omega\rangle^{-1-\frac{1}{20}})}{(1-s^2)^{\frac{3}{4}+\frac{\delta}{2}-\frac{i \omega}{2}}}f(s)  d\omega ds
\\
&=\int_\rho^{1}\int_\R e^{i\omega\tau}(1-\rho^2)^{ -\frac{1}{4}+\frac \delta 2- \frac{i\omega}{2}}1_{(0,\rho_1)}(s)\frac{\chi_{\frac{1}{2}-\delta+i \omega}(s)\O(\rho^{-\frac{4}{5}}s^{\frac{7}{4}}\langle\omega\rangle^{-1-\frac{1}{20}})}{(1-s^2)^{\frac{3}{4}+\frac{\delta}{2}-\frac{i \omega}{2}}}f(s) d\omega ds
\end{align*}
for some $\rho_1<1$ and where $ 1_{(0,\rho_1)}$ is the characteristic function of the interval $ (0,\rho_1)$. Consequently, Lemma \ref{osci1} yields 
\begin{align*}
|T_1(\tau)f(\rho)|
&\lesssim
 \rho^{-\frac{4}{5}}\int_{0}^{\rho_1}\langle\tau-\tfrac12\log(1-\rho^2)+\tfrac12 \log(1-s^2) \rangle^{-2} s^{\frac{7}{4}}|f(s)|ds
 \\
 &\lesssim
 \rho^{-\frac{4}{5}}\langle\tau\rangle^{-2}\int_{0}^{1} s^{\frac{7}{4}}|f(s)|ds
 \\
 &\lesssim
 \langle\tau\rangle^{-2}\rho^{-\frac{4}{5}}\|f\|_{L^{2}(\B^5_1)}\||.|^{-\frac{1}{4}}\|_{L^{2}((0,1))}
 \\
  &\lesssim
 \langle\tau\rangle^{-2}\rho^{-\frac{4}{5}}\|f\|_{L^{\frac{2}{1-2\delta}}(\B^5_1)}.
\end{align*}
Thus, given that 
\begin{align*}
\||.|^{-\frac{4}{5}}\|_{L^{\frac{45}{8}}(\B^5_1)}=\left(\int_0^1
  \rho^{-\frac{1}{2}} d\rho\right)^{\frac{8}{45}} \lesssim 1,
\end{align*}
we see that
\begin{align*}
\|T_1(\tau)f\|_{L^{\frac{45}{8}}(\B^5_1)} \lesssim \langle\tau\rangle^{-2} \|f\|_{L^{\frac{2}{1-2\delta}}(\B^5_1)}
\end{align*}
and so the estimates on $T_1$ follow.

We move on to $T_2$, which, after interchanging the order of integration and using an estimate similar to \eqref{Eq:interchange},
takes the form
\begin{align*}
T_2(\tau)f(\rho)=&\int_\rho^1 \int_\R e^{i\omega\tau} \chi_{\frac{1}{2}-\delta+i \omega}(\rho)(1-\rho^2)^{ -\frac{1}{4}- \frac{i\omega}{2}}\O(\rho^{-\frac{5}{6}}\langle\omega\rangle^{\frac{1}{6}})
\\
&\quad\times\frac{s^2(1-\chi_{\frac{1}{2}-\delta+i \omega}(s))[1+\O(s^{-1}(1-s)\langle\omega\rangle^{-1})]\beta_1(\rho,s,\frac{1}{2}-\delta+i \omega)}{(1-s)^{\frac{1}{2}+\delta-i \omega}}f(s) d\omega ds
\end{align*}
and we can apply Lemma \ref{osci2} to infer that
\begin{align*}
|T_2(\tau)f(\rho)|&\lesssim \rho^{-\frac{5}{6}}\int_\rho^1
\langle\tau+\log(1-s)\rangle^{-2}|\tau-\tfrac{1}{2}\log(1-\rho^2)+\log(1-s)|^{-\frac{1}{6}}
\\
&\quad \times s^2 |f(s)| (1-s)^{-\frac{1}{2}-\delta} ds.
\end{align*}
Hence, Minkowski's inequality implies that
\begin{align*}
\|T_2(\tau)f\|_{L^{\frac{45}{8}}(\B^5_1)}&\lesssim \int_0^1 s^2 |f(s)| (1-s)^{-\frac{1}{2}-\delta} \langle\tau+\log(1-s)\rangle^{-2}
\\
&\quad\times\left(\int_0^1\rho^{-\frac{11}{16}}|\tau-\tfrac{1}{2}\log(1-\rho^2)+\log(1-s)|^{-\frac{15}{16  }} d\rho \right)^{\frac{8}{45}} ds
\end{align*}
and by employing Lemma \ref{teclem5} we obtain 
\begin{align*}
\|T_2(\tau)f\|_{L^{\frac{45}{8}}(\B^5_1)}&\lesssim \int_0^1 s^2 |f(s)| (1-s)^{-\frac{1}{2}-\delta} \langle\tau+\log(1-s)\rangle^{-2}
|\tau+\log(1-s)|^{-\frac{1}{6}}  ds.
\end{align*}
By now changing variables according to $s=1-e^{-y}$ and using Young's inequality, we compute
\begin{align*}
&\quad\|T_2(\tau)f\|_{L^{\frac{2}{1-\delta}}_\tau(\R_+)L^{\frac{45}{8}}(\B^5_1)}
\\
&\lesssim \left\|\int_0^1 s^2 |f(s)| (1-s)^{-\frac{1}{2}-\delta} \langle\tau+\log(1-s)\rangle^{-2}
|\tau+\log(1-s)|^{-\frac{1}{6}} ds\right\|_{L^{\frac{2}{1-\delta}}_\tau(\R_+) }
\\
&\lesssim \left\|\int_0^\infty (1-e^{-y})^2 |f(1-e^{-y})| e^{-(\frac{1}{2}-\delta)y} \langle\tau-y\rangle^{-2}
|\tau-y|^{-\frac{1}{6}}dy\right\|_{L^{\frac{2}{1-\delta}}_\tau(\R) }
\\
&\lesssim \|(1-e^{-y})^2 |f(1-e^{-y})| e^{-(\frac{1}{2}-\delta)y} \|_{L^{\frac{2}{1-2\delta}}_y(\R_+)}\|\langle.\rangle^{-2}
|.|^{-\frac{1}{6}}\|_{L^1(\R_+) }
\\
&\lesssim \|f\|_{L^{\frac{2}{1-2\delta}}(\B^5_1)}.
\end{align*}
As a consequence, the first of the desired estimates on $T_2$ follows.
Since the second can be obtained likewise we turn to $T_3$.
To bound $T_3$, we employ Lemma \ref{osci5} to deduce that
\begin{align*}
|T_3(\tau)f(\rho)|&\lesssim \rho^{-\frac{5}{6}}\int_\rho^1
\langle\tau+\log(1-s)\rangle^{-2}|\tau-\log(1+\rho)+\log(1-s)|^{-\frac{1}{6}}
\\
&\quad \times s^2 |f(s)| (1-s)^{-\frac{1}{2}-\delta} ds.
\end{align*}
So, Minkowski's inequality combined with an application of Lemma \ref{teclem4} yields
\begin{align*}
\|T_3(\tau)f\|_{L^{\frac{45}{8}}(\B^5_1)}\lesssim \int_0^1
\langle\tau+\log(1-s)\rangle^{-2}|\tau+\log(1-s)|^{-\frac{1}{6}}s^2 |f(s)| (1-s)^{-\frac{1}{2}-\delta} ds
\end{align*}
 and one can bound $T_3$ in the same manner as $T_2$.
 Further, since the estimates on the remaining operators can be established by analogous means, we conclude this proof.
\end{proof}
Unfortunately, the operators $T_j$ alone do not suffice to establish
the necessary estimates on the semigroup $\Sf$ since one of the terms in the definition of $F_\lambda$ consists of $(\lambda+2) f_1(\rho)$. To remedy this we define another set of operators $\dot{T}_j$ for $j=1,\dots,9$ and $f\in C^\infty(\overline{\B^5_1})$ by
\begin{align*}
\dot{T}_j(\tau)f(\rho):=\lim_{N\to \infty} \int_{-N}^N i\omega e^{i\omega\tau} G_j(f)(\rho,\tfrac{1}{2}-\delta+i\omega) d\omega .
\end{align*}
This additional power of $\omega$ spoils the absolute convergence of the integral and so, to see that $\dot{T}_j(\tau)f$ is a meaningful expression, one cannot argue as simple as for the operators $T_j$. However, the following lemma shows that the above defined operators $\dot{T}_j(\tau)$ exist as bounded linear operators from a dense subset of $W^{1,\frac{2}{1-2\delta}}(\B^5_1)$ into certain Strichartz spaces.
\begin{lem}\label{lem:Gbounds2}
The operators $\dot T_j$ satisfy the estimates
\begin{align*}
 \|\dot T_j(\tau)f\|_{L^{\frac{2}{1-\delta}}_\tau(\R_+)L^{\frac{45}{8}}(\B^5_1)}&\lesssim   \|f\|_{W^{1,\frac{2}{1-2\delta}}(\B^5_1)}
 \\
 \|\dot T_j(\tau)f\|_{L^\infty _\tau(\R_+)L^{\frac{10}{3}}(\B^5_1)}&\lesssim
                                              \|f\|_{W^{1,\frac{2}{1-2\delta}}(\B^5_1)} \\
\end{align*}
for all $f\in C^\infty(\overline{\B^5_1})$ and $j=1,\dots,9$.
\end{lem}

\begin{proof}
For $\dot{T}_1$ arguments similar to the ones we used for $T_1$ show that
\begin{align*}
\dot{T_1}(\tau)f(\rho)&=
\int_\rho^{1}\int_\R e^{i\omega \tau}(1-\rho^2)^{-\frac{1}{4}+\frac{\delta}{2}-\frac{i \omega}{2}}1_{(0,\rho_1)}(s)\frac{\chi_{\frac{1}{2}-\delta+i \omega}(s)\O(\rho^{-\frac{4}{5}}s^{\frac{3}{4}}\langle\omega\rangle^{-1-\frac{1}{20}})}{(1-s^2)^{\frac{3}{4}+\frac{\delta}{2}-\frac{i \omega}{2}}}f(s) d\omega ds 
\end{align*}
with $\rho_1<1$.
Thus, 
\begin{align*}
|\dot{T_1}(\tau)f(\rho)|
&\lesssim
 \rho^{-\frac{4}{5}}\int_{0}^{\rho_1}\langle\tau-\tfrac12\log(1-\rho^2)+\tfrac12 \log(1-s^2) \rangle^{-2} s^{\frac{3}{4}}|f(s)|ds
 \\
 &\lesssim
 \langle\tau\rangle^{-2}\rho^{-\frac{4}{5}}\||.|^{-1}f\|_{L^{2}(\B^5_1)}\||.|^{-\frac{1}{4}}\|_{L^{2}((0,1))}
 \\
  &\lesssim
 \langle\tau\rangle^{-2}\rho^{-\frac{4}{5}}\||.|^{-1}f\|_{L^2(\B^5_1)}
  \lesssim
 \langle\tau\rangle^{-2}\rho^{-\frac{4}{5}}\|f\|_{H^{1}(\B^5_1)}
   \lesssim
 \langle\tau\rangle^{-2}\rho^{-\frac{4}{5}}\|f\|_{W^{1,\frac{2}{1-2\delta}}(\B^5_1)}
\end{align*}
by Lemma \ref{teclem1}. Consequently, the claimed estimates on $\dot{T}_1$ follow. For $\dot{T}_2$, we perform one integration by parts and exchange powers of $\rho$ for decay in $\omega$ to derive that
\begin{align*}
\dot{T}_2(\tau)f(\rho)&=\int_\R  e^{i\omega\tau} \chi_{\frac{1}{2}-\delta+ i \omega}(\rho)(1-\rho^2)^{ -\frac{1}{4}+\frac{\delta}{2}-\frac{i \omega}{2}}\O(\rho^{0}\langle\omega\rangle^{2})
\\
&\quad\times\int_\rho^1\frac{s^2(1-\chi_{\frac{1}{2}-\delta+i \omega}(s))[1+\O(s^{-1}(1-s)\langle\omega\rangle^{-1})]\beta_1(\rho,s,\frac{1}{2}-\delta+i \omega)}{(1-s)^{\frac{1}{2}+\delta-i \omega}}f(s) ds d\omega 
\\
&=\int_\R  e^{i\omega\tau} \chi_{\frac{1}{2}-\delta+i \omega}(\rho)(1-\rho^2)^{ -\frac{1}{4}+\frac \delta 2-\frac{i \omega}{2}}\rho\O(\rho^{-1}\langle\omega\rangle^{-1})
\\
&\quad\times\frac{(1-\chi_{\frac{1}{2}-\delta+i \omega}(\rho))[1+\O(\rho^{-1}(1-\rho)\langle\omega\rangle^{-1})]\beta_1(\rho,\rho,\frac{1}{2}-\delta+i \omega)}{(1-\rho)^{-\frac{1}{2}+\delta-i \omega}}f(\rho) d\omega 
\\
&\quad+\int_\rho^1\int_\R  e^{i\omega\tau} \chi_{\frac{1}{2}-\delta+i \omega}(\rho)(1-\rho^2)^{ -\frac{1}{4}-\frac{i \omega}{2}}\O(\rho^{-\frac{9}{10}}\langle\omega\rangle^{\frac{1}{10}})
\\
&\quad\times\frac{\partial_s\left(s^2(1-\chi_{\frac{1}{2}-\delta+i \omega}(s))[1+\O(s^{-1}(1-s)\langle\omega\rangle^{-1})]\beta_1(\rho,s,\frac{1}{2}-\delta+i \omega)f(s)\right)}{(1-s)^{-\frac{1}{2}+\delta-i \omega}}  d\omega ds
\\
&=:B_2(f)(\tau,\rho)+I_2(f)(\tau,\rho).
\end{align*}
An application of Lemma \ref{osci3} then yields
\begin{align*}
|B_2(f)(\tau,\rho)|\lesssim \langle\tau\rangle^{-2} \rho |f(\rho)|
\end{align*}
and the estimates for $B_2(f)$ follow from Lemma \ref{teclem2}. To bound $I_2(f)$ we first remark that if the derivative hits $f$, one can argue as for $T_2$. Similarly, if the cut off function gets differentiated, one can argue as for $\dot{T_1}$. Making use of Lemmas \ref{osci3} and \ref{osci5} one sees that the remaining terms $\widehat{I}_2(f)(\tau,\rho)$ satisfy
\begin{align*}
|\widehat{I}_2(f)(\tau,\rho)|&\lesssim \rho^{-\frac{11}{10}}\int_\rho^1
\langle\tau+\log(1-s)\rangle^{-2} s |f(s)| (1-s)^{\frac{1}{2}-\delta} ds
\\
&\lesssim  \langle\tau\rangle^{-2} \rho^{-\frac{11}{10}}\int_0^1
s |f(s)| ds
\end{align*}
and 
\begin{align*}
|\widehat{I}_2(f)(\tau,\rho)|&\lesssim \rho^{-\frac{5}{6}}\int_\rho^1
\langle\tau+\log(1-s)\rangle^{-2}|\tau-\tfrac{1}{2}\log(1-\rho^2)+\log(1-s)|^{-\frac{1}{6}}
\\
&\quad \times s |f(s)| (1-s)^{\frac{1}{2}-\delta} ds.
\end{align*}
As a consequence, we can argue as we did for $T_2$ to derive the desired estimates on $\dot{T}_2$.
For $\dot{T_3}$ we cannot straight away take the limit $N \to \infty$ as the integral is not absolutely convergent. However, by proceeding as before and performing a similar integration by parts, this can be remedied. More precisely, we compute that
\begin{align*}
\dot{T_3}(\tau)f(\rho):=&\lim_{N\to \infty} \int_{-iN}^{iN} i\omega
e^{i\omega\tau} (1-\chi_{\frac{1}{2}-\delta+i \omega}(\rho))\rho^{-2}(1+\rho)^{\frac{1}{2}+\delta-i\omega}[1+\O(\rho^{-1}(1-\rho)\langle\omega\rangle^{-1})]
\\
&\quad\times\int_\rho^1\frac{s^2(1-\chi_{\frac{1}{2}-\delta+i \omega}(s))[1+\O(s^{-1}(1-s)\langle\omega\rangle^{-1})]}{(1-2i\omega)(1-s)^{\frac{1}{2}+\delta-i\omega}}
\\
&\quad \times \gamma_1(\rho,s,\frac 12 -\delta+ i\omega)f(s) ds d \omega
\\
&=\int_\R  e^{i\omega\tau} (1-\chi_{\frac{1}{2}-\delta+i \omega}(\rho))^2(1+\rho)^{\frac{1}{2}+\delta-i\omega}[1+\O(\rho^{-1}(1-\rho)\langle\omega\rangle^{-1})]
\\
&\quad\times\frac{[1+\O(\rho^{-1}(1-\rho)\langle\omega\rangle^{-1})]\O(\langle\omega\rangle^{-1})\gamma_1(\rho,\rho,\frac12 -\delta +i \omega)}{(1-\rho)^{-\frac{1}{2}+\delta-i\omega}}f(\rho)d\omega
\\
&\quad+\int_\rho^1 \int_\R 
e^{i\omega\tau} (1-\chi_{\frac{1}{2}-\delta+i \omega}(\rho))\rho^{-2}(1+\rho)^{\frac{1}{2}+\delta-i\omega}[1+\O(\rho^{-1}(1-\rho)\langle\omega\rangle^{-1})]
\\
&\quad\times\frac{\partial_s\left(s^2(1-\chi_{\frac{1}{2}-\delta+i \omega}(s))[1+\O(s^{-1}(1-s)\langle\omega\rangle^{-1})]\gamma_1(\rho,s,\frac 12 -\delta+ i\omega)f(s)\right)}{(1-s)^{-\frac{1}{2}+\delta-i\omega}}
\\
&\quad \times
\O(\langle\omega\rangle^{-1}) d\omega ds 
\end{align*}
and one can readily check that $\dot{T_3}$ can be bounded in a similar
fashion as $\dot{T}_2$. Furthermore, as the remaining $\dot{T}_j$ can
be bounded by analogous means, we conclude this proof.
\end{proof}
As a result of the last two lemmas one readily establishes the following proposition.
\begin{prop}\label{prop: Strichartz1}
The difference of the semigroups $\Sf$ and $\Sf_0$ satisfies the Strichartz estimates
\begin{align*}
  \|e^{-(\frac{1}{2}-\delta)\tau}[(\Sf(\tau)-\Sf_0(\tau))(\I-\Qf)(\I-\Pf)\ff]_1\|_{L^{\frac{2}{1-2\delta}}_\tau(\R_+)L^{\frac{45}{8}}(\B^5_1)}&\lesssim \|(\I-\Qf)\ff\|_{W^{1,\frac{2}{1-2\delta}}\times L^{\frac{2}{1-2\delta}}(\B^5_1)}
  \\
\|e^{-(\frac{1}{2}-\delta)\tau}[(\Sf(\tau)-\Sf_0(\tau))(\I-\Qf)(\I-\Pf)\ff]_1\|_{L^{\infty}_\tau(\R_+)L^{\frac{10}{3}}(\B^5_1)}&\lesssim
\|(\I-\Qf)\ff\|_{W^{1,\frac{2}{1-2\delta}}\times L^{\frac{2}{1-2\delta}}(\B^5_1)}
\end{align*}
for all $\ff \in C^\infty\times C^\infty(\overline{\B^5_1})$.
\end{prop}
\begin{proof}
By construction the first component of $(\Sf(\tau)-\Sf_0(\tau))(\I-\Qf)(\I-\Pf)\ff$ with $\ff \in C^\infty\times C^\infty(\overline{\B^5_1})$ is up to multiplicative constants given by
\begin{align*}
e^{(\frac{1}{2}-\delta)\tau} \left(\sum_{j=1}^9 T_j(\tau)(\widetilde{f}_1+\widetilde{f}_2) +\sum_{j=1}^9 \dot{T_j}(\tau)\widetilde{f}_1\right)
\end{align*}
with $\widetilde{f}_j=[(\I-\Qf)(\I-\Pf)\ff]_j$ for $j=1,2$.
Consequently, the claim follows immediately from Lemmas \ref{lem:Gbounds1} and \ref{lem:Gbounds2}.
\end{proof}
\subsection{Further estimates}
To be able to control the nonlinearity, we will also need estimates on derivatives. For this we have to exchange derivatives with integrals which are not absolutely convergent. We achieve this by performing enough integrations by parts to render the oscillatory integral absolutely convergent. This allows us to invoke Lemma \ref{lem:interchange} (see below) and variations thereof, which enables us to carry out said interchanging. After this we simply undo the integrations by parts.
\begin{lem}[{\cite[Lemma 6.1]{DonWal22}}] \label{lem:interchange}
Let $f(\omega)=\O\left(\langle\omega\rangle^{-1-\alpha}\right)$ with $\alpha > 0$.
Then
\begin{align*}
\partial_a \int_\R e^{i \omega a} f(\omega) d \omega= i \int_\R \omega e^{i  \omega a}f(\omega) d \omega
\end{align*}
for $a \in \R\setminus\{0\}$.
\end{lem} 
Suppose now, as before, that $\ff \in C^\infty\times C^\infty(\overline{\B^5_1})$, $\widetilde{\ff}=(\I-\Qf)(\I-\Pf)\ff$, and $\lambda=\frac{1}{2}-\delta+i\omega$. Then, by using variations of Lemma \ref{lem:interchange}, we obtain
\begin{align*}
\partial_\rho[\Sf(\tau)\widetilde{\ff}]_1(\rho)=\partial_\rho[\Sf_0(\tau)\widetilde{\ff}]_1(\rho)+\frac{1}{2\pi i}\lim_{N \to \infty}\int_{\frac12-\delta-iN}^{ \frac12-\delta+i N} e^{\lambda \tau}\partial_\rho[\mathcal{R}(F_\lambda)(\rho,\lambda)-\mathcal{R}_{\mathrm{f}}(F_\lambda) (\rho,\lambda) ]d \lambda
\end{align*}
with
\begin{align*}
\partial_\rho \mathcal{R}(f)(\rho,\lambda)=-\frac{\partial_\rho u_0(\rho,\lambda)}{2(1-\lambda)} \int_\rho^{1}\frac{s^4u_1(s,\lambda)f(s)}{(1-s^2)^{1-\lambda}} d s-\frac{\partial_\rho u_1(\rho,\lambda)}{2(1-\lambda)} \int_0^\rho\frac{s^4u_0(s,\lambda)f(s)}{(1-s^2)^{1-\lambda}} d s.
\end{align*}
Hence, our next step is to investigate the oscillatory integral above.

\begin{lem}\label{lem:decomp2}
Let $\Re\lambda=\frac{1}{2}-\delta$ and $f\in C^\infty(\overline{\B^5_1})$.
Then we can decompose $$\partial_\rho[\mathcal{R}(f)(\rho,\lambda)-\mathcal{R}_{\mathrm{f}}(f)(\rho,\lambda)]$$
as 
\begin{align*}
\partial_\rho[\mathcal{R}(f)(\rho,\lambda)-\mathcal{R}_{\mathrm{f}}(f)(\rho,\lambda)]=\sum_{j=1}^9
G_j'(f)(\rho,\lambda)
\end{align*}
with
\begin{align*}
G_1'(f)(\rho,\lambda)&=(1-\rho^2)^{ -\frac{\lambda}{2}}\int_\rho^1\frac{\chi_\lambda(s)\O(\rho^{-1}s\langle\omega\rangle^{-1})}{(1-s^2)^{1-\frac{\lambda}{2}}}f(s) ds
\\
G_2'(f)(\rho,\lambda)&=[\lambda\rho(1-\rho)^{-1}+\rho^{-1}]\chi_\lambda(\rho)(1-\rho^2)^{ -\frac{\lambda}{2}}\O(\rho^0\langle\omega\rangle^{2})
\\
&\quad\times\int_\rho^1\frac{s^2(1-\chi_\lambda(s))[1+\O(s^{-1}(1-s)\langle\omega\rangle^{-1})]\beta_1(\rho,s,\lambda)}{2(1-\lambda)(1-s)^{1-\lambda}}f(s) ds+\widetilde{G}_2(f)(\rho,\lambda)
\\
G_3'(f)(\rho,\lambda)&=\left[\frac{1-\lambda}{1+\rho}-2\rho^{-1}\right]G_3(f)(\rho,\lambda)
\\
&\quad+\O(\rho^{-2}(1-\rho)^0\langle\omega\rangle^{-1})(1-\chi_\lambda(\rho))\rho^{-2}(1+\rho)^{1-\lambda}
\\
&\quad\times\int_\rho^1\frac{s^2(1-\chi_\lambda(s))[1+\O(s^{-1}(1-s)\langle\omega\rangle^{-1})]\gamma_1(\rho,s,\lambda)}{2(1-\lambda)(1-s)^{1-\lambda}}f(s) ds+\widetilde{G}_3(f)(\rho,\lambda)
\end{align*}
and
\begin{align*}
G_4'(f)(\rho,\lambda)&=\left[-\frac{1-\lambda}{1-\rho}-2\rho^{-1}\right]G_4(f)(\rho,\lambda)
\\
&\quad+\O(\rho^{-2}(1-\rho)^0\langle\omega\rangle^{-1})(1-\chi_\lambda(\rho))\rho^{-2}(1-\rho)^{1-\lambda}
\\
&\quad\times\int_\rho^1\frac{s^2(1-\chi_\lambda(s))[1+\O(s^{-1}(1-s)\langle\omega\rangle^{-1})]\gamma_2(\rho,s,\lambda)}{2(1-\lambda)(1-s)^{1-\lambda}}f(s) ds +\widetilde{G}_4(f)(\rho,\lambda)
\\
G_5'(f)(\rho,\lambda)&=\chi_\lambda(\rho)(1-\rho^2)^{-\frac{\lambda}{2}}\int_0^\rho\frac{\O(\rho^{-2} s^2\langle\omega\rangle^{-1}) }{(1-s^2)^{1-\frac{\lambda}{2}}}f(s) ds
\\
G_6'(f)(\rho,\lambda)&=\chi_\lambda(\rho)(1-\rho^2)^{-\frac{\lambda}{2}}\int_0^\rho\frac{\O(\rho^{-2} s^2\langle\omega\rangle^{-1}) }{(1-s^2)^{1-\frac{\lambda}{2}}}f(s) ds
\end{align*}
and
\begin{align*}
G_7'(f)(\rho,\lambda)&=\left[\frac{1-\lambda}{1+\rho}-2\rho^{-1}\right]G_7(f)(\rho,\lambda)
\\
&\quad+\O(\rho^{-2}(1-\rho)^0\langle\omega\rangle^{-1})
(1-\chi_\lambda(\rho))\rho^{-2}(1+\rho)^{1-\lambda}\\
& \quad \times \int_0^\rho\frac{\chi_\lambda(s) \O(s^4\langle\omega\rangle)\beta_1(s,\rho,\lambda)}{(1-s^2)^{1-\frac{\lambda}{2}}}f(s) ds +\widetilde{G}_7(f)
\\
G_8'(f)(\rho,\lambda)&=\left[\frac{1-\lambda}{1+\rho}-2\rho^{-1}\right]G_8(f)(\rho,\lambda)
\\
&\quad+\O(\rho^{-2}(1-\rho)^0\langle\omega\rangle^{-1})
(1-\chi_\lambda(\rho))\rho^{-2}(1+\rho)^{1-\lambda}
\\
&\quad\times\int_0^\rho\frac{s^2(1-\chi_\lambda(s))[1+\O(s^{-1}(1-s)\langle\omega\rangle^{-1})]\gamma_1(s,\rho,\lambda)}{2(1-\lambda)(1-s)^{1-\lambda}}f(s) ds
+\widetilde{G}_8(f)
\\
G_9'(f)(\rho,\lambda)&=\left[\frac{1-\lambda}{1+\rho}-2\rho^{-1}\right]G_9(f)(\rho,\lambda)
\\
&\quad+\O(\rho^{-2}(1-\rho)^0\langle\omega\rangle^{-1})
(1-\chi_\lambda(\rho))\rho^{-2}(1+\rho)^{1-\lambda}
\\
&\quad\times\int_0^\rho\frac{s^2(1-\chi_\lambda(s))[1+\O(s^{-1}(1-s)\langle\omega\rangle^{-1})]\gamma_2(s,\rho,\lambda)}{2(1-\lambda)(1+s)^{1-\lambda}}f(s) ds+\widetilde{G}_9(f)
\end{align*}
where $\widetilde{G}_j(f)(\rho,\lambda)$ are the terms obtained from differentiating either $\beta_1$ or $\gamma_j$ with respect to $\rho$.
\end{lem}
\begin{proof}
This is just a straightforward computation.
\end{proof} Proceeding as above, we define operators $T'_j$ and $\dot{T}'_j$ for $j=1,\dots,9$ and $f\in 
 C^\infty(\overline{\B^5_1})$ as
 \begin{align*}
T_j'(\tau)f(\rho):=\lim_{N\to \infty} \int_{-N}^N e^{i\omega\tau} G_j'(f)(\rho,\tfrac{1}{2}-\delta+i\omega) d\omega
\end{align*}
and
\begin{align*}
\dot{T_j'}(\tau)f(\rho):=\lim_{N\to \infty} \int_{-N}^N i \omega e^{i\omega\tau} G_j'(f)(\rho,\tfrac{1}{2}-\delta+i\omega) d\omega .
\end{align*}
Again, these integrals are not necessarily absolutely convergent. Nevertheless, the operators can be made sense of, as is visible from the following lemma.
\begin{lem}
The estimates
\begin{align*}
\|T_j'(\tau)f\|_{L^6_\tau(\R_+)L^{\frac{45}{23}}(\B^5_1)}\lesssim \|f\|_{L^{\frac{2}{1-2\delta}}(\B^5_1)}
\end{align*}
and 
\begin{align*}
\|\dot{T_j'}(\tau)f\|_{L^6_\tau(\R_+)L^{\frac{45}{23}}(\B^5_1)}\lesssim \|f\|_{W^{1,\frac{2}{1-2\delta}}(\B^5_1)}
\end{align*}
hold for $j=1,\dots,9$ and all $f\in C^\infty(\overline{\B^5_1})$.
\end{lem}
\begin{proof}
We start with $j=1$, in which case we can take the limit $N \to \infty$ for $\dot{T_1'}$ and $T_1'$. Combining this with
$$
\chi_{\frac{1}{2}-\delta+i \omega}(s)\O(\rho^{-1}s\langle\omega\rangle^{-1})=\chi_{\frac{1}{2}-\delta+i \omega}(s)\O(\rho^{-\frac{13}{8}}s^{\frac{25}{16}}\langle\omega\rangle^{-\frac{17}{16}}),
$$ which is valid for $0<\rho\leq s$, yields
\begin{align*}
T_1'(\tau)f(\rho)=\int_\rho^1\int_\R e^{i\omega\tau} (1-\rho^2)^{-\frac{1}{4}+\frac{\delta}{2} -\frac{i \omega}{2}}1_{(0,\rho_1)}(s)   \frac{\chi_{\frac{1}{2}-\delta+i \omega}(s)\O(\rho^{-\frac{13}{8}}s^{\frac{25}{16}}\langle\omega\rangle^{-\frac{17}{16}})}{(1-s^2)^{\frac{1}{4}+\frac{\delta}{2}-\frac{i\omega}{2}}}f(s)  d\omega ds.
\end{align*}
Thus, by employing Lemma \ref{osci2} we obtain
\begin{align*}
|T_1'(\tau)f(\rho)|&\lesssim \langle\tau\rangle^{-2}\rho^{-\frac{13}{8}}\int_0^1s^{\frac{25}{16}}|f(s)| ds 
\lesssim \langle\tau\rangle^{-2}\rho^{-\frac{13}{8}}\left(\int_0^1s^{4}|f(s)|^2 ds\int_0^1s^{-\frac{7}{8}} ds\right)^{\frac{1}{2}}
\\
&\lesssim \langle\tau\rangle^{-2}\rho^{-\frac{13}{8}}\|f\|_{L^2(\B^5_1)}.
\end{align*}
Consequently,
\begin{align*}
\|T_1'(\tau)f\|_{L^6_\tau(\R_+)L^{\frac{45}{23}}(\B^5_1)}\lesssim \|f\|_{L^{\frac{2}{1-2\delta}}(\B^5_1)}.
\end{align*}
Moreover, to bound $\|\dot{T}'_1(\tau)f\|_{L^6_\tau(\R_+)L^{\frac{45}{23}}(\B^5_1)}$ one argues similarly to deduce that
\begin{align*}
\|\dot{T}'_1(\tau)f\|_{L^6_\tau(\R_+)L^{\frac{45}{23}}(\B^5_1)}\lesssim \||.|^{-1}f\|_{L^2(\B^5_1)}\lesssim \|f\|_{W^{1,\frac{2}{1-2\delta}}(\B^5_1)}
\end{align*}
by Lemma \ref{teclem1}.
For $j=2$ we can again interchange the order of integration and take the limit $N\to \infty$ in both $T'_2$ and $\dot{T'}_2$. Note that the hardest term over which we have to obtain control in order to bound $T'_2$ is given by
\begin{align*}
&\int_\rho^1\int_\R e^{i\omega\tau}\chi_{\frac{1}{2}-\delta+i \omega}(\rho) (1-\rho^2)^{-\frac{1}{4} +\frac{\delta}{2}-\frac{i \omega}{2}}\O(\rho^{-1}\langle\omega\rangle^{2})
\\
&\times\frac{s^2(1-\chi_{\frac{1}{2}-\delta+i \omega}(s))[1+\O(s^{-1}(1-s)\langle\omega\rangle^{-1})]\beta_1(\rho,s,\frac{1}{2}-\delta+i \omega)}{(1-2i \omega)(1-s)^{\frac{1}{2}+\delta -i\omega}}f(s) d \omega ds.
\end{align*}
By using that 
$$
\chi_{\frac{1}{2}-\delta+i \omega}(\rho)\O(\rho^{-1}\langle\omega\rangle^{0})=\chi_{\frac{1}{2}-\delta+i \omega}(\rho)\O(\rho^{-\frac52}\langle\omega\rangle^{-\frac32})
$$
we deduce that 
\begin{align*}
\|T'_2(\tau)f\|_{L^{\frac{45}{23}}(\B^5_1)} \lesssim  \||.|^{-\frac52}\|_{L^{\frac{45}{23}}(\B^5_1)}\int_0^1\langle\tau+\log(1-s)\rangle^{-2}|f(s)| s^2 (1-s)^{-\frac12-\delta}ds.
\end{align*}
Consequently, by employing previously used arguments, one readily establishes the desired estimate on $T'_2$.
Next, when estimating  $\dot{T_2'}$ the hardest term is given by 
\begin{align*}
&\quad\int_\R\int_\rho^1 e^{i\omega\tau}\chi_{\frac{1}{2}-\delta+i \omega}(\rho)(1-\rho^2)^{-\frac{1}{4}+\frac{\delta}{2}- \frac{i\omega}{2}}\O(\rho^{-1}\langle\omega\rangle^{3})
\\
&\quad\times\frac{s^2(1-\chi_{\frac{1}{2}-\delta+i \omega}(s))[1+\O(s^{-1}(1-s)\langle\omega\rangle^{-1})]\beta_1(\rho,s,\frac{1}{2}-\delta+i \omega)}{(1-2i \omega)(1-s)^{\frac{1}{2}+\delta-i \omega}}f(s)  dsd \omega
\\
&=\int_\R e^{i\omega\tau}\chi_{\frac{1}{2}-\delta+i \omega}(\rho)(1-\chi_{\frac{1}{2}-\delta+i \omega}(\rho))(1-\rho^2)^{ -\frac{1}{4}+\frac{\delta}{2}-\frac{i \omega}{2}}\O(\rho\langle\omega\rangle^{1})(1-\rho)^{\frac{1}{2}-\delta+i \omega}
\\
&\quad\times[1+\O(\rho^{-1}(1-\rho)\langle\omega\rangle^{-1})]\beta_1(\rho,\rho,\frac{1}{2}-\delta+i \omega)f(\rho)  d \omega
\\
&\quad+\int_\R\int_\rho^1 e^{i\omega\tau}\chi_{\frac{1}{2}-\delta+i \omega}(\rho)(1-\rho^2)^{-\frac14 +\frac \delta 2 -\frac{i\omega}{2}}\O(\rho^{-1}\langle\omega\rangle)(1-s)^{\frac{1}{2}-\delta+i\omega}
\\
&\quad\times\partial_s[s^2(1-\chi_{\frac{1}{2}-\delta+i \omega}(s))[1+\O(s^{-1}(1-s)\langle\omega\rangle^{-1})]\beta_1(\rho,s,\frac{1}{2}-\delta+i\omega)f(s)] d \omega ds
\\
&=:\dot{B}_2'(\tau)f(\rho)+\dot{I}_2'(\tau)f(\rho).
\end{align*}
For $\dot{B}_2'(\tau)f(\rho)$ we use Lemma \ref{osci1} to compute that
\begin{align*}
|\dot{B}_2'(\tau)f(\rho)|\lesssim \langle\tau\rangle^{-2}\rho^{-\frac{1}{6}}|f(\rho)|.
\end{align*}
Hence, 
\begin{align*}
\|\dot{B}_2'(\tau)f\|_{L^6_\tau(\R_+)L^{\frac{45}{23}}(\B^5_1)}\lesssim \|f\|_{W^{1,\frac{2}{1-2\delta}}(\B^5_1)},
\end{align*}
thanks to Lemma \ref{teclem1}.
Similarly,
\begin{align*}
|\dot{I}_2'(\tau)f(\rho)|\lesssim\langle\tau \rangle^{-2} \rho^{-2-\frac{1}{6}}\int_0^1 s |f(s)|+ s^2 |f'(s)| ds
\end{align*}
and so,
\begin{align*}
\|\dot{I}_2'(\tau)f\|_{L^6_\tau(\R_+)L^{\frac{45}{23}}(\B^5_1)}\lesssim \|f\|_{W^{1,\frac{2}{1-2\delta}}(\B^5_1)}.
\end{align*}
We proceed with $T_3'$, which we estimate according to
\begin{align*}
|T'_3(\tau)f(\rho)|&\lesssim \rho^{-2}\int_0^1 s^2 |f(s)| (1-s)^{-\frac{1}{2}-\delta} 
\\
&\quad \times \langle\tau+\log(1-s) \rangle^{-2}[1+|\tau-\log(1+\rho)+\log(1-s)|^{-\frac{1}{8}}] ds.
\end{align*}
Further, 
\begin{align*}
\|\rho^{-2}[1+|\tau-\log(1+\rho)+\log(1-s)|^{-\frac{1}{8}}]\|_{L_\rho^{\frac{45}{23}}(\B^5_1)}
\lesssim[1+|\tau+\log(1-s)|^{-\frac{1}{8}}]
\end{align*}
and the claimed estimate on $T'_3$ follows. The bound on $\dot{T}'_3$ follows by integrating parts once and then arguing in similar fashion. Moving on, Lemma \ref{osci2} shows that
\begin{align*}
|T'_4(\tau)f(\rho)|&\lesssim \rho^{-2}\int_0^1 s^2 |f(s)| (1-s)^{-\frac{1}{2}-\delta} (1-\rho)^{-\frac{1}{2}-\delta}
\\
&\quad \times \langle\tau-\log(1-\rho)+\log(1-s) \rangle^{-2}[1+|\tau-\log(1-\rho)+\log(1-s)|^{-\frac{1}{10^4}}] ds.
\end{align*}
Observe now, that the estimate
\begin{align*}
( 1-\rho)^{\frac{1}{100}} \langle\tau-\log(1-\rho)+\log(1-s)\rangle^{-2}\lesssim \langle\tau+\log(1-s)\rangle^{-2} 
\end{align*}
holds. 
Hence,
\begin{align*}
&\quad \|\rho^{-2} (1-\rho)^{-\frac{1}{2}-\delta} \langle\tau-\log(1-\rho)+\log(1-s) \rangle^{-2}
\\
&\quad\times[1+|\tau-\log(1-\rho)+\log(1-s)|^{-\frac{1}{10^4}}] \|_{L_\rho^{\frac{45}{23}}(\B^5_1)}
\\ 
 &\lesssim
 \langle\tau+\log(1-s) \rangle^{-2} \|(1-\rho)^{-\frac{1}{2}-\frac{1}{100}-\delta}[1+|\tau-\log(1-\rho)+\log(1-s)|^{-\frac{1}{10^4}}] \|_{L_\rho^{\frac{45}{23}}((0,1))}
 \\
  &\lesssim
 \langle\tau+\log(1-s) \rangle^{-2} \|(1-\rho)^{-\frac{51}{100}-\delta} \|_{L_\rho^{\frac{49}{25}}((0,1))}
 \\
 &\quad\times \| 1+|\tau-\log(1-\rho)+\log(1-s)|^{-\frac{1}{10^4}} \|_{L_\rho^{\frac{2205}{2}}((0,1))}
 \\
 &\lesssim 
 \langle\tau+\log(1-s) \rangle^{-2} [1+|\tau+\log(1-s)|^{-\frac{1}{10^4}}].
\end{align*}
Therefore,
\begin{align*}
\|T'_4(\tau)f\|_{L^{\frac{45}{23}}(\B^5_1)} &\lesssim \int_0^1 s^2 |f(s)| (1-s)^{-\frac{1}{2}-\delta} 
\langle\tau+\log(1-s) \rangle^{-2}[1+|\tau+\log(1-s)|^{-\frac{1}{10^4}}] ds
\end{align*}
and the claimed estimate follows. Furthermore, estimating $\dot{T}_4'$ is achieved by first integrating by parts once in the $s$ integral to recover decay in $\omega$ and a similar calculation. Analogously, one can bound the remaining operators, so, we conclude this proof.
\end{proof}
\begin{prop}\label{prop:Strichartz2}
The difference of $\Sf$ and $\Sf_0$ satisfies
\begin{align*}
\|e^{-(\frac{1}{2}-\delta)\tau}[(\Sf(\tau)-\Sf_0(\tau))(\I-\Qf)(\I-\Pf)\ff]_1\|_{L^6_\tau(\R_+)W^{1,\frac{45}{23}}(\B^5_1)}\lesssim \|(\I-\Qf)\ff\|_{W^{1,\frac{2}{1-\delta}}\times L^{\frac{2}{1-2\delta}}(\B^5_1)}
\end{align*}
for all $\ff\in C^\infty \times C^\infty(\overline{\B^5_1}).$
\end{prop}
With this result, our task of establishing Strichartz estimates on the $W^{1,\frac{2}{1-\delta}}\times L^{\frac{2}{1-2\delta}}$ level has come to an end and we move on to the next set of estimates.
\section{Strichartz estimates in $W^{2,\frac{2}{1+2\delta}}$}
We now move on to $W^{2,\frac{2}{1+2\delta}}\times W^{1,\frac{2}{1+2\delta}}$-type Strichartz estimates, i.e., estimates of the form 
\begin{align*}
\left\|[e^{(\frac{1}{2}-\delta)\tau }\Sf(\tau)\widetilde
{\mathbf f}]_1\right\|_{L^p_\tau(\R_+) L^q(\B^5_1)}\lesssim \|\widetilde{\ff}\|_{W^{2,\frac{2}{1+2\delta}}\times W^{1,\frac{2}{1+2\delta}}(\B^5_1)}
\end{align*}
For this we break the difference $\mathcal{R}(f)-\mathcal{R_{\mathrm{f}}}(f)$, into smaller pieces. The first part we look at is given by
$$W_1(f)(\rho,\lambda):=b_\lambda(f)u_0(\rho,\lambda) -b_{\mathrm{f}_\lambda}(f)u_{\mathrm{f}_0}(\rho,\lambda). $$
\begin{lem}
Let $\Re \lambda=-\frac12+\delta$. Then we can decompose $W_1(f)(\rho,\lambda)$ as 
$$
W_1(f)(\rho,\lambda)=f(1)\sum_{j=1}^3 H_j(\rho,\lambda)$$ 
where
\begin{align*}
H_1(\rho,\lambda):&=\chi_\lambda(\rho)(1-\rho^2)^{-\lambda}\O(\rho^0\langle\omega\rangle^{-2})
\\
H_2(\rho,\lambda):&=(1-\chi_\lambda(\rho))\rho^{-2}(1+\rho)^{1-\lambda}[1+\O(\rho^{-1}(1-\rho)\langle\omega\rangle^{-1})]\\
&\times\left[\O(\langle\omega\rangle^{-4})+(1-\rho)\O(\langle\omega\rangle^{-5})+\O(\rho^{-1}(1-\rho)^2\langle\omega\rangle^{-5})\right]
\\
H_3(\rho,\lambda):&=(1-\chi_\lambda(\rho))\rho^{-2}(1-\rho)^{1-\lambda}[1+\O(\rho^{-1}(1-\rho)\langle\omega\rangle^{-1})]\\
&\times\left[\O(\langle\omega\rangle^{-4})+(1-\rho)\O(\langle\omega\rangle^{-5})+\O(\rho^{-1}(1-\rho)^2\langle\omega\rangle^{-5})\right].
\end{align*}

\end{lem}
\begin{proof}
We start by looking at 
\begin{align*}
b_\lambda(f)&=\frac{ f(1)}{2\lambda(1-\lambda)}\int_0^1\partial_s[s^4u_1(s,\lambda)(1+s)^{-1+\lambda}](1-s)^\lambda ds
\\
&=\frac{ f(1)}{2\lambda(1-\lambda)}\bigg(\int_0^1\chi_\lambda(s)\partial_s[s^4u_1(s,\lambda)(1+s)^{-1+\lambda}](1-s)^\lambda ds
\\
&\quad+\int_0^1(1-\chi_\lambda(s))\partial_s[s^4u_1(s,\lambda)(1+s)^{-1+\lambda}](1-s)^\lambda ds\bigg)
\end{align*}
and claim that $ b_\lambda(f)= f(1)\O(\langle\omega\rangle^{-3}).$
For the first of the above terms on the right side one readily computes that
\begin{align*}
\int_0^1&\chi_\lambda(s)\partial_s[s^4u_1(s,\lambda)(1+s)^{-1+\lambda}](1-s)^\lambda ds
=\int_0^1\chi_\lambda(s)\partial_s\O(s\langle\omega\rangle^{-1}) ds
=\O(\langle\omega\rangle^{-1}).
\end{align*}
The second term we split according to
\begin{align*}
&\quad\int_0^1(1-\chi_\lambda(s))\partial_s[s^2 [1+(1-s)\O(\langle\omega\rangle^{-1})+\O(s^{-1}(1-s)^2\langle\omega\rangle^{-1})]]
 (1-s)^{\lambda} ds
 \\
&=\int_0^1(1-\chi_\lambda(s))\partial_s[s^2 [1+(1-s)\O(\langle\omega\rangle^{-1})]]
 (1-s)^{\lambda} ds
 \\
 &\quad+\int_0^1(1-\chi_\lambda(s))\partial_s[s^2 [\O(s^{-1}(1-s)^2\langle\omega\rangle^{-1})]]
 (1-s)^{\lambda} ds
 =:I_1(\lambda)+I_2(\lambda).
\end{align*}
Observe that
$I_2(\lambda) =\O(\langle\omega\rangle^{-1})$, while an integration by parts yields
\begin{align*}
I_1(\lambda)
&=\O(\langle\omega\rangle^{-1})\int_0^1\partial_s\left((1-\chi_\lambda(s))\partial_s[s^2 [1+(1-s)\O(\langle\omega\rangle^{-1})]]\right)
 (1-s)^{1+\lambda} ds 
\\ 
&=\O(\langle\omega\rangle^{-1}).
\end{align*}
Consequently, the claim follows.
Similarly, one computes that $$b_\lambda(f)- b_{\mathrm{f}_\lambda}(f)=f(1)\O(\langle\omega\rangle^{-4}).$$
Therefore, one establishes the desired decomposition by plugging in the explicit forms of the solutions and a straightforward computation.
\end{proof}
Motivated by this decomposition we define the operators 
\begin{align*}
S_{j}(\tau)f(\rho)=\int_\R e^{i \omega\tau}f(1)H_{j}(\rho,-\tfrac{1}{2}+\delta+i \omega) d\omega
\end{align*}
and
\begin{align*}
\dot{S}_{j}(\tau)f(\rho)=\int_\R  \omega e^{i \omega\tau}f(1)H_{j}(\rho,-\tfrac{1}{2}+\delta+i \omega) d\omega
\end{align*}
for $j=1,2,3$ and $f \in C^\infty(\overline{\B^5_1})$.
\begin{lem} \label{lem:Hbounds1}
The estimates
\begin{align*} 
\|S_{j}(\tau)f\|_{L^{\frac{2}{1+2\delta}}_\tau(\R_+)L^{45}(\B^5_1)}&\lesssim \|f\|_{W^{1,\frac{2}{1+2\delta}}(\B^5_1)}
\\
\|S_{j}(\tau)f\|_{L^\infty_\tau(\R_+)L^{10}(\B^5_1)}&\lesssim \|f\|_{W^{1,\frac{2}{1+2\delta}}(\B^5_1)}
\end{align*}
and
\begin{align*}
\|\dot S_{j}(\tau)f\|_{L^{\frac{2}{1+2\delta}}_\tau(\R_+)L^{45}(\B^5_1)}&\lesssim \|f\|_{W^{2,\frac{2}{1+2\delta}}(\B^5_1)}
\\
\|\dot{S}_{j}(\tau)f\|_{L^\infty_\tau(\R_+)L^{10}(\B^5_1)}&\lesssim \|f\|_{W^{2,\frac{2}{1+2\delta}}(\B^5_1)}
\end{align*}
hold for $j=1,2,3$ and $f\in C^\infty(\overline{\B^5_1})$.
\end{lem}
\begin{proof}
From Lemma \ref{osci1} we see that
\begin{align*}
|S_1(\tau)f(\rho)|\lesssim |f(1)|\langle\tau\rangle^{-2},
\end{align*}
hence, the bounds on $S_1$ follow from the Sobolev embedding
\begin{equation*}
W^{1,\frac{2}{1+2\delta}}((0,1))\hookrightarrow L^\infty([0,1])
\end{equation*}
and \ref{teclem1}.
To establish the estimates on $\dot{S}_1$ we use that 
\begin{align*}
\chi_{-\frac{1}{2}+\delta+i \omega}(\rho)\O(\rho^0\langle\omega\rangle^{-1})=\chi_{-\frac{1}{2}+\delta+i \omega}(\rho)\O(\rho^{-\frac{1}{100}}\langle\omega\rangle^{-1-\frac{1}{100}})
\end{align*} 
and employ Lemma \ref{osci1} to establish that
\begin{align*}
|\dot{S}_1(\tau)f(\rho)|\lesssim \rho^{-\frac{1}{100}} |f(1)|\langle\tau\rangle^{-2}.
\end{align*}
Consequently,
the estimates on $\dot{S}_1$ follow.
The remaining bounds can be obtained in a similar fashion by making use of Lemma \ref{osci3}.
\end{proof}
Next, we take a closer look at
\begin{align*}
W_2(f)(\rho,\lambda):&=u_0(\rho,\lambda)U_1(\rho,\lambda)f(\rho)-u_{\mathrm{f}_0}(\rho,\lambda)U_{\mathrm{f}_1}(\rho,\lambda)f(\rho)
\\
&\quad-u_1(\rho,\lambda) U_0(\rho,\lambda)f(\rho)
+u_{\mathrm{f}_1}(\rho,\lambda) U_{\mathrm{f}_0}(\rho,\lambda)f(\rho)
\end{align*}
\begin{lem}
Let $\Re \lambda=-\frac12+\delta$. Then we can decompose  $W_2(f)$ as
\begin{align*}
W_2(f)(\rho,\lambda)=f(\rho)\sum_{j=4}^8 H_j(\rho,\lambda)
\end{align*}
where
\begin{align*}
H_4(\rho,\lambda):&=(1-\rho^2)^{-\frac{\lambda}{2}}\chi_\lambda(\rho)\int_0^\rho\frac{\O(\rho^0s\langle\omega\rangle^{-1})}{(1-s^2)^{1-\frac{\lambda}{2}}} ds
\\
H_5(\rho,\lambda):&=(1-\chi_\lambda(\rho))\rho^{-2}(1+\rho)^{1-\lambda}[1+\O(\rho^{-1}(1-\rho)\langle\omega\rangle^{-1})]
\\
&\quad \times \int_0^\rho\frac{\chi_\lambda(s)\O(s\langle\omega\rangle^{-2})}{(1-s^2)^{1-\frac{\lambda}{2}}} \beta_2(\rho,s,\lambda) ds
\\
H_6(\rho,\lambda):&=(1-\chi_\lambda(\rho))\rho^{-2}(1-\rho)^{1-\lambda}[1+\O(\rho^{-1}(1-\rho)\langle\omega\rangle^{-1})]
\\
&\quad \times\int_0^\rho\frac{\chi_\lambda(s)\O(s\langle\omega\rangle^{-2})}{(1-s^2)^{1-\frac{\lambda}{2}}} \beta_3(\rho,s,\lambda) ds
\\
H_7(\rho,\lambda):&=(1-\chi_\lambda(\rho))\rho^{-2}(1-\rho)^{1-\lambda}[1+\O(\rho^{-1}(1-\rho)\langle\omega\rangle^{-1})]
\\
&\quad\times\int_0^\rho\frac{s^2(1-\chi_\lambda(s))[1+\O(s^{-1}(1-s)\langle\omega\rangle^{-1})]\gamma_3(\rho,s,\lambda)}{2(1-\lambda)(1-s)^{1-\lambda}}ds
\\
H_8(\rho,\lambda):&=(1-\chi_\lambda(\rho))\rho^{-2}(1+\rho)^{1-\lambda}[1+\O(\rho^{-1}(1-\rho)\langle\omega\rangle^{-1})]
\\
&\quad\times\int_0^\rho\frac{s^2(1-\chi_\lambda(s))[1+\O(s^{-1}(1-s)\langle\omega\rangle^{-1})]\gamma_4(\rho,s,\lambda)}{2(1-\lambda)(1+s)^{1-\lambda}}ds
\end{align*}
with $\beta_j$ and $\gamma_j$ as in Lemma \ref{lem:decomp1}. 
\end{lem}
As before we define operators corresponding to the kernels $H_j$ as 
\begin{align*}
S_j(\tau)f(\rho):=\lim_{N\to \infty} \int_{-N}^N  e^{i\omega\tau} f(\rho)H_j(\rho,-\tfrac{1}{2}+\delta +i\omega) d\omega
\end{align*}
and
\begin{align*}
\dot{S}_j(\tau)f(\rho):=\lim_{N\to \infty} \int_{-N}^N \omega e^{i\omega\tau}  f(\rho)H_j(\rho,-\tfrac{1}{2}+\delta+ i \omega) d\omega
\end{align*}
for $j=4,\dots,8$ and $f\in C^\infty(\overline{\B^5_1})$.
\begin{lem} \label{lem:Hbounds2}
The estimates
\begin{align*} 
\|S_{j}(\tau)f\|_{L^{\frac{2}{1+2\delta}}_\tau(\R_+)L^{45}(\B^5_1)}&\lesssim \|f\|_{W^{1,\frac{2}{1+2\delta}}(\B^5_1)}
\\
\|S_{j}(\tau)f\|_{L^\infty_\tau(\R_+)L^{10}(\B^5_1)}&\lesssim \|f\|_{W^{1,\frac{2}{1+2\delta}}(\B^5_1)}
\end{align*}
and
\begin{align*}
\|\dot S_{j}(\tau)f\|_{L^{\frac{2}{1+2\delta}}_\tau(\R_+)L^{45}(\B^5_1)}&\lesssim \|f\|_{W^{2,\frac{2}{1+2\delta}}(\B^5_1)}
\\
\|\dot{S}_{j}(\tau)f\|_{L^\infty_\tau(\R_+)L^{10}(\B^5_1)}&\lesssim \|f\|_{W^{2,\frac{2}{1+2\delta}}(\B^5_1)}
\end{align*}
for $j=4,\dots,8$ and $f\in C^\infty(\overline{\B^5_1})$.
\end{lem}
\begin{proof}
For $j=4$ 
exchanging a small power of $\rho$ for decay in $\omega$ and applying Lemma \ref{osci1} yields the estimate
\begin{align*}
|S_4(\tau)f(\rho)|\lesssim \langle\tau\rangle^{-2}\rho^{2-\delta}|f(\rho)|.
\end{align*}
So,
\begin{align*}
\|S_4(\tau)f(\rho)\|_{L^{45}_\rho(\B^5_1)} &\lesssim \langle\tau\rangle^{-2} \|\rho^{2+\frac{1}{90}}f(\rho)\|_{L^{45}_\rho((0,1))}
\end{align*}
provided that $\delta$ is sufficiently small.
Hence, from the embedding
$W^{1,\frac{2}{1+2\delta}}((0,1))\hookrightarrow L^{45}([0,1])$
we conclude that
\begin{align*}
\|\rho^{2+\frac{1}{90}}f(\rho)\|_{L^{45}_\rho((0,1))}&\lesssim \|\rho^{2+\frac{1}{90}}f(\rho)\|_{W^{1,\frac{2}{1+2\delta}}_\rho((0,1))}\lesssim \|f(\rho)\|_{W^{1,\frac{2}{1+2\delta}}(\B^5_1)}+\|\rho^{1+\frac{1}{90}}f(\rho)\|_{L^{1}_\rho((0,1))}
\\
& \lesssim \|f(\rho)\|_{W^{1,\frac{2}{1+2\delta}}(\B^5_1)}
\end{align*}
where the last inequality follows from Theorem 1 in \cite{Ost22}.
Hence the desired estimates on $S_4$ follow.
To estimate $\dot{S}^4$ one computes that 
\begin{align*}
|\dot S_4(\tau)f(\rho)|\lesssim \langle\tau\rangle^{-2}\rho^{1-\delta}|f(\rho)|
\end{align*}
and so the estimates on $\dot S_4$ follow from similar considerations.
 \ref{lem:Gbounds2}. 
For $j=5$ we apply Lemma \ref{osci3} to see that
\begin{align*}
|S_5(\tau)f(\rho)|\lesssim \langle\tau\rangle^{-2}\rho^2|f(\rho)|.
\end{align*}
and 
\begin{align*}
|\dot S_5(\tau)f(\rho)|\lesssim \langle\tau\rangle^{-1}\rho^2|f(\rho)|.
\end{align*}
Thus, the bounds on $S_5$ and $\dot S_5$ follow.  Moreover, as the estimates for $j=6$ can be obtained
likewise, we move on to $S_7$. Here, an application of Lemma \ref{osci3} shows that
\begin{align*}
|S_7(\tau)f(\rho)|&\lesssim \rho^{-1}|f(\rho)| (1-\rho)^{\frac{3}{2}-\delta}\int_0^\rho\frac{\langle\tau-\log(1-\rho)+\log(1-s)\rangle^{-2}s^2 }{(1-s)^{\frac{3}{2}-\delta}}ds
\\
&\lesssim \langle\tau\rangle^{-2} \rho^{2}|f(\rho)|
\end{align*}
and again the desired bounds follow.
 To bound $\dot{S}_7$ we integrate by parts once to see that
\begin{align*}
H_7(\rho,\lambda)&=(1-\chi_\lambda(\rho))\rho^{-2}(1-\rho)^{1-\lambda}[1+\O(\rho^{-1}(1-\rho)\langle\omega\rangle^{-1})]
\\
&\quad\times\int_0^\rho\frac{s^2(1-\chi_\lambda(s))[1+\O(s^{-1}(1-s)\langle\omega\rangle^{-1})]\gamma_3(\rho,s,\lambda)}{2(1-\lambda)(1-s)^{1-\lambda}}ds
\\
&=(1-\chi_\lambda(\rho))^2(1-\rho)[1+\O(\rho^{-1}(1-\rho)\langle\omega\rangle^{-1})]
\\
&\quad\times [1+\O(\rho^{-1}(1-\rho)\langle\omega\rangle^{-1})]\frac{\gamma_3(\rho,\rho,\lambda)}{2\lambda(1-\lambda)}
\\
&\quad+(1-\chi_\lambda(\rho))\rho^{-2}(1-\rho)^{1-\lambda}[1+\O(\rho^{-1}(1-\rho)\langle\omega\rangle^{-1})]
\\
&\quad\times\int_0^\rho\frac{\partial_s[s^2(1-\chi_\lambda(s))[1+\O(s^{-1}(1-s)\langle\omega\rangle^{-1})]\gamma_3(\rho,s,\lambda)]}{2\lambda(1-\lambda)(1-s)^{-\lambda}}ds.
\end{align*}
By recasting $H_7$ as such and employing Lemma \ref{osci3} 
the claimed bounds follow. Finally, as $S_8$ and $\dot{S}_8$ can be bounded likewise, we conclude this proof.
\end{proof}
To proceed, we take a closer look at
\begin{align*}
W_3(f)(\rho,\lambda):&=u_0(\rho,\lambda)\int_\rho^1 U_1(s,\lambda)f'(s) ds -u_{\mathrm{f}_0}(\rho,\lambda)\int_\rho^1 U_{\mathrm{f}_1}(s,\lambda)f'(s) ds 
\\
&\quad+u_1(\rho,\lambda)\int_0^\rho U_0(s,\lambda)f'(s)d s-u_{\mathrm{f}_1}(\rho,\lambda)\int_0^\rho U_{\mathrm{f}_0}(s,\lambda)f'(s)d s.
\end{align*} 
\begin{lem}
Let $\Re \lambda=-\frac12+\delta$. Then we can decompose $W_3(\rho,\lambda)$ as
\begin{equation*}
W_3(f)(\rho,\lambda)=\sum_{j=9}^{18} H_j(f)(\rho,\lambda)
\end{equation*} 
where
\begin{align*}
H_9(f)(\rho,\lambda):&=(1-\rho^2)^{-\frac{\lambda}{2}}\chi_\lambda(\rho)\int_\rho^1 f'(s)\int_0^s\frac{\chi_\lambda(t)\O(\rho^0t\langle\omega\rangle^{-1})}{(1-t^2)^{1-\frac{\lambda}{2}}} dt ds
\\
H_{10}(f)(\rho,\lambda):&=(1-\rho^2)^{-\frac{\lambda}{2}}\chi_\lambda(\rho)\int_\rho^1 f'(s)
\\
&\quad \times \int_0^s\frac{(1-\chi_\lambda(t))\O(\rho^0t^2\langle\omega\rangle)[1+\O(t^{-1}(1-t)\langle\omega\rangle^{-1})]}{(1-t)^{1-\lambda}}\beta_4(\rho,t,\lambda) dt ds
\\
H_{11}(f) (\rho,\lambda):&=(1-\chi_\lambda(\rho))\rho^{-2}(1-\rho)^{1-\lambda}[1+\O(\rho^{-1}(1-\rho)\langle\omega\rangle^{-1})]
\\
&\quad\times\int_\rho^1 f'(s) \int_0^s\frac{\chi_\lambda(t)\O(t\langle\omega\rangle^{-2})}{(1-t^2)^{1-\frac{\lambda}{2}}}\beta_5(t,\rho,\lambda) dt ds
\end{align*}
and 
\begin{align*}
H_{12}(f) (\rho,\lambda):&=(1-\chi_\lambda(\rho))\rho^{-2}(1+\rho)^{1-\lambda}[1+\O(\rho^{-1}(1-\rho)\langle\omega\rangle^{-1})]
\\
&\quad\times\int_\rho^1 f'(s) \int_0^s\frac{\chi_\lambda(t)\O(t\langle\omega\rangle^{-2})}{(1-t^2)^{1-\frac{\lambda}{2}}}\beta_6(t,\rho,\lambda) dt ds
\\
H_{13}(f)(\rho,\lambda):&=(1-\chi_\lambda(\rho))\rho^{-2}(1-\rho)^{1-\lambda}[1+\O(\rho^{-1}(1-\rho)\langle\omega\rangle^{-1})]
\\
&\quad\times\int_\rho^1 f'(s)\int_0^s\frac{t^2(1-\chi_\lambda(t))[1+\O(t^{-1}(1-t)\langle\omega\rangle^{-1})]\gamma_3(\rho,t,\lambda)}{2(1-\lambda)(1-t)^{1-\lambda}}dt ds
\\
H_{14}(f)(\rho,\lambda):&=(1-\chi_\lambda(\rho))\rho^{-2}(1+\rho)^{1-\lambda}[1+\O(\rho^{-1}(1-\rho)\langle\omega\rangle^{-1})]
\\
&\quad\times\int_\rho^1 f'(s)\int_0^s\frac{t^2(1-\chi_\lambda(t))[1+\O(t^{-1}(1-t)\langle\omega\rangle^{-1})]\gamma_4(\rho,t,\lambda)}{2(1-\lambda)(1-t)^{1-\lambda}}dt ds
\end{align*}
and
\begin{align*}
H_{15}(f)(\rho,\lambda):&=(1-\rho^2)^{-\frac{\lambda}{2}}\chi_\lambda(\rho)\int_0^\rho f'(s)\int_0^s\frac{\O(\rho^0t\langle\omega\rangle^{-1})}{(1-t^2)^{1-\frac{\lambda}{2}}} dt ds
\\
H_{16}(f) (\rho,\lambda):&=(1-\chi_\lambda(\rho))\rho^{-2}(1+\rho)^{1-\lambda}[1+\O(\rho^{-1}(1-\rho)\langle\omega\rangle^{-1})]
\\
&\quad\times\int_0^\rho f'(s) \int_0^s\frac{\chi_\lambda(t)\O(t\langle\omega\rangle^{-1})}{(1-t^2)^{1-\frac{\lambda}{2}}}\beta_7(t,\rho,\lambda) dt ds
\\
H_{17}(f)(\rho,\lambda):&=(1-\chi_\lambda(\rho))\rho^{-2}(1+\rho)^{1-\lambda}[1+\O(\rho^{-1}(1-\rho)\langle\omega\rangle^{-1})]
\\
&\quad\times\int_0^\rho f'(s)\int_0^s\frac{t^2(1-\chi_\lambda(t))[1+\O(t^{-1}(1-t)\langle\omega\rangle^{-1})]\gamma_5(\rho,t,\lambda)}{2(1-\lambda)(1-t)^{1-\lambda}}dt ds
\\
H_{18}(f)(\rho,\lambda):&=(1-\chi_\lambda(\rho))\rho^{-2}(1+\rho)^{1-\lambda}[1+\O(\rho^{-1}(1-\rho)\langle\omega\rangle^{-1})]
\\
&\quad\times\int_0^\rho f'(s)\int_0^s \frac{t^2(1-\chi_\lambda(t))[1+\O(t^{-1}(1-t)\langle\omega\rangle^{-1})]\gamma_6(\rho,t,\lambda)}{2(1-\lambda)(1+t)^{1-\lambda}}dt ds
\end{align*}
with $\beta_j$ and $\gamma_j$ as in Lemma \ref{lem:decomp1}.
\end{lem}
Continuing, we set 
\begin{align*}
S_{j}(\tau)f(\rho)=\lim_{N \to \infty}\int_N^{-N} e^{i \omega\tau}H_{j}(f)(\rho,-\tfrac{1}{2}+\delta+i \omega) d\omega
\end{align*}
and
\begin{align*}
\dot{S}_{j}(\tau)f(\rho)=\lim_{N \to \infty}\int_N^{-N} \omega e^{i \omega\tau}H_{j}(f)(\rho,-\tfrac{1}{2}+\delta+i \omega) d\omega
\end{align*}
for $j=9,\dots,18$ and $f \in C^\infty(\overline{\B^5_1})$.
\begin{lem}\label{lem:Hbounds3}
The estimates
\begin{align*} 
\|S_{j}(\tau)f\|_{L^{\frac{2}{1+2\delta}}_\tau(\R_+)L^{45}(\B^5_1)}&\lesssim \|f\|_{W^{1,\frac{2}{1+2\delta}}(\B^5_1)}
\\
\|S_{j}(\tau)f\|_{L^\infty_\tau(\R_+)L^{10}(\B^5_1)}&\lesssim \|f\|_{W^{1,\frac{2}{1+2\delta}}(\B^5_1)}
\end{align*}
and
\begin{align*}
\|\dot S_{j}(\tau)f\|_{L^{\frac{2}{1+2\delta}}_\tau(\R_+)L^{45}(\B^5_1)}&\lesssim \|f\|_{W^{2,\frac{2}{1+2\delta}}(\B^5_1)}
\\
\|\dot{S}_{j}(\tau)f\|_{L^\infty_\tau(\R_+)L^{10}(\B^5_1)}&\lesssim \|f\|_{W^{2,\frac{2}{1+2\delta}}(\B^5_1)}
\end{align*}
hold for $j=9,\dots,18$ and $f\in C^\infty(\overline{\B^5_1})$.
\end{lem}
\begin{proof}
Lemma \ref{osci1} yields the estimate
\begin{align*}
|S_{9}(\tau)f(\rho)|\lesssim \langle\tau\rangle^{-2} \int_0^1 |f'(s)| s^{\frac{7}{4}}ds 
\end{align*}
and so, from the Cauchy Schwarz inequality we can immediately infer the desired estimates on $S_9$.
Analogously, one derives that
\begin{align*}
\|\dot{S}_{9}(\tau)f\|_{L^p_\tau(\R_+)L^q(\B^5_1)}\lesssim \||.|^{-1}f\|_{H^1(\B^5_1)}\lesssim \|f\|_{H^2(\B^5_1)}.
\end{align*}
For $j=10$ we perform an integration by parts to conclude that
\begin{align*}
H_{10}(f)(\rho,\lambda)&=(1-\rho^2)^{-\frac{\lambda}{2}}\chi_\lambda(\rho)\int_\rho^1 f'(s)
\\
&\quad \times\frac{(1-\chi_\lambda(s))\O(\rho^0s^2\langle\omega\rangle^0)[1+\O(s^{-1}(1-s)\langle\omega\rangle^{-1})]}{(1-s)^{-\lambda}}\beta_1(\rho,s,\lambda)ds 
\\
&\quad+(1-\rho^2)^{-\frac{\lambda}{2}}\chi_\lambda(\rho)\int_\rho^1 f'(s)
\\
&\quad \times\int_0^s\frac{\partial_t\left((1-\chi_\lambda(t))\O(\rho^0t^2\langle\omega\rangle^{0})[1+\O(t^{-1}(1-t)\langle\omega\rangle^{-1})]\beta_1(\rho,t,\lambda)\right)}{(1-t)^{-\lambda}} dt ds
\\
&=:B_{10}(f)(\rho,\lambda)+I_{10}(f)(\rho,\lambda).
\end{align*}
By employing Lemma \ref{osci2} and substituting $s=1-e^{-y}$ we estimate
\begin{align*}
&\quad\left|\int_\R e^{i \omega\tau}B_{10}(f)(\rho,-\tfrac{1}{2}+\delta+i \omega) d\omega\right|
\\
&\lesssim \rho^{-\frac{1}{90}} \int_\rho^1|\tau-\tfrac{1}{2}\log(1-\rho^2)+\log(1-s)|^{-\frac{1}{8}} \langle\tau +\log(1-s)\rangle^{-2}|f'(s)| s^{2+\frac{1}{90}} (1-s)^{-\frac{1}{2}+\delta} ds
\\
&\lesssim
\rho^{-\frac{1}{90}}\int_0^\infty|\tau-\tfrac{1}{2}\log(1-\rho^2)-y|^{-\frac{1}{8}} \langle\tau-y\rangle^{-2}|f'(1-e^{-y})| (1-e^{-y})^{2+\frac{1}{90}} e^{-(\frac{1}{2}+\delta)y} ds
\\
&\lesssim \rho^{-\frac{1}{90}} \left(\int_0^\infty|\tau-\tfrac{1}{2}\log(1-\rho^2)-y|^{-\frac{1}{3}} \langle\tau-y\rangle^{-2} d y\right)^{\frac{1-2\delta }{2}}
\\
&\quad \times \left(\int_0^\infty\langle\tau-y\rangle^{-2}|f'(1-e^{-y})|^{\frac{2}{1+2\delta}} (1-e^{-y})^4 e^{-y} dy \right)^{\frac{1+2\delta}{2}}.
\end{align*}
Observe that
\begin{align*}
\int_0^\infty|\tau-\tfrac{1}{2}\log(1-\rho^2)-y|^{-\frac{1}{3}} \langle\tau-y\rangle^{-2} d y &\lesssim \int_\R |\tfrac{1}{2}\log(1-\rho^2)+y|^{-\frac{1}{3}} \langle y\rangle^{-2} d y
\\
&\lesssim 
\int_{-\frac{1}{2}\log(1-\rho^2)-1}^{-\frac{1}{2}\log(1-\rho^2)+1}|\tfrac{1}{2}\log(1-\rho^2)+y|^{-\frac{1}{3}} d y
\\
&\quad+\int_\R  \langle y\rangle^{-2} d y.
\end{align*}
Thus, as 
\begin{align*}
\int_{-\frac{1}{2}\log(1-\rho^2)-1}^{-\frac{1}{2}\log(1-\rho^2)+1}|\tfrac{1}{2}\log(1-\rho^2)+y|^{-\frac{1}{3}} d y=\int_{-1}^{1}|y|^{-\frac{1}{3}} d y \lesssim 1
\end{align*}
we deduce that 
\begin{align*}
&\quad\left\|\int_\R e^{i \omega\tau}B_{10}(f)(\rho,-\tfrac{1}{2}+\delta+i \omega) d\omega\right \|_{L^2_\tau(\R_+)L^\infty_\rho(\B^5_1)}^2 
\\
&\lesssim \int_0^\infty \int_0^\infty \langle\tau-y\rangle^{-2}|f'(1-e^{-y})|^{\frac{2}{1+2\delta}} (1-e^{-y})^4 e^{-y} dy d\tau.
\end{align*}
Hence, the bounds on $ \int_\R e^{i \omega\tau}B_{10}(f)(\rho,-\tfrac{1}{2}+\delta+i \omega) d\omega$ follow from  Young's inequality.
To proceed, we illustrate the general procedure on how to bound 
$$
\int_\R e^{i \omega\tau}I_{10}(f)(\rho,-\tfrac{1}{2}+\delta+i \omega) d\omega
$$
with \begin{align*}
\widetilde{I}_{10}(f)(\rho,\lambda)&=(1-\rho^2)^{-\frac{\lambda}{2}}\chi_\lambda(\rho)\int_\rho^1 f'(s)
\\
&\quad\times\int_0^s\frac{(1-\chi_\lambda(t))\O(\rho^0t^2\langle\omega\rangle^{0})\O(t^{-2}(1-t)^0\langle\omega\rangle^{-1})]\beta_1(\rho,t,\lambda)}{(1-t)^{-\lambda}} dt ds
\end{align*}
i.e., the term we obtain when the $t$ derivative hits $[1+\O(t^{-1}(1-t)\langle\omega\rangle^{-1})]$.
Here, we use Lemma \ref{osci3} to derive that
\begin{align*}
 \left|\int_\R e^{i \omega\tau}\widetilde{I}_{10}(f)(\rho,-\tfrac{1}{2}+\delta+i \omega) d\omega\right|
&\lesssim \int_0^1 |f'(s)|\int_0^s \langle\tau +\log(1-t)\rangle^{-2}
\frac{t}{(1-t)^{\frac12-\delta}} dt ds.
\end{align*}
So, 
\begin{align*}
&\quad \left\|\int_\R e^{i \omega\tau}\widetilde{I}_{10}(f)(\rho,-\tfrac{1}{2}+\delta+i \omega) d\omega\right\|_{L^{\frac{2}{1+2\delta}}_\tau(\R_+)L^{45}_\rho(\B^5_1)}
\\
&\lesssim \int_0^1 |f'(s)| \left\|\int_0^s \langle\tau +\log(1-t)\rangle^{-2}
\frac{t}{(1-t)^{\frac12 -\delta}} dt \right\|_{L^{\frac{2}{1+2\delta}}_\tau(\R_+)} ds.
\end{align*}
Furthermore,
\begin{align*}
&\quad\left\|\int_0^s \langle\tau +\log(1-t)\rangle^{-2}
\frac{t}{(1-t)^{\frac12 -\delta}} dt \right\|_{L^{\frac{2}{1+2\delta}}_\tau(\R_+)} ds
\\
&=\left\|\int_0^{-\log(1-s)} \langle\tau -y\rangle^{-2}
(1-e^{-y})
e^{-(\frac{1}{2}+\delta)} dy \right\|_{L^{\frac{2}{1+2\delta}}_\tau(\R_+)}
\end{align*}
and using Young's inequality yields
\begin{align*}
&\quad\left\|\int_0^{-\log(1-s)} \langle\tau -y\rangle^{-2}
(1-e^{-y})e^{-\frac{y}{2}} dy \right\|_{L^{\frac{2}{1+2\delta}}_\tau(\R_+)}
\\
&\lesssim \|1_{(0,-\log(1-s))}(y)
(1-e^{-y})e^{-\frac{y}{2}}\|_{L^\frac32_{y}(\R)}
\\
&\lesssim \left(\int_0^{s}(1-t)^{-\frac{1}{4}}t^{\frac{3}{2}} dt \right)^{\frac{2}{3}}
\lesssim (1-s)^{-\frac{1}{6}}s^{\frac{5}{3}}.
\end{align*}
Thus,
\begin{align*}
&\quad\left\|\int_\R e^{i \omega\tau}\widetilde{I}_{10}(f)(\rho,-\tfrac{1}{2}+\delta+i \omega) d\omega\right\|_{L^{\frac{2}{1+2\delta}}_\tau (\R_+)L^{45}_\rho(\B^5_1)}^2
\\
&\lesssim \left(\int_0^1 |f'(s)|(1-s)^{-\frac{1}{6}}s^{\frac{5}{3}} ds\right)^2
\lesssim \|f\|_{W^{1,\frac{2}{1+2\delta}}(\B^5_1)}^2\int_0^1 (1-s)^{-\frac{2}{3}}s^{-\frac{3}{4}} ds
\lesssim\|f\|_{W^{1,\frac{2}{1+2\delta}}(\B^5_1)}^2.
\end{align*}
To bound the remaining terms, one integrates by parts once more and uses similar reasoning.
Hence,
\begin{align*}
\|S_{10}(\tau)f\|_{L^{\frac{2}{1+2\delta}}_\tau(\R_+)L^{45}(\B^5_1)}\lesssim \|f\|_{W^{1,\frac{2}{1+2\delta}}(\B^5_1)}.
\end{align*}
The second estimate on $S_{10}$ then follows from similar reasoning and we move to $\dot S_{10}$.
To derive the stated estimates on $\dot{S}_{10}$, we again first take a look at $B_{10}(f)$. Integrating by parts once again yields
\begin{align*}
B_{10}(f)(\rho,\lambda)&=(1-\rho^2)^{-\frac{\lambda}{2}}\chi_\lambda(\rho)\int_\rho^1 f'(s)
\\
&\quad \times\frac{(1-\chi_\lambda(s))\O(\rho^0s^2\langle\omega\rangle^0)[1+\O(s^{-1}(1-s)\langle\omega\rangle^{-1})]}{(1-s)^{-\lambda}}\beta_4(\rho,s,\lambda)ds 
\\
&=
(1-\rho^2)^{-\frac{\lambda}{2}}\chi_\lambda(\rho)f'(\rho)(1-\chi_\lambda(\rho))\O(\rho^2\langle\omega\rangle^{-1})[1+\O(\rho^{-1}(1-\rho)\langle\omega\rangle^{-1})]
\\
&\quad\times(1-\rho)^{1+\lambda}\beta_1(\rho,\rho,\lambda) +(1-\rho^2)^{-\frac{\lambda}{2}}\chi_\lambda(\rho)\int_\rho^1 (1-s)^{1+\lambda}
\\
&\quad \times \partial_s[f'(s)(1-\chi_\lambda(s))\O(\rho^0s^2\langle\omega\rangle^0)[1+\O(s^{-1}(1-s)\langle\omega\rangle^{-1})]\beta_4(\rho,s,\lambda)]ds 
\\
&:=B_{10}^1(f)(\rho,\lambda)+B_{10}^2(f)(\rho)(\lambda).
\end{align*}
Now, by using that $\chi_\lambda(\rho) \O(\rho^2\langle\omega\rangle^{-1})=\chi_\lambda(\rho) \O(\rho\langle\omega\rangle^{-2})$ and Lemma \ref{osci3}, one establishes the estimate
\begin{align*}
\left|\int_\R \omega e^{i \omega\tau}B_{10}^1(f)(\rho,-\tfrac{1}{2}+\delta+i \omega) d\omega\right|\lesssim \langle\tau\rangle^{-2}\rho^2|f'(\rho)|
\end{align*}
from which one concludes the desired bounds by already exhibited means. 
Moreover, the remaining kernels can be bounded by implementing
essentially the same strategies that we used for $j=9, 10$ in Lemma \ref{lem:Gbounds2} and we conclude this proof.
\end{proof}
These last couple of estimates now add together to our next set of Strichartz estimates.
\begin{prop}\label{prop:Strichartz3}
The difference of the semigroups $\Sf$ and $\Sf_0$ satisfies the Strichartz estimates
\begin{align*}
\|e^{(\frac{1}{2}-\delta)\tau}[(\Sf(\tau)-\Sf_0(\tau))(\I-\Qf)(\I-\Pf)\ff]_1\|_{L^2_\tau(\R_+)L^{45}(\B^5_1)}\lesssim \|(\I-\Qf)\ff\|_{W^{2,\frac{2}{1+2\delta}}\times W^{1,\frac{2}{1+2\delta}}(\B^5_1)}
\end{align*}
and
\begin{align*}
\|e^{(\frac{1}{2}-\delta)\tau}[(\Sf(\tau)-\Sf_0(\tau))(\I-\Qf)(\I-\Pf)\ff]_1\|_{L^\infty_\tau(\R_+)L^{10}(\B^5_1)}\lesssim \|(\I-\Qf)\ff\|_{W^{2,\frac{2}{1+2\delta}}\times W^{1,\frac{2}{1+2\delta}}(\B^5_1)}
\end{align*}
for all $\ff \in C^\infty\times C^\infty(\overline{\B^5_1})$.
\end{prop}
\subsection{Even more estimates}
Unfortunately, we still need one  more estimate at the $W^{2,\frac{2}{1+2\delta}}\times W^{1,\frac{2}{1+2\delta}}$ level, which is of the form 
\begin{align*}
\|e^{(\frac{1}{2}-\delta)\tau}[(\Sf(\tau)-\Sf_0(\tau))(\I-\Qf)(\I-\Pf)\ff]_1\|_{L^6_\tau(\R_+)W^{1,\frac{9}{2}}(\B^5_1)}\lesssim \|\ff\|_{W^{2,\frac{2}{1+2\delta}}\times W^{1,\frac{2}{1+2\delta}}(\B^5_1)}.
\end{align*}
As done above we use a variant of Lemma \ref{lem:interchange} to reduce the problem to estimating 
\begin{align*}
\int_\R e^{i\omega \tau} \partial_\rho[\mathcal{R}(F_\lambda)(\rho,\lambda)-\mathcal{R}_{\mathrm{f}}(F_\lambda)(\rho,\lambda)] d\omega
\end{align*}
and 
\begin{align*}
\int_\R e^{i\omega \tau} \omega \partial_\rho[\mathcal{R}(F_\lambda)(\rho,\lambda)-\mathcal{R}_{\mathrm{f}}(F_\lambda)(\rho,\lambda)] d\omega
\end{align*}
with $\lambda=-\frac{1}{2}+\delta+i\omega$.
For this we remark that \begin{align*}
\partial_\rho \mathcal{R}(f)(\rho,\lambda):=&\partial_\rho u_0(\rho,\lambda)[b_\lambda(f)+U_1(\rho,\lambda)f(\rho)]+\partial_\rho u_0(\rho,\lambda)\int_\rho^1U_1(s,\lambda)f'(s) ds 
\\
&-\partial_\rho u_1(\rho,\lambda) U_0(\rho,\lambda)f(\rho)+\partial_\rho u_1(\rho,\lambda)\int_0^\rho U_0(s,\lambda)f'(s)d s.
\end{align*}
We kick off this round of estimates by first looking at
$$ W_1'(f)(\rho,\lambda):=b_\lambda(f)\partial_\rho u_0(\rho,\lambda) -b_{\mathrm{f}_\lambda}(f)\partial_\rho u_{\mathrm{f}_0}(\rho,\lambda). $$
\begin{lem}
Let $\Re \lambda=-\frac12+\delta$. Then we can decompose $ W_1'(f)(\rho,\lambda)$ as
$$ W'_1(f)(\rho,\lambda)=f(1)\sum_{j=1}^3 H'_j(\rho,\lambda)$$ where
\begin{align*}
H'_1(\rho,\lambda):&=\chi_\lambda(\rho)(1-\rho^2)^{-\lambda}\O(\rho^{-1}\langle\omega\rangle^{-2})
\\
H'_2(\rho,\lambda):&=[(1+\rho)^{-1}\langle\omega\rangle-2\rho^{-1}]H_2(\rho,\lambda)
+(1-\chi_\lambda(\rho))\rho^{-2}(1+\rho)^{1-\lambda}
\\
&\times \O(\rho^{-2}(1-\rho)^{-1}\langle\omega\rangle^{-1})\left[\O(\langle\omega\rangle^{-4})+(1-\rho)\O(\langle\omega\rangle^{-5})+\O(\rho^{-1}(1-\rho)^2\langle\omega\rangle^{-5})\right]
\\
&\quad+\widetilde{H}_2(\rho,\lambda)
\\
H'_3(\rho,\lambda):&=[(1-\rho)^{-1}\langle\omega\rangle-2\rho^{-1}]H_3(\rho,\lambda)
+(1-\chi_\lambda(\rho))\rho^{-2}(1-\rho)^{1-\lambda}
\\
&\times\O(\rho^{-2}(1-\rho)^{-1}\langle\omega\rangle^{-1})\left[\O(\langle\omega\rangle^{-4})+(1-\rho)\O(\langle\omega\rangle^{-5})+\O(\rho^{-1}(1-\rho)^2\langle\omega\rangle^{-5})\right]
\\
&\quad+\widetilde{H}_3(\rho,\lambda)
\end{align*}
where $\widetilde{H}_j(\rho,\lambda)$ are the terms we obtain when a $\rho$-derivative hits the perturbative terms
$$
(1-\rho)\O(\langle\omega\rangle^{-5})+\O(\rho^{-1}(1-\rho)^2\langle\omega\rangle^{-5}).
$$
\end{lem}
\begin{proof}
This follows immediately by differentiating $W_1$ and noting that the derivatives which hit cutoffs cancel each other. 
\end{proof}
Proceeding, we set 
\begin{align*}
S_{j}'(\tau)f(\rho)=\int_\R e^{i \omega\tau}f(1)H_{j}(\rho,-\tfrac{1}{2}+\delta+i \omega) d\omega
\end{align*}
and
\begin{align*}
\dot{S}_{j}'(\tau)f(\rho)=\int_\R \omega e^{i \omega\tau}f(1)H_{j}(\rho,-\tfrac{1}{2}+\delta+i \omega) d\omega
\end{align*}
for $j=1,2,3$ and $f \in C^\infty(\overline{\B^5_1})$.
\begin{lem}
The estimates
\begin{align*} 
\|S_{j}'(\tau)f\|_{L^6_\tau(\R_+)L^{\frac{9}{2}}(\B^5_1)}\lesssim \|f\|_{W^{1,\frac{2}{1+2\delta}}(\B^5_1)}
\end{align*}
and 
\begin{align*}
\|\dot{S}_{j}'(\tau)f\|_{L^6_\tau(\R_+)L^{\frac{9}{2}}(\B^5_1)}\lesssim \|f\|_{W^{2,\frac{2}{1+2\delta}}(\B^5_1)}
\end{align*}
hold for $j=1,2,3$ and $f\in C^\infty(\overline{\B^5_1})$.
\end{lem}

\begin{proof}
In essence $H'_j$ differs from $H_j$ by a loss of either one power in $\rho$ or one power of decay in $\omega$.
But since $|.|^{-1}\in L^{\frac{9}{2}}(\B^5_1)$, this loss can be
compensated and
the claimed estimates follow just as the ones established in Lemma \ref{lem:Hbounds1}.
\end{proof}

\begin{lem}
Let $\Re \lambda=-\frac12+\delta$. Then we can decompose
\begin{align*}
W'_2(f)(\rho,\lambda)&:=\partial_\rho u_0(\rho,\lambda)U_1(\rho,\lambda)f(\rho)-\partial_\rho u_{\mathrm{f}_0}(\rho,\lambda)U_{\mathrm{f}_1}(\rho,\lambda)f(\rho)
\\
&\quad-\partial_\rho u_1(\rho,\lambda) U_0(\rho,\lambda)f(\rho)
+\partial_\rho u_{\mathrm{f}_1}(\rho,\lambda) U_{\mathrm{f}_0}(\rho,\lambda)f(\rho)
\end{align*}
as
\begin{align*}
W'_2(f)(\rho,\lambda)=f(\rho)\sum_{j=4}^{8}H_j(\rho,\lambda)
\end{align*}
where
\begin{align*}
H'_4(\rho,\lambda)&=(1-\rho^2)^{-\frac{\lambda}{2}}\chi_\lambda(\rho)\int_0^\rho\frac{\O(\rho^{-1}s\langle\omega\rangle^{-1})}{(1-s^2)^{1-\frac{\lambda}{2}}} ds
\\
H'_5(\rho,\lambda)&=\left[\frac{1-\lambda}{1+\rho}-2\rho^{-1}\right]H_5(\rho,\lambda)
+(1-\chi_\lambda(\rho))\rho^{-2}(1+\rho)^{1-\lambda}\O(\rho^{-2}(1-\rho)^0\langle\omega\rangle^{-1}) 
\\
&\quad\times \int_0^\rho\frac{\chi_\lambda(s)\O(s\langle\omega\rangle^{-2})}{(1-s^2)^{1-\frac{\lambda}{2}}} \beta_2(\rho,s,\lambda) ds+\widetilde{H}_5(\rho,\lambda)
\\
H'_6(\rho,\lambda)&=\left[-\frac{1-\lambda}{1-\rho}-2\rho^{-1}\right]H_6(\rho,\lambda)
+(1-\chi_\lambda(\rho))\rho^{-2}(1-\rho)^{1-\lambda}\O(\rho^{-2}(1-\rho)^0\langle\omega\rangle^{-1}) 
\\
&\quad \times\int_0^\rho\frac{\chi_\lambda(s)\O(s\langle\omega\rangle^{-2})}{(1-s^2)^{1-\frac{\lambda}{2}}} \beta_3(\rho,s,\lambda) ds+\widetilde{H}_6(\rho,\lambda)
\\
H'_7(\rho,\lambda)&=\left[-\frac{1-\lambda}{1-\rho}-2\rho^{-1}\right]H_7(\rho,\lambda)+(1-\chi_\lambda(\rho))\rho^{-2}(1-\rho)^{1-\lambda}\O(\rho^{-2}(1-\rho)^0\langle\omega\rangle^{-1})
\\
&\quad\times\int_0^\rho\frac{s^2(1-\chi_\lambda(s))[1+\O(s^{-1}(1-s)\langle\omega\rangle^{-1})]\gamma_3(\rho,s,\lambda)}{2(1-\lambda)(1-s)^{1-\lambda}}ds+\widetilde{H}_7(\rho,\lambda)
\\
H'_8(\rho,\lambda)&=\left[\frac{1-\lambda}{1+\rho}-2\rho^{-1}\right]H_8(\rho,\lambda)+(1-\chi_\lambda(\rho))\rho^{-2}(1+\rho)^{1-\lambda}\O(\rho^{-1}(1-\rho)^0\langle\omega\rangle^{-1})
\\
&\quad\times\int_0^\rho\frac{s^2(1-\chi_\lambda(s))[1+\O(s^{-1}(1-s)\langle\omega\rangle^{-1})]\gamma_4(\rho,s,\lambda)}{2(1-\lambda)(1+s)^{1-\lambda}}ds+\widetilde{H}_8(\rho,\lambda)
\end{align*}
with $\beta_j$ and $\gamma_j$ as in Lemma \ref{lem:decomp1} and where $\widetilde{H}_j(\rho,\lambda)$ are the terms we obtain when $\partial_\rho$ hits either $\beta_j$ or $\gamma_j$. 
\end{lem}
Again, we define operators corresponding to the kernels $H'_j$ as 
\begin{align*}
S_j'(\tau)f(\rho):=\lim_{N\to \infty} \int_{-N}^N  e^{i\omega\tau} f(\rho)H_j(\rho,-\tfrac{1}{2}+\delta+i\omega) d\omega
\end{align*}
and
\begin{align*}
\dot{S}_j'(\tau)f(\rho):=\lim_{N\to \infty} \int_{-N}^N  \omega e^{i\omega\tau}  f(\rho)H_j(\rho,-\tfrac{1}{2}+\delta+i \omega) d\omega
\end{align*}
for $j=4,\dots,8$ and $f\in C^\infty(\overline{\B^5_1})$.
\begin{lem}
The estimates
\begin{align*} 
\|S_{j}'(\tau)f\|_{L^6_\tau(\R_+)L^{\frac{9}{2}}(\B^5_1)}\lesssim \|f\|_{W^{1,\frac{2}{1+2\delta}}(\B^5_1)}
\end{align*}
and 
\begin{align*}
\|\dot{S}_{j}'(\tau)f\|_{L^6_\tau(\R_+)L^{\frac{9}{2}}(\B^5_1)}\lesssim \|f\|_{W^{2,\frac{2}{1+2\delta}}(\B^5_1)}
\end{align*}
hold for $j=4,\dots,8$ and $f\in C^\infty(\overline{\B^5_1})$.
\end{lem}
\begin{proof}
The Lemma follows by adapting the proof of Lemma \ref{lem:Hbounds2} slightly.
\end{proof}
Lastly, we come to
\begin{align*}
W'_3(f)(\rho,\lambda)&:=\partial_\rho u_0(\rho,\lambda)\int_\rho^1 U_1(s,\lambda)f'(s) ds -\partial_\rho u_{\mathrm{f}_0}(\rho,\lambda)\int_\rho^1 U_{\mathrm{f}_1}(s,\lambda)f'(s) ds 
\\
&\quad+\partial_\rho u_1(\rho,\lambda)\int_0^\rho U_0(s,\lambda)f'(s)d s-\partial_\rho u_{\mathrm{f}_1}(\rho,\lambda)\int_0^\rho U_{\mathrm{f}_0}(s,\lambda)f'(s)d s.
\end{align*} 
\begin{lem}
Let $\Re \lambda=-\frac12+\delta$. Then we can decompose 
$W'_3(\rho,\lambda)$ as
\begin{equation*}
W'_3(f)(\rho,\lambda)=\sum_{j=9}^{18} H'_j(f)(\rho,\lambda)
\end{equation*} 
where
\begin{align*}
H'_9(f)(\rho,\lambda):&=(1-\rho^2)^{-\frac{\lambda}{2}}\chi_\lambda(\rho)\int_\rho^1 f'(s)\int_0^s\frac{\chi_\lambda(t)\O(\rho^{-1}t\langle\omega\rangle^{-1})}{(1-t^2)^{1-\frac{\lambda}{2}}} dt ds
\\
H'_{10}(f)(\rho,\lambda):&=(1-\rho^2)^{-\frac{\lambda}{2}}\chi_\lambda(\rho)\int_\rho^1 f'(s)\int_0^s(1-\chi_\lambda(t))
\\
&\quad \times\frac{\O(\rho^{-1}t^2\langle\omega\rangle)[1+\O(t^{-1}(1-t)\langle\omega\rangle^{-1})]}{(1-t)^{1-\lambda}}\beta_4(\rho,t,\lambda) dt ds +\widetilde{H}_{10}(f)(\rho,\lambda)
\\
H'_{11}(f) (\rho,\lambda):&=\left[-\frac{1-\lambda}{1-\rho}-2\rho^{-1}\right]H_{11}(f)(\rho,\lambda)+(1-\chi_\lambda(\rho))\rho^{-2}(1-\rho)^{1-\lambda}
\\
&\quad\times\O(\rho^{-2}(1-\rho)^0\langle\omega\rangle^{-1})\int_\rho^1 f'(s) \int_0^s\frac{\chi_\lambda(t)\O(t\langle\omega\rangle^{-2})}{(1-t^2)^{1-\frac{\lambda}{2}}}\beta_5(t,\rho,\lambda) dt ds 
\\
&\quad+\widetilde{H}_{11}(f)(\rho,\lambda)
\end{align*}
and
\begin{align*}
H'_{12}(f) (\rho,\lambda):&=\left[\frac{1-\lambda}{1+\rho}-2\rho^{-1}\right]H_{12}(f)(\rho,\lambda)+(1-\chi_\lambda(\rho))\rho^{-2}(1+\rho)^{1-\lambda}
\\
&\quad \times\O(\rho^{-2}(1-\rho)^0\langle\omega\rangle^{-1})\int_\rho^1 f'(s) \int_0^s\frac{\chi_\lambda(t)\O(t\langle\omega\rangle^{-2})}{(1-t^2)^{1-\frac{\lambda}{2}}}\beta_6(t,\rho,\lambda) dt ds
\\
&\quad+\widetilde{H}_{12}(f)(\rho,\lambda)
\\
H'_{13}(f)(\rho,\lambda):&=\left[-\frac{1-\lambda}{1-\rho}-2\rho^{-1}\right]H_{13}(f)(\rho,\lambda)+(1-\chi_\lambda(\rho))\rho^{-2}(1-\rho)^{1-\lambda}
\\
&\quad\times\O(\rho^{-2}(1-\rho)^0\langle\omega\rangle^{-1})\int_\rho^1 f'(s)
\\
&\quad \times\int_0^s\frac{t^2(1-\chi_\lambda(t))[1+\O(t^{-1}(1-t)\langle\omega\rangle^{-1})]\gamma_3(\rho,t,\lambda)}{2(1-\lambda)(1-t)^{1-\lambda}}dt ds +\widetilde{H}_{13}(f)(\rho,\lambda)
\\
H'_{14}(f)(\rho,\lambda):&=\left[\frac{1-\lambda}{1+\rho}-2\rho^{-1}\right]H_{14}(f)(\rho,\lambda)+(1-\chi_\lambda(\rho))\rho^{-2}(1+\rho)^{1-\lambda}
\\
&\quad\times\O(\rho^{-2}(1-\rho)^0\langle\omega\rangle^{-1})\int_\rho^1 f'(s)
\\
&\quad\times\int_0^s\frac{t^2(1-\chi_\lambda(t))[1+\O(t^{-1}(1-t)\langle\omega\rangle^{-1})]\gamma_4(t,\rho,\lambda)}{2(1-\lambda)(1-t)^{1-\lambda}}dt ds
+\widetilde{H}_{14}(f)(\rho,\lambda)
\end{align*}
and
\begin{align*}
H'_{15}(f)(\rho,\lambda):&=(1-\rho^2)^{-\frac{\lambda}{2}}\chi_\lambda(\rho)\int_0^\rho f'(s)\int_0^s\frac{\O(\rho^{-1}t\langle\omega\rangle^{-1})}{(1-t^2)^{1-\frac{\lambda}{2}}} dt ds
\\
H'_{16}(f) (\rho,\lambda):&=\left[\frac{1-\lambda}{1+\rho}-2\rho^{-1}\right]H_{16}(f)(\rho,\lambda)+(1-\chi_\lambda(\rho))\rho^{-2}(1+\rho)^{1-\lambda}
\\
&\quad\times\O(\rho^{-2}(1-\rho)^0\langle\omega\rangle^{-1})\int_0^\rho f'(s) 
\\
&\quad \times \int_0^s\frac{\chi_\lambda(t)\O(t\langle\omega\rangle^{-1})}{(1-t^2)^{1-\frac{\lambda}{2}}}\beta_7(t,\rho,\lambda) dt ds+\widetilde{H}_{16}(f)(\rho,\lambda)
\\
H'_{17}(f)(\rho,\lambda):&=\left[\frac{1-\lambda}{1+\rho}-2\rho^{-1}\right]H_{17}(f)(\rho,\lambda)+(1-\chi_\lambda(\rho))\rho^{-2}(1+\rho)^{1-\lambda}
\\
&\quad\times\O(\rho^{-2}(1-\rho)^0\langle\omega\rangle^{-1})\int_0^\rho f'(s)
\\
&\quad \times \int_0^s\frac{t^2(1-\chi_\lambda(t))[1+\O(t^{-1}(1-t)\langle\omega\rangle^{-1})]\gamma_5(\rho,t,\lambda)}{2(1-\lambda)(1-t)^{1-\lambda}}dt ds +\widetilde{H}_{17}(f)(\rho,\lambda)
\\
H'_{18}(f)(\rho,\lambda):&=\left[\frac{1-\lambda}{1+\rho}-2\rho^{-1}\right]H_{18}(f)(\rho,\lambda)+(1-\chi_\lambda(\rho))\rho^{-2}(1+\rho)^{1-\lambda}
\\
&\quad\times\O(\rho^{-2}(1-\rho)^0\langle\omega\rangle^{-1})\int_0^\rho f'(s)
\\
&\quad \times \int_0^s \frac{t^2(1-\chi_\lambda(t))[1+\O(t^{-1}(1-t)\langle\omega\rangle^{-1})]\gamma_6(\rho,t,\lambda)}{2(1-\lambda)(1+t)^{1-\lambda}}dt ds +\widetilde{H}_{18}(f)(\rho,\lambda)
\end{align*}
with $\beta_j$ and $\gamma_j$ as in Lemma \ref{lem:decomp1} and where $\widetilde{H}_{j}(f)(\rho,\lambda)$ are terms obtained when a $\rho$ derivative hits $\beta_j$ or $\gamma_j$.
\end{lem}
One last time we define operators $S_j'$ and $\dot{S}'_j$
\begin{align*}
S_{j}'(\tau)f(\rho):=\lim_{N\to \infty}\int_{-N}^N e^{i \omega\tau}H'_{j}(f)(\rho,-\tfrac{1}{2}+\delta+i \omega) d\omega
\end{align*}
and
\begin{align*}
\dot{S}'_{j}(\tau)f(\rho):=\lim_{N\to \infty}\int_{-N}^N  \omega e^{i \omega\tau}H'_{j}(f)(\rho,-\tfrac{1}{2}+\delta+i \omega) d\omega
\end{align*}
for $j=9,\dots,18$ and $f \in C^\infty(\overline{\B^5_1})$.
\begin{lem}
The estimates
\begin{align*} 
\|S_{j}'(\tau)f\|_{L^6_\tau(\R_+)L^{\frac{9}{2}}(\B^5_1)}\lesssim \|f\|_{W^{1,\frac{2}{1+2\delta}}(\B^5_1)}
\end{align*}
and
\begin{align*}
\|\dot{S}_{j}'(\tau)f\|_{L^6_\tau(\R_+)L^{\frac{9}{2}}(\B^5_1)}\lesssim \|f\|_{W^{2,\frac{2}{1+2\delta}}(\B^5_1)}
\end{align*}
hold for $j=9,\dots,18$ and $f\in C^\infty(\overline{\B^5_1})$.
\end{lem}
\begin{proof}
The estimates can be established by adapting the procedures used in the proof of Lemma \ref{lem:Hbounds3} in a straightforward way.
\end{proof}
\begin{prop}\label{prop:Strichartz4}
The difference of $\Sf$ and $\Sf_0$ satisfies
\begin{align*}
\|e^{(\frac{1}{2}-\delta)\tau}[(\Sf(\tau)-\Sf_0(\tau))(\I-\Qf)(\I-\Pf)\ff]_1\|_{L^6_\tau(\R_+)W^{1,\frac{9}{2}}(\B^5_1)}\lesssim \|(\I-\Qf)\ff\|_{W^{2,\frac{2}{1+2\delta}}\times W^{1,\frac{2}{1+2\delta}}(\B^5_1)}
\end{align*}
for all $\ff\in C^\infty \times C^\infty(\overline{\B^5_1}).$
\end{prop}
We now turn to interpolating the previously derived Strichartz estimates to obtain estimates on the $H^{\frac{3}{2}}\times H^{\frac{1}{2}}$ level.
For the notations and conventions appearing in the context of interpolation throughout the following proof we refer the reader to the Appendix and \cite{BerLof12}.
We also recall that we constructed a subset $X\subset H^3\times H^2(\B^5_1)$ which lies dense in $\mathcal{H}:=H^{\frac{3}{2}}\times H^{\frac{1}{2}}(\B^5_1)$ such that the spectral projection $\Qf$ agrees on $X$ with a bounded linear operator $\widehat{\Qf}:\H\to \H$. 
\begin{prop}\label{prop: Strichartzfinal}
Let $p\in [2,\infty]$ and $q\in [5,10]$ be such that $\frac{1}{p}+\frac{5}{q}=1$. Then, the semigroup $\Sf$ satisfies the Strichartz estimates
\begin{align*}
\|[\Sf(\tau)(\I-\Pf)\ff]_1\|_{L^p_\tau(\R_+)L^q(\B^5_1)}\lesssim \|\ff\|_{\H}
\end{align*}
for all $\ff \in \mathcal{H}$.
Furthermore, also the inhomogeneous estimate
\begin{align*}
\left\|\int_0^\tau\left[\Sf(\tau-\sigma)(\I-\Pf)\hfh(\sigma)\right]_1 d\sigma\right\|_{L^p_\tau(I)L^q(\B^5_1)}\lesssim \|\hfh\|_{L^1(I)\mathcal{H}}
\end{align*}
holds for all $\hfh \in L^1(\R_+,\mathcal{H})$ and all intervals $I
\subset [0,\infty)$ containing $0$.
\end{prop}
\begin{proof}
We start by setting
\begin{align*}
\|u\|_{L^p(\R_+, e^{a\tau}d\tau)L^{q}(\B^5_1)}^p:=\int_{\R_+}\|u(\tau)\|_{L^q(\B^5_1)}^p e^{a\tau} d\tau.
\end{align*} 
for $a\in \R$ and $\widetilde \Sf (\tau)= \Sf(\tau)-\Sf_0(\tau)$.
Then, by a density argument we have that 
\begin{align*}
&\quad\|[\widetilde \Sf(\tau)(\I-\Qf)(\I-\Pf)\ff]_1\|_{L^{\frac{2}{1-2\delta}}(\R_+,e^{-\frac{1+2\delta}{1-2\delta}\tau}d\tau)L^{\frac{45}{8}}(\B^5_1)}
\\
&=\|e^{-(\frac{1}{2}+\delta)\tau}[ \widetilde \Sf(\tau)(\I-\Qf)(\I-\Pf)\ff]_1\|_{L^{\frac{2}{1-2\delta}}_\tau(\R_+)L^{\frac{45}{8}}(\B^5_1)}
\\
&\lesssim \|(\I-\Qf)\ff\|_{W^{1,\frac{2}{1-2\delta}}\times L^{\frac{2}{1-2\delta}}(\B^5_1)}
\end{align*} 
for all $\ff \in X$ thanks to Proposition \ref{prop: Strichartz1}.
Similarly, from Proposition \ref{prop:Strichartz3} we know that
\begin{align*}
&\quad \|[ \widetilde \Sf(\tau)(\I-\Qf)(\I-\Pf)\ff]_1\|_{L^{\frac{2}{1+2\delta }}(\R_+, e^{\frac{1}{1+2\delta}\tau}d\tau)L^{45}(\B^5_1)}
\\
&=\|e^{(\frac{1}{2}-\delta)\tau}[\widetilde\Sf(\tau)(\I-\Qf)(\I-\Pf)\ff]_1\|_{L^{\frac{2}{1+2\delta }}_\tau(\R_+)L^{45}(\B^5_1)}
\\
&\lesssim \|(\I-\Qf)\ff\|_{W^{2,\frac{2}{1+2\delta}}\times W^{1,\frac{2}{1+2\delta}}(\B^5_1)},
\end{align*}  Hence, by invoking Proposition \ref{prop:interpolation} and using that 
$$\mathcal{H}=H^{\frac32}  \times H^{\frac12}(\B^5_1)=(W^{2,\frac{2}{1+2\delta}}\times W^{1,\frac{2}{1+2\delta}}(\B^5_1),W^{1,\frac{2}{1-2\delta}}\times L^{\frac{2}{1-2\delta}}(\B^5_1))_{[\frac{1}{2}]},$$ see \cite[p.~317, Subsection 4.3.1.1, Theorem 1]{Tri95}, we conclude that
\begin{align}\label{esti:1}
\|[\widetilde \Sf(\tau)(\I-\Qf)(\I-\Pf)\ff]_1\|_{L^{2}(\R_+)L^{10}(\B^5_1)} \lesssim \|(\I-\Qf)\ff\|_{\mathcal{H}}.
\end{align} 
In addition, since 
\begin{align*}
\|e^{-(\frac{1}{2}-\delta)\tau}[\widetilde\Sf(\tau)(\I-\Qf)(\I-\Pf)\ff]_1\|_{L^{\frac{10}{3}}(\B^5_1)}&\lesssim \|(\I-\Qf)\ff\|_{W^{1,\frac{2}{1+2\delta}}\times L^{\frac{2}{1+2\delta}}(\B^5_1)}
\end{align*}
and
\begin{align}\label{esti:2}
\|e^{(\frac{1}{2}-\delta)\tau}[\widetilde\Sf(\tau)(\I-\Qf)(\I-\Pf)\ff]_1\|_{L^{10}(\B^5_1)}&\lesssim \|(\I-\Qf)\ff\|_{W^{2,\frac{2}{1+2\delta}}\times W^{1,\frac{2}{1+2\delta}}(\B^5_1)}
\end{align}
for all $\tau \geq 0$ interpolating yields
\begin{align*}
\|[\widetilde \Sf(\tau)(\I-\Qf)(\I-\Pf)\ff]_1\|_{L^{\infty}(\R_+)L^{5}(\B^5_1)} \lesssim \|(\I-\Qf)\ff\|_{\mathcal{H}}
\end{align*} 
 for all $\ff\in X$. Hence, elementary interpolation between  \eqref{esti:1} and \eqref{esti:2} combined with the estimates on $\Sf_0$ in Lemma \ref{lem:freeStrichartz} yields
\begin{align*}
\|[\Sf(\tau)(\I-\Qf)(\I-\Pf)\ff]_1\|_{L^{p}(\R_+)L^{q}(\B^5_1)} \lesssim \|(\I-\Qf)\ff\|_{\mathcal{H}}
\end{align*}
$p\in [2,\infty]$ and $q\in [5,10]$ are such that $\frac{1}{p}+\frac{5}{q}=1$. 
Furthermore, by construction $\Qf$ agrees with a bounded linear operator $\widehat{\Qf}:\mathcal{H}\to \mathcal{H}$ on $X$ and so
\begin{align*}
\|(\I-\Qf)\ff\|_{\mathcal{H}}=\|(\I-\widehat \Qf)\ff\|_{\mathcal{H}}\lesssim \|\ff\|_{\mathcal{H}}
\end{align*}
for all $\ff\in X$. Next, we turn to $\Sf(\tau)\Qf$.
From the Sobolev embedding 
$H^2(\B^5_1)\hookrightarrow L^{10}(\B^5_1), $ we deduce that 
\begin{align*}
\|[\Sf(\tau)\Qf(\I-\Pf)\ff]_1\|_{L^p_\tau(\R_+)L^q(\B^5_1)} \lesssim \|[\Sf(\tau)\Qf(\I-\Pf)\ff]_1\|_{L^p_\tau(\R_+)H^2(\B^5_1)}
\end{align*}
for all admissible pairs $(p,q)$. Given that the range of $\Qf$ is contained in the union of finitely many generalized eigenspaces corresponding to eigenvalues which all have negative real parts, we infer the  existence of an $\varepsilon>0$ such that
\begin{align*}
\|[\Sf(\tau)\Qf(\I-\Pf)\ff]_1\|_{H^2(\B^5_1)}\lesssim e^{-\varepsilon \tau} \|\Qf\ff\|_{H^2\times H^1 (\B^5_1)}
\end{align*} 
on $X$. Moreover, since the range of $\Qf$ is finite-dimensional, we see that 
\begin{align*}
 \|\Qf\ff\|_{H^2\times H^1 (\B^5_1)}&\lesssim  \|\Qf\ff\|_{\mathcal{H}}=\|\widehat \Qf\ff\|_{\mathcal{H}}\lesssim \|\ff\|_{\mathcal{H}}.
\end{align*} 
for all $\ff\in X$.
Thus, the estimate 
\begin{align*}
\|\Sf(\tau)[(\I-\Pf)\ff]_1\|_{L^p(\R_+)L^q(\B^5_1)} \lesssim \|\ff\|_{\mathcal{H}}
\end{align*}
holds for all claimed pairs $(p,q)$ and all $\ff \in X$ and by density
for all $\ff\in \mathcal H$.
 For the inhomogeneous estimates one uses Minkowski's inequality as in the proof of Lemma 3.7 in \cite{DonWal22}.
\end{proof}
Analogously, one proves Strichartz estimates involving (fractional) derivatives. 
\begin{prop}\label{prop:strichartzfinal2}
The estimates 
\begin{align*}
\|[\Sf(\tau)(\I-\Pf)\ff]_1\|_{L^2_\tau(\R_+)W^{\frac{1}{2},5}(\B^5_1)}\lesssim \|\ff\|_{\H}
\end{align*}
and 
\begin{align*}
\|[\Sf(\tau)(\I-\Pf)\ff]_1\|_{L^6_\tau(\R_+)W^{1,\frac{30}{11}}(\B^5_1)}\lesssim \|\ff\|_{\H}
\end{align*}
hold for all $\ff\in \mathcal{H}$.
Furthermore, also the inhomogeneous estimates
\begin{align*}
\left\|\int_0^\tau\left[\Sf(\tau-\sigma)(\I-\Pf)\hfh(\sigma)\right]_1 d\sigma\right\|_{L^2_\tau(I)W^{\frac{1}{2},5}(\B^5_1)}\lesssim \|\hfh\|_{L^1(I)\mathcal{H}}
\end{align*}
and
\begin{align*}
\left\|\int_0^\tau\left[\Sf(\tau-\sigma)(\I-\Pf)\hfh(\sigma)\right]_1 d\sigma\right\|_{L^6_\tau(I)L^{\frac{30}{11}}(\B^5_1)}\lesssim \|\hfh\|_{L^1(I)\mathcal{H}}
\end{align*}
hold for all $\hfh \in L^1(\R_+,\mathcal{H})$ and all intervals $I\subset [0,\infty)$ containing $0$.
\end{prop}
\begin{proof}
Note that
\begin{align*}
\left(W^{1,\frac{9}{2}}(\B^5_1),L^{\frac{45}{8}}(\B^5_1)\right)_{[\frac{1}{2}]}= W^{\frac{1}{2},5}(\B^5_1).
\end{align*}
and
\begin{align*}
\left(W^{1,\frac{9}{2}}(\B^5_1),W^{1,\frac{45}{23}}(\B^5_1)\right)_{[\frac{1}{2}]}= W^{1,\frac{30}{11}}(\B^5_1)
\end{align*}
thanks to \cite[p.~317, Subsection 4.3.1.1, Theorem 1]{Tri95}.
Consequently, the desired estimates follow from Propositions
\ref{prop:Strichartz2}, \ref{prop:Strichartz4},
\ref{prop:interpolation} and the arguments employed in the proof of
Proposition \ref{prop: Strichartzfinal}.
\end{proof}

\section{Nonlinear Theory}
We now take a closer look at our nonlinearity
\begin{align*}
N(u)(\rho)=\frac{\sin(4\arctan(\rho)+2\rho u(\rho))-2\rho u(\rho)}{\rho^3}-\frac{\sin(4\arctan(\rho))}{\rho^3}+\frac{16}{(1+\rho^2)^2}u(\rho)
\end{align*}
which we recast as
\begin{align*}
N(u)(\rho)&=\underbrace{ -\frac{16(1-\rho^2)}{(1+\rho^2)^2}}_{=:V_N(\rho)}u_1(\rho)^2-4\int_0^{u(\rho)}\cos(4\arctan(\rho)+2\rho t)(u(\rho)-t)^2 dt
\end{align*}
by performing a Taylor expansion.
\begin{lem}
The estimates
\begin{align*}
\|N(u)\|_{H^\frac{1}{2}(\B^5_1)}& \lesssim \|u\|_{L^{10}(\B^5_1)}^2+\|u\|_{L^{5}(\B^5_1)}^3+\|u\|_{L^{\frac{20}{3}}(\B^5_1)}^4
\\
&\quad+\|u\|_{W^{\frac{1}{2},5}(\B^5_1)}\|u\|_{L^{10}(\B^5_1)}+ \|u\|_{W^{1,\frac{30}{11}}(\B^5_1)}\|u\|_{L^{\frac{60}{7}}(\B^5_1)}^2
\end{align*}
and
\begin{align*}
\|N(u)-N(v)\|_{H^\frac{1}{2}(\B^5_1)}&\lesssim \|u-v\|_{L^{10}(\B^5_1)}\left(\|u\|_{L^{10}(\B^5_1)}+\|v\|_{L^{10}(\B^5_1)}+\|u\|_{L^{\frac{20}{3}}(\B^5_1)}^2+\|v\|_{L^{\frac{20}{3}}(\B^5_1)}^2\right)
\\
&\quad+\|u-v\|_{L^{10}(\B^5_1)}\left(\|u\|_{L^6(\B^5_1)}^3+\|v\|_{L^6(\B^5_1)}^3 \right)
\\
&\quad+\|u-v\|_{L^{10}(\B^5_1)}\left(\|u\|_{W^{\frac{1}{2},5}(\B^5_1)}+\|v\|_{W^{\frac{1}{2},5}(\B^5_1)} \right)
\\
&\quad+\|u-v\|_{W^{\frac{1}{2},5}(\B^5_1)}\left(\|u\|_{L^{10}(\B^5_1)}+\|v\|_{L^{10}(\B^5_1)}\right)
\\
&\quad+
\|u-v\|_{W^{1,\frac{30}{11}}(\B^5_1)}\|u\|_{L^{\frac{60}{7}}(\B^5_1)}^2
\\
&\quad+\|u-v\|_{L^{\frac{60}{7}}(\B^5_1)}\|v\|_{W^{1,\frac{30}{11}}(\B^5_1)}\left(\|u\|_{L^{\frac{60}{7}}(\B^5_1)}+\|u\|_{L^{\frac{60}{7}}(\B^5_1)}\right)
\end{align*}
hold for all $u,v\in C^\infty(\overline{\B_1^5})$.
\end{lem}
\begin{proof}
We start off with the easier quadratic term and use the product rule for fractional derivatives twice to compute that
\begin{align*}
\|V_N u^2\|_{H^\frac{1}{2}(\B^5_1)}&\lesssim \|V_N\|_{W^{\frac{1}{2},\frac{10}{3}}(\B^5_1)}\|u^2\|_{L^5(\B^5_1)}+\|V_N\|_{L^5(\B^5_1)}\|u^2\|_{W^{\frac{1}{2},\frac{10}{3}}(\B^5_1)}
\\
&\lesssim \|u\|_{L^{10}(\B^5_1)}^2+\|u\|_{W^{\frac{1}{2},5}(\B^5_1)}\|u\|_{L^{10}(\B^5_1)}.
\end{align*}
For the cubic term, we use the Sobolev inequality $\|.\|_{H^\frac{1}{2}(\B^5_1)}\lesssim \|.\|_{W^{1,\frac{5}{3}}(\B^5_1)}
$ to estimate that
\begin{align*}
&\quad\left\|\int_0^{u(\rho)}\cos(4\arctan(\rho)+2\rho t)(u(\rho)-t)^2 dt\right\|_{H^\frac{1}{2}_\rho(\B^5_1)}
\\
&\lesssim \left\|\int_0^{u(\rho)}\cos(4\arctan(\rho)+2\rho t)(u(\rho)-t)^2 dt\right\|_{W^{1,\frac{5}{3}}_\rho(\B^5_1)}
\\
&\lesssim \|u^2 u'\|_{L^{\frac{5}{3}}(\B^5_1)}+\|u^3\|_{L^{\frac{5}{3}}(\B^5_1)}+\|u^4\|_{L^{\frac{5}{3}}(\B^5_1)}
\\
&\lesssim  \|u'\|_{L^{\frac{30}{11}}(\B^5_1)}\|u\|_{L^{\frac{60}{7}}(\B^5_1)}^2+\|u\|_{L^{5}(\B^5_1)}^3+\|u\|_{L^{\frac{20}{3}}(\B^5_1)}^4.
\end{align*}
To establish local Lipschitz estimates we let $u,v\in C^\infty(\overline{\B^5_1})$ and again first take a look at the easier quadratic term
\begin{align*}
\|V_N (u^2-v^2)\|_{H^\frac{1}{2}(\B^5_1)}
&\lesssim \|V_N\|_{W^{\frac{1}{2},\frac{10}{3}}(\B^5_1)}\|u^2-v^2\|_{L^5(\B^5_1)}+\|V_N\|_{L^5(\B^5_1)}\|u^2-v^2\|_{W^{\frac{1}{2},\frac{10}{3}}(\B^5_1)}
\\
&\lesssim \|(u-v)(u+v)\|_{L^5(\B^5_1)}+\|(u-v)(u+v)\|_{W^{\frac{1}{2},\frac{10}{3}}(\B^5_1)}
\\
&\lesssim \|u-v\|_{L^{10}(\B^5_1)}(\|u\|_{L^{10}(\B^5_1)}+\|v\|_{L^{10}(\B^5_1)})
\\
&\quad+\|u-v\|_{W^{\frac{1}{2},5}(\B^5_1)}(\|u\|_{L^{10}(\B^5_1)}+\|v\|_{L^{10}(\B^5_1)})
\\
&\quad+\|u-v\|_{L^{10}(\B^5_1)}(\|u\|_{W^{\frac{1}{2},5}(\B^5_1)}+\|v\|_{W^{\frac{1}{2},5}(\B^5_1)}).
\end{align*}
Next, consider the function
$n:\R\times [0,1]\to \R$,
\begin{align*}
n(x,\rho):=4\int_0^{x}\cos(4\arctan(\rho)+2\rho t)(x-t)^2 dt
\end{align*}
and note that
\begin{align*}
|\partial_1n(x,\rho)|&\lesssim |x|^2
\\
|\partial_2n(x,\rho)|&\lesssim |x|^4
\\
|\partial_1^2n(x,\rho)|&\lesssim |x|
\\
|\partial_1\partial_2n(x,\rho)|&\lesssim |x|^3.
\end{align*}
Consequently,
\begin{align*}
\|n(u,.)-n(v,.)\|_{L^{\frac{5}{3}}(\B^5_1)}&\lesssim \|(|u|^2+|v|^2)(u-v)\|_{L^{\frac{5}{3}}(\B^5_1)}
\\
&\lesssim \|u-v\|_{L^{10}(\B^5_1)}\left(\|u\|_{L^{4}(\B^5_1)}^2+\|v\|_{L^{4}(\B^5_1)}^2\right),
\end{align*}
as well as
\begin{align*}
\|n(u,.)-n(v,.)\|_{\dot{W}^{1,\frac{5}{3}}(\B^5_1)}&\lesssim \|\partial_2(n(u,.)-n(v,.))\|_{L^{\frac{5}{3}}(\B^5_1)}+\|u'\partial_1n(u,.)-v'\partial_1n(v,.)\|_{L^{\frac{5}{3}}(\B^5_1)}
\\
&:=N_1 +N_2.
\end{align*}
For $N_1$ we obtain
\begin{align*}
N_1\lesssim \|(|u|^3+|v|^3)|u-v|\|_{L^{\frac{5}{3}}(\B^5_1)}\lesssim\|u-v\|_{L^{10}(\B^5_1)}\left(\|u\|_{L^6(\B^5_1)}^3+\|v\|_{L^6(\B^5_1)}^3\right).
\end{align*}
Further,
\begin{align*}
N_2&\lesssim \|(u'-v')\partial_1n(u,.)\|_{L^{\frac{5}{3}}(\B^5_1)}+\|v'(\partial_1n(u,.)-\partial_1n(v,.))\|_{L^{\frac{5}{3}}(\B^5_1)}
\\
&\lesssim \||u'-v'|u^2\|_{L^{\frac{5}{3}}(\B^5_1)}+\||v'|(|u|+|v|)|u-v|\|_{L^{\frac{5}{3}}(\B^5_1)}
\\
&\lesssim \|u'-v'\|_{L^{\frac{30}{11}}(\B^5_1)}\|u\|_{L^{\frac{60}{7}}(\B^5_1)}^2+\|u-v\|_{L^{\frac{60}{7}}(\B^5_1)}\|v'\|_{L^{\frac{30}{11}}(\B^5_1)}\left(\|u\|_{L^{\frac{60}{7}}(\B^5_1)}+\|u\|_{L^{\frac{60}{7}}(\B^5_1)}\right).
\end{align*}
\end{proof}
Motivated by these estimates on the nonlinearity, we define the space $\mathcal{X}$ to be the completion of $C^\infty_c(\R_+\times\overline{\B^5_1})$ with respect to the norm
\begin{align*}
\|\phi\|_{\mathcal{X}}&=\|\phi\|_{L^2(\R_+)L^{10}(\B^5_1)}+\|\phi\|_{L^{\frac{12}{5}}(\R_+)L^{\frac{60}{7}}(\B^5_1)}+\|\phi\|_{L^3(\R_+)L^{\frac{15}{2}}(\B^5_1)}\\
&\quad+\|\phi\|_{L^4(\R_+)L^{\frac{20}{3}}(\B^5_1)}+\|\phi\|_{L^2(\R_+)W^{\frac{1}{2},5}(\B^5_1)}+\|\phi\|_{L^6(\R_+)W^{1,\frac{30}{11}}(\B^5_1)}.
\end{align*}
Moreover, we set 
\begin{align*}
\mathcal{X}_\delta:=\{\phi\in \mathcal{X}: \|\phi\|_{\mathcal{X}}\leq \delta\}
\end{align*}
and for $\uf\in \mathcal{H}$ and $\phi \in C^\infty_c(\R_+\times\overline{\B^5_1})$ we define
\begin{align*}
\K_{\uf}(\phi)(\tau):=\left[\Sf(\tau)\uf+\int_0^\tau \Sf(\tau-\sigma)\Nf((\phi(\sigma),0)) d\sigma-\Cf(\uf,\phi)(\tau)\right]_1
\end{align*}
where the correction term $\Cf$, which we add to suppress the unstable direction induced by the eigenvalue $1$, is given by
\begin{align*}
\Cf(\uf,\phi)(\tau):=\Pf\left(e^{\tau}\uf+\int_0^\infty e^{\tau-\sigma} \Nf((\phi(\sigma),0)) d\sigma \right).
\end{align*}

\begin{lem}\label{lem:localbound}
We have that
$
\K_\uf(\phi)\in \mathcal{X}
$
for all $\uf \in \H$ and all $\phi\in C^\infty_c(\R_+\times\overline{\B^5_1})$. Moreover,
\begin{align*}
\|\K_\uf(\phi)\|_{\X}\lesssim \|\uf\|_{\mathcal{H}}+\|\phi\|_{\X}^2+\|\phi\|_{\X}^4
\end{align*}
for all $\uf \in \H$ and all $\phi\in C^\infty_c(\R_+\times\overline{\B^5_1})$.
\end{lem}
\begin{proof}
We split $\K_\uf(\phi)$ into 
\begin{align*}
\K_\uf(\phi)&=\left[(\I-\Pf)\Sf(\tau)\uf+(\I-\Pf)\int_0^\tau \Sf(\tau-\sigma)\Nf((\phi(\sigma),0)) d\sigma\right]_1
\\
&\quad+\left[\Pf\Sf(\tau)\uf+\Pf\int_0^\tau \Sf(\tau-\sigma)\Nf((\phi(\sigma),0)) d\sigma-\Cf(\uf,\phi)(\tau)\right]_1
\\
&=:(\I-\Pf)\K_\uf(\phi)+\Pf\K_\uf(\phi)
\end{align*} and investigate $(\I-\Pf)\K_\uf(\phi)$ and $\Pf \K_\uf(\phi)$ separately. For the first one we observe that that
\begin{align*}
(\I-\Pf)\K_{\uf}(\phi)(\tau)=\left[\Sf(\tau)(\I-\Pf)\uf+\int_0^\tau \Sf(\tau-\sigma)(\I-\Pf)\Nf((\phi(\sigma),0)) d\sigma\right]_1.
\end{align*}
Hence, we use Propositions \ref{prop: Strichartzfinal} and  \ref{prop:strichartzfinal2} to deduce that
\begin{align*}
\|(\I-\Pf)\K_{\uf}(\phi)\|_{\mathcal{X}} &\lesssim \|\uf\|_{\mathcal{H}}+\left\|\int_0^\tau[\Sf(\tau-\sigma)(\I-\Pf)\Nf(\phi(\sigma),0))]_1 d\sigma\right\|_{\X}
\\
&\lesssim \|\uf\|_{\mathcal{H}}+\int_0^\infty \|N(\phi(\sigma))\|_{H^{\frac12}(\B^5_1)} d \sigma
\\
&\lesssim
\|\uf\|_{\mathcal{H}}+\int_0^\infty \|\phi(\sigma)\|_{W^{1,\frac{30}{11}}(\B^5_1)}\|\phi(\sigma)\|_{L^{\frac{60}{7}}(\B^5_1)}^2+\|\phi(\sigma)\|_{L^{5}(\B^5_1)}^3 d\sigma
\\
&\quad+\int_0^\infty\|\phi(\sigma)\|_{L^{\frac{20}{3}}(\B^5_1)}^4+\|\phi(\sigma)\|_{L^{10}(\B^5_1)}^2+\|\phi(\sigma)\|_{W^{\frac{1}{2},5}(\B^5_1)}\|\phi(\sigma)\|_{L^{10}(\B^5_1)}d\sigma
\\
&\lesssim 
\|\uf\|_{\mathcal{H}}+\|\phi\|_{\X}^2+\|\phi\|_{\X}^4.
\end{align*}
We move on to $\Pf \K_\uf(\phi)$, where we first discern that
\begin{align*}
\Pf \K_{\uf}(\phi)(\tau)=\left[\int_\tau^\infty e^{\tau-\sigma}\Pf\Nf((\phi(\sigma),0)) d\sigma \right]_1.
\end{align*}
We also remark that as  $\Pf$ has rank $1$ there exists a unique $\widetilde{\gf}\in \mathcal{H}$ such that $\Pf \ff=(\ff,\widetilde{\gf})_{\mathcal{H}}\gf$ for all $\ff\in \mathcal{H}$. Hence,
\begin{align*}
\|\Pf \K_{\uf}(\phi)(\tau)\|_{L^p(\B^5_1)}+\|\Pf \K_{\uf}(\phi)(\tau)\|_{W^{\frac{1}{2},5}(\B^5_1)}+\|\Pf \K_{\uf}(\phi)(\tau)\|_{W^{1,\frac{30}{11}}(\B^5_1)}\lesssim \|N(\phi(\tau))\|_{H^\frac{1}{2}(\B^5_1)} 
\end{align*}
for any $2\leq p\leq \infty$. So,
\begin{align*}
\|\Pf \K_{\uf}(\phi)(\tau)\|_{L^p_\tau(\R_+)L^q(\B^5_1)}\lesssim \left\|\int_\tau^\infty e^{\tau-\sigma}\|N(\phi(\sigma))\|_{H^\frac{1}{2}(\B^5_1)} d \sigma\right\|_{L^p_\tau(\R_+)}
\end{align*}
and Young's inequality implies that 
\begin{align*}
\|\Pf \K_{\uf}(\phi)(\tau)\|_{L^p_\tau(\R_+)L^q(\B^5_1)}&\lesssim \|N(\phi)\|_{L^1(\R_+)H^\frac{1}{2}(\B^5_1)}\|1_{(-\infty,0]}(\tau) e^{\tau}\|_{L^p_\tau(\R_+)}
\\
&\lesssim \|\phi\|_{\X}^2+\|\phi\|_{\X}^4.
\end{align*}
As the remaining spacetime norms can be bounded likewise, one obtains the desired estimate
\begin{align*}
\|\Pf \K_{\uf}(\phi)\|_{\X}
&\lesssim \|\phi\|_{\X}^2+\|\phi\|_{\X}^4.
\end{align*}
\end{proof}

\begin{lem}\label{lem:locallip}
The estimate 
\begin{align*}
\|\K_{\uf}(\phi)-\K_\uf(\psi)\|_{\X}\lesssim (\|\phi\|_{\X}+\|\phi\|_{\X}^3+\|\psi\|_{\X}+\|\psi\|_{\X}^3)\|\phi-\psi\|_{\X}
\end{align*}
holds for all $\uf\in \mathcal{H}$ and all $\phi,\psi\in C^\infty_c(\R_+\times \overline{\B^5_1})$.
\end{lem}
\begin{proof}
Invoking Propositions \ref{prop: Strichartzfinal} and \ref{prop:strichartzfinal2} yields
\begin{align*}
\|(\I-\Pf)(\K_{\uf}(\phi)-\K_\uf(\psi))\|_{\X}&\lesssim \int_0^\infty \|N(\phi(\sigma))-N(\psi(\sigma))\|_{H^{\frac12}(\B^5_1)}
\\
&\lesssim
\int_0^\infty \|\phi(\sigma)-\psi(\sigma)\|_{L^{10}(\B^5_1)}\left(\|\phi(\sigma)\|_{L^{10}(\B^5_1)}+\|\psi(\sigma)\|_{L^{10}(\B^5_1)}\right)
\\
&\quad  +\|\phi(\sigma)-\psi(\sigma)\|_{L^{10}(\B^5_1)}\left(\|\phi(\sigma)\|_{L^{\frac{20}{3}}(\B^5_1)}^2+\|\psi(\sigma)\|_{L^{\frac{20}{3}}(\B^5_1)}^2\right)
\\
&\quad+\|\phi(\sigma)-\psi(\sigma)\|_{W^{\frac{1}{2},5}(\B^5_1)}\left(\|\phi(\sigma)\|_{L^{10}(\B^5_1)}+\|\psi(\sigma)\|_{L^{10}(\B^5_1)}\right)
\\
&\quad+\|\phi(\sigma)-\psi(\sigma)\|_{L^{10}(\B^5_1)}
 \left(\|\phi(\sigma)\|_{W^{\frac{1}{2},5}(\B^5_1)}+\|\psi(\sigma)\|_{W^{\frac{1}{2},5}(\B^5_1)}\right)
\\
&\quad+\|\phi(\sigma)-\psi(\sigma)\|_{L^{10}(\B^5_1)}\left(\|\phi(\sigma)\|_{L^6(\B^5_1)}^3+\|\psi(\sigma)\|_{L^6(\B^5_1)}^3\right)
\\
&\quad+
\|\phi(\sigma)-\psi(\sigma)\|_{W^{1,\frac{30}{11}}(\B^5_1)}\|\phi(\sigma)\|_{L^{\frac{60}{7}}(\B^5_1)}^2
\\
&\quad+\|\phi(\sigma)-\psi(\sigma)\|_{L^{\frac{60}{7}}(\B^5_1)}\|\psi(\sigma)\|_{W^{1,\frac{30}{11}}(\B^5_1)}
\\
&\quad\times\left(\|\phi(\sigma)\|_{L^{\frac{60}{7}}(\B^5_1)}+\|\psi(\sigma)\|_{L^{\frac{60}{7}}(\B^5_1)}\right) d\sigma
\\
&\lesssim 
\|\phi-\psi\|_{\X}\left(\|\phi\|_{\X}+\|\psi\|_{\X}+\|\phi\|_{\X}^3+\|\psi\|_{\X}^3\right).
\end{align*}
Estimating $\Pf(\K_\uf(\phi)-\K_\uf(\psi))$ can be done by employing the same strategy as in the proof of Lemma \ref{lem:localbound}.
\end{proof}
The last two lemmas combined with an application of the contraction mapping principle yield the next result.
\begin{lem}\label{lem:ex1}
For any $\uf \in \mathcal{H}$ fixed, the operator $\K_\uf$ extends to an operator on all of $\X$. Moreover, there exist $\delta>0$ and $C>1$ such that there exists a unique $\phi\in \X_\delta$ with 
\begin{align*}
\K_\uf(\phi)=\phi
\end{align*}
 whenever $\|\uf\|_{\H}\leq\frac{\delta}{C}$.
\end{lem}
\subsection{Proof of Theorem \ref{stability}}
To prove Theorem \ref{stability} we still have to get rid of the correction term $\Cf$. We achieve this by picking the right blowup time $T$  close to $1$.
For this, we recall that the prescribed initial data 
$$
\Phi(0)=(\phi_1(0,.),\phi_2(0,.))
$$
are given by
\begin{align*}
	\phi_1(0,\rho)&=\psi_1(0,\rho)-\frac{2\arctan\left(\rho\right)}{\rho}= Tf(T\rho)-\frac{2\arctan\left(\rho\right)}{\rho},
	\\
	\phi_2(0,\rho)&=\psi_2(0,\rho)-\frac{2}{1+\rho^2}= T^2 g(T\rho)-\frac{2}{1+\rho^2}.
\end{align*}
Furthermore, $u^1_*[0]$ transformed to similarity coordinates is given by
\begin{align*}
\psi^1_{1_*}(0,\rho)=\frac{2\arctan\left(T \rho\right)}{\rho}, \qquad \psi^1_{2_*}(0,\rho)=\frac{ 2T^2}{1+T^2\rho^2}.
\end{align*}
This explicit dependence of $T$ of the initial data motivates the
definition of the operator $$\Uf:[1-\delta,1+\delta]\times (H^{\frac{3}{2}} \times H^{\frac{1}{2}})(\B^5_{1+\delta}) \to \mathcal{H}$$ by
\begin{align*}
\Uf(T,\vf)(\rho)=(Tv_1(T\rho),T^2 v_2(T\rho))+(\psi^1_{1_*}(0,\rho),\psi^1_{2_*}(0,\rho))
-\left(\frac{2\arctan(\rho)}{\rho},\frac{2}{1+\rho^2}
\right).
\end{align*}
Note that for $\delta \in (0,\frac{1}{2})$ and any $\vf$ fixed, this defines a continuous map $$U(.,\vf):[1-\delta,1+\delta]\to \mathcal{H}$$ (this follows as the first part of Lemma 8.2 in \cite{Glo22}). Also, the two identities
\begin{align*}
\Uf(1,\textbf{0})=\textbf{0}
\end{align*}
and
\begin{align*}
\Phi(0,\rho)=\Uf\left(T,\left(f(\rho)-\frac{2\arctan\left(\rho\right)}{\rho},g(\rho)-\frac{2}{1+\rho^2}\right)\right)
\end{align*}
hold.
Furthermore, by arguing as in the proof of Lemma 8.2 in \cite{Glo22}, one shows that the estimate
\begin{align*}
\|\Uf(T,\vf)\|_{\mathcal{H}} \lesssim \|\vf\|_{H^{\frac{3}{2}} \times H^{\frac{1}{2}}(\B^5_{1+\delta})}
+|1-T|
\end{align*}
is true for all $ T \in [1-\delta,1+\delta]$. 
\begin{lem}\label{lem:ex2}
There exist constants $M \geq 1$ and $\delta >0$ such that if $\vf \in H^{\frac{3}{2}} \times H^{\frac{1}{2}}(\B^5_{1+\delta})$ satisfies $\|\vf\|_{H^{\frac{3}{2}} \times H^{\frac{1}{2}}(\B^5_{1+\delta})}\leq \frac{\delta}{M}$, then there exists a unique $T^* \in [1-\delta,1+\delta]$ and a unique $\phi \in \X_\delta$ with 
$ \phi=\K_{\Uf(T^*,\vf)}(\phi)$ and $\Cf(\phi,\Uf(T^*,\vf))= 0$.
\end{lem}
\begin{proof}
Since 
$$
\partial_T \left(\frac{2\arctan(T \rho)}{\rho},\frac{2T^2 }{1+T^2\rho^2}\right)\bigg|_{T=1}=2\gf(\rho),
$$
the claim follows by an application of Brouwer's fixed point theorem, see
the proof of Lemma 6.5 in \cite{Don17} for the details.
\end{proof}
This allows us to give rigorous meaning to the notion of solutions in our topology.
\begin{defi}\label{def:solutionstrichartz}
  Let
  \[ \Gamma^T:=\{(t,r)\in [0,T)\times [0,\infty): r\leq T-t\}. \]
  We say that $u: \Gamma^T\to \R$ is a Strichartz solution of
\[ \left(\partial_t^2-\partial_r^2-\frac{4}{r}\partial_r\right)
  u(t,r) +\frac{ \sin(2ru(t,r))- 2 r u (t,r)}{
    r^3}=0 \]
if $\phi=\Phi_1:=[\Psi-\Psi_*]_1$, with
\[ \Psi(\tau,\rho):=
  \begin{pmatrix}
    \psi(\tau,\rho) \\ (1+\partial_\tau+\rho\partial_\rho)\psi(\tau,\rho)
  \end{pmatrix},\qquad \psi(\tau,\rho):=Te^{-\tau}u(T-Te^{-\tau}, Te^{-\tau}\rho)
\]
  belongs to $\mathcal{X}$ and satisfies
\begin{align*}
\phi=\K_{\Phi(0)}(\phi)
\end{align*}
and $\Cf(\phi,\Phi(0))=\textup{\textbf{0}}$.
\end{defi}
\begin{proof}[Proof of Theorem \ref{stability}]
Let $\delta >0$ be small enough, choose $M\geq0$ sufficiently large,
and, let $\vf=(f,g)-u^1_*[0] \in C^\infty \times C^\infty(\overline{\B^5_{1+\delta}})$  be such that
\begin{align*}
\|(f,g)-u^1_*[0]\|_{H^{\frac{3}{2}}\times H^{\frac{1}{2}}(\B^5_{1+\delta})}\leq \frac{\delta}{M}.
\end{align*}
Then, by Lemmas \ref{lem:ex1} and \ref{lem:ex2} there exists a
Strichartz solution $u$ with that initial data. Therefore, the
associated $\phi$ is the unique fixed point of $\K $ in $\X_\delta$
with vanishing correction term. Moreover, by standard partition
arguments one shows that this $\phi$ is in fact the unique fixed point
in all of $\X$, see for instance \cite{DonWal22}, Lemma 7.6. Furthermore, by classical Gronwall type arguments one shows that $u$ is in fact a smooth function on $\Gamma^T$, where $T$ denotes the blowup time. We calculate
\begin{align*}
\delta^2&\geq\| \phi\|_{L^2(\R_+)L^{10}(\B_1^5)}^2= \int_0^\infty \left\|\psi(\tau,.)-2|.|^{-1}\arctan\left(|.|\right)\right\|^2_{L^{10}(\B_1^5)}d \tau\\
&=\int_0^T \left\|\psi(-\log(T-t)+\log T,.)-2|.|^{-1}\arctan\left(|.|\right)\right\|^2_{L^{10}(\B_1^5)}\frac{dt}{T-t}\\
&=\int_0^T \bigg\|(T-t)^{-1}\psi(-\log(T-t)+\log T,\frac{.}{T-t})
\\
&\quad-2|.|^{-1}\arctan\left(\frac{|.|}{T-t}\right)\bigg\|^2_{L^{10}(\B_{T-t}^5)} dt
\\
&=\int_0^T \left\|u(t,.)-u^T_*(t,r)\right\|^2_{L^{10}(\B_{T-t}^5)}dt
\end{align*}
and similarly one computes
\begin{equation}
\delta^6\geq\int_0^T \left\|u(t,.)-u^T_*(t,.)\right \|_{\dot{W}^{1,\frac{30}{11}}(\mathbb
     B^5_{T-t})}^6dt.
\end{equation}
\end{proof}

\begin{proof}[Proof of Theorem \ref{maintheorem}]
Establishing Theorem \ref{maintheorem} reduces to two tasks. First one needs to prove that $u$ can be extended to all of \[ \Omega_{T}^5:=\left ([0, \infty)\times \mathbb
    R^5\right )\setminus
  \left \{(t,x)\in [T,\infty)\times \mathbb R^5: |x|\leq t-T\right
  \}. \]
This is a consequence of $N$ being a smooth bounded function away from $r=0$ and we refer the reader to section 2 and Lemma 8.3 of \cite{DonWal22} where this was done for one dimension higher. Secondly, 
one has to show that all estimates on $u$ ascend to estimates on \begin{align*}
U_u(t,.)= 
    \begin{pmatrix}
      \sin(|.|(u(t,.))\frac{.}{|.|} \\
      \cos(|.|u(t,.))
    \end{pmatrix}.
\end{align*}
This procedure was also carried out for $d=4$ in section 8 of \cite{DonWal22} and can be adapted in a straightforward way to the three dimensional case.
\end{proof}
\section{Appendix. Interpolation theory}
This appendix is concerned with our required interpolation result for
weighted Strich-artz spaces. The presentation given here is based on
the book ``Interpolation Spaces'' by J. Bergh and J. Löfström
\cite{BerLof12}. Following this reference, we let $(X_0,X_1)$ be a tuple of Banach spaces out of which we form the Banach space $(X_0+X_1,\|.\|_{X_0+X_1})$ where
\begin{align*}
\|x\|_{X_0+X_1}:=\inf_{x=x_0+x_1,x_j\in X_j, j=1,2}  ( \|x_0\|_{X_0}+\|x_1\|_{X_1})
\end{align*}
for $x\in X_0+ X_1$.
We now set $S:=\{ z\in \C:0\leq z\leq 1\}$ and 
consider the set $F(X_0,X_1)$ consisting of all continuous functions $f:S\to X_0+X_1$ that are analytic on the interior of $S$. Moreover, for $ f$ to be an element of $F(X_0,X_1)$, we require
the function $t\mapsto f(j+it)$, for $j=0,1$ to be a continuous function from $\R$ to $X_j$ which tends to $0$ as $|t| \to \infty.$
Then, $F(X_0,X_1)$ is a vector space and by equipping it with the norm
\begin{align*}
\|f\|_{F(X_0,X_1)}:=\max\left \{ \sup_{t\in\R} \| f(it)\|_{X_0},
  \sup_{t\in\R} \|f(1+it)\|_{X_1}\right \}
\end{align*}
it becomes a Banach space, see Lemma \cite[p.~88, Lemma 4.1.1.]{BerLof12}.
Next, for $\theta \in (0,1)$, we define the interpolation functor $C_\theta$ as follows. Let $(X_0,X_1)_{[\theta]}=C_\theta(X_0,X_1)$ be the set of all $ x\in X_0+X_1$ for which there exists an $f \in F(X_0,X_1)$ with $f(\theta)=x$. Furthermore, for any such $x$ we set
\begin{align*}
\|x\|_{(X_0,X_1)_{[\theta]}}:=\inf\{ \|f\|_{F(X_0,X_1)},f\in F(X_0,X_1): f(\theta)=x \}
\end{align*}
Then, $((X_0,X_1)_{[\theta]},\|.\|_{(X_0,X_1)_{[\theta]}})$ is Banach space and $C_\theta$ is an exact interpolation functor of order $\theta$ (see \cite[p.~88, Theorem 4.1.2.]{BerLof12}). 
Moreover, for a given Sobolev norm $\|.\|_{W^{s,q}(\B^5_1)}$, with $s\geq 0$ and $1\leq q\leq \infty$ as well as $a\in \R$ we let $L^p(\R_+, e^{a\tau}d\tau)W^{s,q}(\B^5_1)$ with $1\leq p<\infty$ be the completion of $C^\infty_c (\R_+\times\overline{\B^5_1})$ with respect to the norm
\begin{align*}
\|f\|_{L^p(\R_+,e^{a\tau}d\tau) W^{s,q}(\B^5_1)}^p:=\int_{\R_+}\|f(\tau,.)\|_{W^{s,q}(\B^5_1)}^p e^{a\tau} d\tau.
\end{align*}
Finally, we once more employ \cite[p.~317, Subsection 4.3.1.1, Theorem 1]{Tri95} to infer that 
for any $1\leq p,q_0,q_1\leq \infty$ and $0\leq s_0,s_1<\infty$
one has that
\begin{align*}
(W^{s_0,q_0}(\B^5_1),W^{s_1,q_1}(\B^5_1))_{[\frac{1}{2}]}=W^{s_{\frac{1}{2}},q_{\frac{1}{2}}}(\B^5_1)
\end{align*}
where $s_{\frac{1}{2}}=\tfrac 12(s_0+s_2)$ and $\tfrac{1}{q_{\frac{1}{2}}}=\tfrac{1}{2}(\tfrac{1}{q_0}+\tfrac{1}{q_1})$.
Having concluded these preliminaries, we come to the desired interpolation result.

\begin{prop}\label{prop:interpolation}
Let $1\leq q_0,q_1\leq \infty$, $0\leq s_0,s_1<\infty$, $1\leq p_0,p_1<\infty$, and $a\in \R$.   Then
\begin{align*}
\left(L^{p_0}(\R_+,e^{-a p_0\tau}d\tau) W^{s_0,q_0}(\B^5_1) ,L^{p_1}(\R_+,e^{a p_1\tau}d\tau)W^{s_1,q_1}(\B^5_1) \right)_{[\frac{1}{2}]}=L^{p_{\frac12}}(\R_+)W^{s_{\frac{1}{2}},q_{\frac{1}{2}}}(\B^5_1).
\end{align*}
\end{prop}
\begin{proof}
The proposition follows by slightly modifying the ideas of \cite[p.~107, Theorem 5.1.2]{BerLof12}), which we illustrate here for the convenience of the reader. To simplify notation, we set $W_0=W^{s_0,q_0}(\B^5_1)$, $W_1=W^{s_1,q_1}(\B^5_1)$ and $p=p_{\frac12}$. By construction, $C^\infty_c(\R_+\times \overline{\B^5_1})$ lies dense in $L^{p_0}(\R_+,e^{-a p_0\tau}d\tau)W_0\cap L^{p_1}(\R_+,e^{a p_1\tau}d\tau)W_1$ and so by \cite[p.~91, Theorem 4.2.2]{BerLof12} also in
 $$\left(L^{p_0}(\R_+,e^{-a p_0\tau}d\tau) W_0,L^{p_1}(\R_+,e^{a p_1\tau}d\tau)W_1\right)_{[\frac{1}{2}]}\quad \text{ and } \quad L^p(\R_+)(W_0,W_1)_{[\frac{1}{2}]}.$$ Consequently, it suffices to consider $C^\infty_c(\R_+\times \overline{\B^5_1}).$ We start with the inequality
\begin{align*}
\|u\|_{(L^{p_0}(\R_+,e^{-a p_0\tau}d\tau)W_0,L^{p_1}(\R_+,e^{a p_1\tau}d\tau)W_1 )_{[\frac{1}{2}]}}\leq \|u\|_{L^p(\R_+)(W_0,W_1)_{[\frac{1}{2}]}}.
\end{align*}
Let $u\in C^\infty_c(\R_+\times \overline{\B^5_1})$ with $u \neq 0$. Then, for every
$\varepsilon>0$ and every $\tau\geq 0$, there exists an $f(\tau)\in F(W_0,W_1)$ with $f(\tau)(\frac{1}{2})=u(\tau,.)$ and 
$$
\|f(\tau)\|_{F(W_0,W_1)}\leq (1+\varepsilon)\|u(\tau,.)\|_{(W_0,W_1)_{[\frac{1}{2}]}}.
$$
Set
$$
g(\tau)(z)=f(\tau)(z)e^{ 2a(\frac{1}{2}-z)\tau}\left(\frac{\|u(\tau)\|_{(W_0,W_1)_{[\frac{1}{2}]}}}{\|u\|_{L^{p}(\R_+)(W_0,W_1)_{[\frac{1}{2}]}}}\right)^{p(\frac{1}{p_0}-\frac1{p_1})(\frac{1}{2}-z)}.
$$
Then, clearly $g(\tau)(\frac{1}{2})=u(\tau,.)$ and since
$$
p_0+p_0 \frac p2(\frac{1}{p_0}-\frac 1{p_1})=p
$$ one readily computes that
\begin{align*}
\|g(\tau)(it)\|_{L^{p_0}(\R_+,e^{-a p_0\tau}d\tau)W_0}^{p_0} &=\int_{\R_+}\|f(\tau)(it)\|_{W_0}^{p_0}\left(\frac{\|u(\tau)\|_{(W_0,W_1)_{[\frac{1}{2}]}}}{\|u\|_{L^{p}(\R_+)(W_0,W_1)_{[\frac{1}{2}]}}}\right)^{p_0\frac p2(\frac{1}{p_0}-\frac{1}{p_1})} d\tau
\\
&\leq (1+\varepsilon)^{p_0}\|u\|_{L^{p}(\R_+)(W_0,W_1)_{[\frac{1}{2}]}}^{-p_0\frac p2(\frac{1}{p_0}-\frac{1}{p_1})}\int_{\R_+}\|u(\tau,.)\|_{(W_0,W_1)_{[\frac{1}{2}]}}^{p} d\tau
\\
&=(1+\varepsilon)^{p_0}\|u\|_{L^p(\R_+)(W_0,W_1)_{[\frac{1}{2}]}}^{p_0} 
\end{align*}
and similarly
\begin{align*}
\|g(\tau)(1+it)\|_{L^{p_1}(\R_+,e^{a p_1\tau}d\tau)W_1 }^{p_1} &\leq (1+\varepsilon)^p\|u\|_{L^p(\R_+)(W_0,W_1)_{[\frac{1}{2}]}}^{p_1}.
\end{align*}
Hence, as $\varepsilon>0$ was chosen arbitrarily, the claim follows.

For the other inequality, we invoke \cite[p.~93, Lemma 4.3.2]{BerLof12} which states that any $f\in F(W_0,W_1)$ satisfies
\begin{align}\label{Eq:Poisson kernel}
\|f\|_{(W_0,W_1)_{[\theta]}}\leq \left(\frac{1}{1-\theta}\int_\R \|f(it)\|_{W_0} P_0(\theta,t) dt\right)^{1-\theta}\left(\frac{1}{\theta}\int_\R \|f(1+it)\|_{W_1} P_1(\theta,t) dt\right)^{\theta}
\end{align}
where 
$$
P_j(x+iy,t):=\frac{e^{-\pi(t-y)}\sin (\pi x)}{\sin(\pi x)^2+(\cos(\pi x)-e^{ij\pi-\pi (t-y)})^2}
$$
are the Poisson kernels of the strip $S$. 
Further, for $u\in C^\infty_c(\R_+\times \overline{\B^5_1})$ let $f(\tau)\in F(W_0,W_1)$ be such that $f(\tau)(\frac{1}{2})=u(\tau,.)$. Then, \eqref{Eq:Poisson kernel}, Hölder's inequality,  and the identity
$\frac 1p =\frac{1}{2p_0}+\frac{1}{2p_1}$ imply that
\begin{align*}
\|u\|_{L^p(\R_+)(W_0,W_1)_{[\frac{1}{2}]}}&\leq 4\bigg\| \int_{\R_+} \bigg[\left(\int_\R \|f(\tau)(it)\|_{W_0} P_0\left(\tfrac12,t\right) dt\right)^{\frac12}
\\
&\quad \times \left(\int_\R \|f(\tau)(1+it)\|_{W_1} P_1\left(\tfrac12,t\right) dt\right)^{\frac12}\bigg\|_{L^p_\tau(\R_+)}
\\
&\leq 4\left\|e^{-a\tau}\int_\R \|f(\tau)(it)\|_{W_0} P_0\left(\tfrac12,t\right) dt \right\|_{L^{p_0}_\tau(\R_+)}^{\frac{1}{2}}
\\
&\quad\times \left\| e^{a\tau}\int_\R\|f(\tau)(1+it)\|_{W_1} P_1\left(\tfrac12,t\right) dt\right\|_{L^{p_1}_\tau(\R_+)}^{\frac{1}{2}}.
\end{align*}
Next, by Minkowski's inequality
\begin{align*}
\left\|\int_\R \|f(\tau)(it)\|_{W_0} P_0\left(\tfrac12,t\right) dt e^{-a\tau}\right\|_{L^{p_0}_\tau(\R_+)}&\leq \int_\R \left\|\|f(\tau)(it)\|_{W_0} e^{-a\tau}\right\|_{L^{p_0}_\tau(\R_+)} P_0\left(\tfrac12,t\right) dt
\\
&\leq\sup_{t\in \R}\|f(\tau)(it)\|_{L^{p_0}(\R_+,e^{-a p_0\tau} d\tau) W_0}
\\
&\quad\times \int_\R P_0\left(\tfrac12,t\right) dt
\end{align*}
and analogously one estimates the second factor.
Observe now that for $j=0,1$
\begin{align*}
\int_\R P_j \left(\tfrac12,t\right) dt=\int_\R \frac{e^{-\pi t}}{1+e^{-2\pi t}}dt=\frac{1}{2}.
\end{align*}
Therefore, 
\begin{align*}
\|u\|_{L^p(\R_+)(W_0,W_1)_{[\frac{1}{2}]}}^p &\leq \sup_{t\in \R}\|f(\tau)(it)\|_{L^{p_0}(\R_+,e^{-a p_0\tau} d\tau) W_0}^{\frac{1}{2}}\sup_{t\in \R}\|f(\tau)(1+it)\|_{L^{p_1}(\R_+,e^{a p_1\tau} d\tau) W_1}^{\frac{1}{2}}
\\
&\leq \|f\|_{F(L^p(\R_+,e^{-a p_0\tau}d\tau)W_0, L^p(\R_+,e^{a p_1\tau}d\tau)W_1)}.
\end{align*}
\end{proof}

\bibliography{references}
\bibliographystyle{plain}

\end{document}